\documentclass{amsart}%
\usepackage{xcolor}
\usepackage{amssymb}
\usepackage{amsfonts}
\usepackage{geometry}
\usepackage{amsmath}
\usepackage{graphicx}%
\providecommand{\U}[1]{\protect\rule{.1in}{.1in}}
\newtheorem{theorem}{Theorem}
\theoremstyle{plain}

\newtheorem{corollary}[theorem]{Corollary}

\newtheorem{definition}[theorem]{Definition}

\newtheorem{lemma}[theorem]{Lemma}
\newtheorem{notation}[theorem]{Notation}
\newtheorem{problem}[theorem]{Problem}
\newtheorem{proposition}[theorem]{Proposition}
\newtheorem{remark}[theorem]{Remark}

\numberwithin{equation}{section}
\begin{document}
\title[A weak to strong type $T1$ theorem]{A weak to strong type $T1$ theorem for general smooth Calder\'{o}n-Zygmund
operators with doubling weights, II}
\author[M. Alexis]{Michel Alexis}
\address{Department of Mathematics \& Statistics, McMaster University, 1280 Main Street
West, Hamilton, Ontario, Canada L8S 4K1}
\email{alexism@mcmaster.ca}
\author[E.T. Sawyer]{Eric T. Sawyer}
\address{Department of Mathematics \& Statistics, McMaster University, 1280 Main Street
West, Hamilton, Ontario, Canada L8S 4K1 }
\email{Sawyer@mcmaster.ca}
\author[I. Uriarte-Tuero]{Ignacio Uriarte-Tuero}
\address{Department of Mathematics, University of Toronto\\
Room 6290, 40 St. George Street, Toronto, Ontario, Canada M5S 2E4\\
(Adjunct appointment)\\
Department of Mathematics\\
619 Red Cedar Rd., room C212\\
Michigan State University\\
East Lansing, MI 48824 USA}
\email{ignacio.uriartetuero@utoronto.ca}
\thanks{E. Sawyer is partially supported by a grant from the National Research Council
of Canada}
\thanks{I. Uriarte-Tuero has been partially supported by grant MTM2015-65792-P
(MINECO, Spain), and is partially supported by a grant from the National
Research Council of Canada}

\begin{abstract}
We consider the weak to strong type problem for two weight norm inequalities
for Calder\'{o}n-Zygmund operators with \emph{doubling} weights. We show that
if a Calder\'{o}n-Zygmund operator $T$ is weak type $\left(  2,2\right)  $
with doubling weights, then it is strong type $\left(  2,2\right)  $ if and
only if the dual cube testing condition for $T^{\ast}$ holds, alternatively if
and only if the dual cancellation condition of Stein holds. This continues the
weighted theory begun in \cite{Saw6}.

The testing condition can be taken with respect to either cubes or balls, and
more generally, this is extended to a weak form of $Tb$ theorem.

Finally, we show that for all pairs of locally finite positive Borel measures,
and all Stein elliptic Calder\'{o}n-Zygmund operators $T$, the weak type
$\left(  2,2\right)  $ inequalities for $T$ and and its associated maximal
truncations operator $T_{\flat}$ are equivalent. Thus the characterization of
weak type for $T_{\flat}$ in \cite{LaSaUr1} applies to $T$ as well.

\end{abstract}
\maketitle
\tableofcontents

\section{Introduction}

\textit{This paper is a sequel to the second author's paper \cite{Saw6}, and
continues to be dedicated to the memory of Professor Elias M. Stein.}

In this paper we attack the problem of characterizing the two weight norm
inequality for Calder\'{o}n-Zygmund operators in two separate steps. First, we
characterize the weak type problem for general weights in terms of a testing
type condition. Second, we solve the \emph{weak} to \emph{strong} type problem
for doubling weights in terms of the standard cube testing condition.

\begin{problem}
Given a pair of locally finite positive Borel measures on $\mathbb{R}^{n}$,
and an $\alpha$-fractional Calder\'{o}n-Zygmund operator $T^{\alpha}$ that is
weak type $\left(  2,2\right)  $, i.e.%
\begin{equation}
\left\Vert T_{\sigma}^{\alpha}f\right\Vert _{L^{2,\infty}\left(
\omega\right)  }\leq\operatorname*{weak}\mathfrak{N}_{T^{\alpha}}\left(
\sigma,\omega\right)  \left\Vert f\right\Vert _{L^{2}\left(  \sigma\right)  },
\label{weak}%
\end{equation}
characterize when $T_{\sigma}^{\alpha}$ is strong type $\left(  2,2\right)  $,
i.e.
\begin{equation}
\left\Vert T_{\sigma}^{\alpha}f\right\Vert _{L^{2}\left(  \omega\right)  }%
\leq\mathfrak{N}_{T^{\alpha}}\left(  \sigma,\omega\right)  \left\Vert
f\right\Vert _{L^{2}\left(  \sigma\right)  }. \label{strong}%
\end{equation}

\end{problem}

One rationale for considering the two step approach used here is the notorious
difficulty of the general $T1$ conjecture\footnote{see e.g. \cite{Hyt2},
\cite{Lac}, \cite{LaSaShUr3}, \cite{LaWi}, \cite{NTV4}, \cite{SaShUr7},
\cite{SaShUr12}, \cite{Vol} and many more of the references.}, namely that an
$\alpha$-fractional Calder\'{o}n-Zygmund operator $T^{\alpha}$ is strong type
if and only if both the testing condition and its dual formulation hold, as
well as an extension of the classical $A_{2}$ condition of Muckenhoupt. In the
two step approach used here, we are concerned with two logically easier
problems, that of characterizing weak type boundedness, followed by
characterizing the passage from weak type to strong type. It is notable that
there is no known counterexample to the general $T1$ conjecture, yet the
conjecture has only been shown to hold for the Hilbert transform (see the two
part paper \cite{LaSaShUr3}, \cite{Lac}) and certain perturbations (see
\cite{SaShUr10}).

A corollary of our main theorem below is that when $T^{\alpha}$ is a Stein
elliptic Calder\'{o}n-Zygmund operator on $\mathbb{R}^{n}$, and when $\sigma$
and $\omega$ are both doubling measures, then given that $T^{\alpha}$ is weak
type $\left(  2,2\right)  $, it is strong type $\left(  2,2\right)  $ if and
only if the cube testing constant $\mathfrak{T}_{T^{\alpha}}\left(
\sigma,\omega\right)  $ for $T^{\alpha}$ is finite, where%
\begin{equation}
\mathfrak{T}_{T^{\alpha}}\left(  \sigma,\omega\right)  \equiv\sup
_{Q\in\mathcal{P}^{n}}\sqrt{\frac{1}{\left\vert Q\right\vert _{\sigma}}%
\int_{Q}\left\vert T_{\sigma}^{\alpha}\mathbf{1}_{Q}\right\vert ^{2}\omega}.
\label{Muck and test}%
\end{equation}
Under the above hypotheses of doubling measures, we can also replace the cube
testing constant $\mathfrak{T}_{T^{\alpha}}\left(  \sigma,\omega\right)  $
with the \emph{cancellation} constant of Stein $\mathfrak{A}_{K^{\alpha}%
}\left(  \sigma,\omega\right)  $, which is defined as the best constant
$C\left(  \sigma,\omega\right)  $ in the inequality,%
\begin{align}
&  \int_{\left\vert x-x_{0}\right\vert <N}\left\vert \int_{\varepsilon
<\left\vert x-y\right\vert <N}K^{\alpha}\left(  x,y\right)  d\sigma\left(
y\right)  \right\vert ^{2}d\omega\left(  x\right)  \leq C\left(  \sigma
,\omega\right)  \ \int_{\left\vert x_{0}-y\right\vert <N}d\sigma\left(
y\right)  ,\label{can cond}\\
&  \ \ \ \ \ \ \ \ \ \ \ \ \ \ \ \ \ \ \ \ \ \ \ \ \ \text{for all
}0<\varepsilon<N\text{ and }x_{0}\in\mathbb{R}^{n}.\nonumber
\end{align}
The cancellation constant $\mathfrak{A}_{K^{\alpha}}\left(  \sigma
,\omega\right)  $ has the historical form of bounding, in an average sense,
integrals of the kernel over annuli, and arose originally in the setting of
Lebesgue measure, in the form of the $T1$ theorem presented by E. Stein in
\cite[Theorem 4 page 306]{Ste2}). Similar suprema and inequalities define the
dual constants $\mathfrak{T}_{T^{\alpha,\ast}}$ and $\mathfrak{A}%
_{K^{\alpha,\ast}}$, in which the measures $\sigma$ and $\omega$ are
interchanged and $K^{\alpha}\left(  x,y\right)  $ is replaced by
$K^{\alpha,\ast}\left(  x,y\right)  =K^{\alpha}\left(  y,x\right)  $. It
should also be noted that (\ref{can cond}) is \emph{not} simply the testing
condition for a truncation of $T$ over a ball.

Define the classical Muckenhoupt constant $A_{2}^{\alpha}\left(  \sigma
,\omega\right)  $ by%
\begin{equation}
A_{2}^{\alpha}\left(  \sigma,\omega\right)  \equiv\sup_{Q\in\mathcal{P}^{n}%
}\frac{\left\vert Q\right\vert _{\sigma}}{\left\vert Q\right\vert
^{1-\frac{\alpha}{n}}}\frac{\left\vert Q\right\vert _{\omega}}{\left\vert
Q\right\vert ^{1-\frac{\alpha}{n}}}, \label{A2}%
\end{equation}
and the indicator / cube testing constant $\mathfrak{T}_{T^{\alpha}%
}^{\operatorname*{ind}}\left(  \sigma,\omega\right)  $ by%
\begin{equation}
\mathfrak{T}_{T^{\alpha}}^{\operatorname*{ind}}\left(  \sigma,\omega\right)
\equiv\sup_{Q\in\mathcal{P}^{n}}\frac{1}{\sqrt{\left\vert Q\right\vert
_{\sigma}}}\sup_{E\subset Q}\left\vert \int_{Q}\left(  T_{\sigma}^{\alpha
}\mathbf{1}_{E}\right)  ^{2}\omega\right\vert . \label{ind}%
\end{equation}

Our main theorem is this, which improves upon the corresponding theorem in
\cite{Saw6} by eliminating one of the indicator / cube testing consitions, as
well as a comparability hypothesis on the pair of measures $\left(
\sigma,\omega\right)  $. As in \cite{Saw6}, one can also replace the testing
constants\ with the cancellation constants of Stein.

\begin{theorem}
\label{main}Suppose $\sigma$ and $\omega$ are doubling measures on
$\mathbb{R}^{n}$, and that $T^{\alpha}$ is a smooth Stein elliptic
Calder\'{o}n-Zygmund operator on $\mathbb{R}^{n}$. Then (\ref{strong}) holds
if and only if the testing constants in (\ref{Muck and test}) and (\ref{ind})
are finite, and furthermore, we can replace (\ref{Muck and test}) with
(\ref{can cond}) here. Moreover, we have the equivalences,%
\begin{align}
\mathfrak{N}_{T^{\alpha}}\left(  \sigma,\omega\right)   &  \approx
\mathfrak{T}_{T^{\alpha}}\left(  \sigma,\omega\right)  +\mathfrak{T}%
_{T^{\alpha,\ast}}^{\operatorname*{ind}}\left(  \omega,\sigma\right)
\label{in line}\\
&  \approx\mathfrak{A}_{K^{\alpha}}\left(  \sigma,\omega\right)
+\mathfrak{T}_{T^{\alpha,\ast}}^{\operatorname*{ind}}\left(  \omega
,\sigma\right) \nonumber\\
&  \approx\mathfrak{T}_{T^{\alpha}}\left(  \sigma,\omega\right)
+\mathfrak{N}_{T^{\alpha}}^{\operatorname*{weak}}\left(  \sigma,\omega\right)
\nonumber\\
&  \approx\mathfrak{A}_{K^{\alpha}}\left(  \sigma,\omega\right)
+\mathfrak{N}_{T^{\alpha}}^{\operatorname*{weak}}\left(  \sigma,\omega\right)
,\nonumber
\end{align}
and the corresponding equivalences with $T^{\alpha}$ and $T^{\alpha,\ast}$ and
their constants interchanged. Here $\mathfrak{N}_{T^{\alpha}}%
^{\operatorname*{weak}}\left(  \sigma,\omega\right)  $ denotes the weak type
norm of $T^{\alpha}$.
\end{theorem}

\begin{remark}
It is easy to see that the second line in (\ref{in line}) follows from the
first line in (\ref{in line}) using
\[
\mathfrak{N}_{T^{\alpha}}\left(  \sigma,\omega\right)  \geq
\operatorname*{weak}\mathfrak{N}_{T^{\alpha}}\left(  \sigma,\omega\right)
\geq\mathfrak{T}_{T^{\alpha,\ast}}^{\operatorname*{ind}}\left(  \omega
,\sigma\right)  .
\]
To see this last display, recall that the dual space of the Banach space
$L^{2,1}\left(  \mu\right)  $ is $L^{2,\infty}\left(  \mu\right)  $ for any
$\sigma$-finite nonatomic measure, see e.g. \cite[Theorem 1.4.17 (v) page
52]{Gra}. Thus%
\begin{align*}
&  \mathfrak{N}_{T^{\alpha,\ast}}^{\operatorname*{rest}}\left(  \omega
,\sigma\right)  \equiv\sup_{\left\Vert g\right\Vert _{L^{2,1}\left(
\omega\right)  }\leq1}\left\Vert T_{\omega}^{\alpha,\ast}g\right\Vert
_{L^{2}\left(  \sigma\right)  }=\sup_{\left\Vert g\right\Vert _{L^{2,1}\left(
\omega\right)  }\leq1}\sup_{\left\Vert f\right\Vert _{L^{2}\left(
\sigma\right)  }\leq1}\left\vert \int\left(  T_{\omega}^{\alpha,\ast}g\right)
fd\sigma\right\vert \\
&  =\sup_{\left\Vert f\right\Vert _{L^{2}\left(  \sigma\right)  }\leq1}%
\sup_{\left\Vert g\right\Vert _{L^{2,1}\left(  \omega\right)  }\leq
1}\left\vert \int g\left(  T_{\sigma}^{\alpha}f\right)  d\sigma\right\vert
=\sup_{\left\Vert f\right\Vert _{L^{2}\left(  \sigma\right)  }\leq1}\left\Vert
T_{\sigma}^{\alpha}f\right\Vert _{L^{2,1}\left(  \omega\right)  ^{\ast}}%
=\sup_{\left\Vert f\right\Vert _{L^{2}\left(  \sigma\right)  }\leq1}\left\Vert
T_{\sigma}^{\alpha}f\right\Vert _{L^{2,\infty}\left(  \omega\right)
}=\operatorname*{weak}\mathfrak{N}_{T^{\alpha}}\left(  \sigma,\omega\right)  ,
\end{align*}
\ and then by \cite[Theorem 3.13 page 195]{StWe} we have
\begin{align*}
\mathfrak{N}_{T^{\alpha}}\left(  \sigma,\omega\right)   &  \geq
\operatorname*{weak}\mathfrak{N}_{T^{\alpha}}\left(  \sigma,\omega\right)
=\mathfrak{N}_{T^{\alpha,\ast}}^{\operatorname*{rest}}\left(  \omega
,\sigma\right)  \approx\sup_{E}\sqrt{\frac{\int\left\vert T_{\omega}%
^{\alpha,\ast}\mathbf{1}_{E}\right\vert ^{2}d\sigma}{\left\vert E\right\vert
_{\omega}}}\\
&  \geq\sup_{Q\in\mathcal{D}}\sup_{E\subset Q}\sqrt{\frac{\int_{Q}\left\vert
T_{\omega}^{\alpha,\ast}\mathbf{1}_{E}\right\vert ^{2}d\sigma}{\left\vert
Q\right\vert _{\omega}}}=\mathfrak{T}_{T^{\alpha,\ast}}^{\operatorname*{ind}%
}\left(  \omega,\sigma\right)  .
\end{align*}

\end{remark}

We thus have the following solution to a case of the weak to strong type problem.

\begin{corollary}
Suppose $\sigma$ and $\omega$ are doubling measures, and $T^{\alpha}$ is a
Stein elliptic Calder\'{o}n-Zygmund operator on $\mathbb{R}^{n}$ of weak type
$\left(  2,2\right)  $ with respect to $\left(  \sigma,\omega\right)  $. Then%
\begin{align*}
&  T^{\alpha}\text{ is strong type }\left(  2,2\right) \\
&  \Longleftrightarrow\mathfrak{T}_{T^{\alpha}}\left(  \sigma,\omega\right)
\text{ is finite}\\
&  \Longleftrightarrow\mathfrak{A}_{K^{\alpha}}\left(  \sigma,\omega\right)
\text{ is finite},
\end{align*}
and moreover,%
\[
\mathfrak{N}_{T^{\alpha}}\left(  \sigma,\omega\right)  \approx\mathfrak{N}%
_{T^{\alpha}}^{\operatorname*{weak}}\left(  \sigma,\omega\right)
+\mathfrak{T}_{T^{\alpha}}\left(  \sigma,\omega\right)  \approx\mathfrak{N}%
_{T^{\alpha}}^{\operatorname*{weak}}\left(  \sigma,\omega\right)
+\mathfrak{A}_{K^{\alpha}}\left(  \sigma,\omega\right)  .
\]

\end{corollary}

\subsection{Weak type inequalities}

The above corollary begs the question of characterizing the weak type norm
$\mathfrak{N}_{T^{\alpha}}^{\operatorname*{weak}}\left(  \sigma,\omega\right)
$ in terms of a testing condition, which is the first step in our two step
program. In order to shed light on this question, we recall the maximal
truncation operator $T_{\flat}^{\alpha}$ associated to $T^{\alpha}$ defined by
$T_{\flat,\sigma}f\left(  x\right)  \equiv\sup_{\varepsilon>0}\left\vert
T_{\varepsilon,\sigma}f\left(  x\right)  \right\vert $. In the final section
of the paper we show that for a Stein elliptic operator $T$, the weak type
inequalities for $T$ and $T_{\flat}$ are equivalent.

\begin{theorem}
\label{weak type}Suppose $T^{\alpha}$ is an $\alpha$-fractional
Calder\'{o}n-Zygmund operator in $\mathbb{R}^{n}$ with kernel satisfying just
(\ref{basic}) below, and $\sigma$ and $\omega$ are locally finite positive
Borel measures in $\mathbb{R}^{n}$. Then we have%
\begin{equation}
\mathfrak{N}_{T^{\alpha}}^{\operatorname*{weak}}\left(  \sigma,\omega\right)
\lesssim\mathfrak{N}_{T_{\flat}^{\alpha}}^{\operatorname*{weak}}\left(
\sigma,\omega\right)  \lesssim\mathfrak{N}_{T^{\alpha}}^{\operatorname*{weak}%
}\left(  \sigma,\omega\right)  +\sqrt{A_{2}^{\alpha}\left(  \sigma
,\omega\right)  },\label{weak flat}%
\end{equation}
where $\mathfrak{N}_{T^{\alpha}}^{\operatorname*{weak}}\left(  \sigma
,\omega\right)  $ and $\mathfrak{N}_{T_{\flat}^{\alpha}}^{\operatorname*{weak}%
}\left(  \sigma,\omega\right)  $ are the weak type $\left(  2,2\right)  $
norms of $T^{\alpha}$ and $T_{\flat}^{\alpha}$ respectively. If in addition
$T^{\alpha}$ is Stein ellliptic, then the Muckenhoupt constant $\sqrt
{A_{2}^{\alpha}\left(  \sigma,\omega\right)  }$ can be dropped from the right
hand side of (\ref{weak flat}).
\end{theorem}

Then from the characterization of weak type for maximal truncations
in\ \cite[Theorem 1.8 (1)]{LaSaUr1}, we obtain the corollary that%
\begin{equation}
\mathfrak{N}_{T^{\alpha}}^{\operatorname*{weak}}\left(  \sigma,\omega\right)
\approx\mathfrak{N}_{T_{\flat}^{\alpha}}^{\operatorname*{weak}}\left(
\sigma,\omega\right)  \approx\mathfrak{T}_{T_{\flat}^{\alpha}}%
^{\operatorname*{flat}}\left(  \sigma,\omega\right)  , \label{weak type char}%
\end{equation}
where the \emph{flat testing} constant $\mathfrak{T}_{T_{\flat}^{\alpha}%
}^{\operatorname*{flat}}\left(  \sigma,\omega\right)  $ (introduced in
\cite{LaSaUr1}) is the best constant in the inequality,%
\begin{equation}
\left\vert \int_{Q}T_{\flat,\sigma}^{\alpha}\left(  \mathbf{1}_{Q}f\right)
\left(  x\right)  d\omega\left(  x\right)  \right\vert \leq\mathfrak{T}%
_{T_{\flat}^{\alpha}}^{\operatorname*{flat}}\left(  \sigma,\omega\right)
\left\Vert f\right\Vert _{L^{2}\left(  \sigma\right)  }\left\Vert
\mathbf{1}_{Q}\right\Vert _{L^{2}\left(  \omega\right)  }\ . \label{two test}%
\end{equation}
Note that if $T_{\flat}^{\alpha}$ were \emph{linear}, then the inequality in
(\ref{two test}) would be equivalent to the dual cube testing condition,%
\[
\int_{Q}T_{\flat,\omega}^{\alpha,\ast}\left(  \mathbf{1}_{Q}\right)  \left(
x\right)  ^{2}d\sigma\left(  x\right)  \lesssim\left\vert Q\right\vert
_{\omega}\ .
\]
In general, we can only apply duality to linearizations $L$ of $T_{\flat
}^{\alpha}$ (see \cite{LaSaUr1} for definitions), and (\ref{two test}) is then
equivalent to%
\[
\int_{Q}L_{\omega}\left(  \mathbf{1}_{Q}\right)  \left(  x\right)  ^{2}%
d\sigma\left(  x\right)  \leq C\mathfrak{T}_{T_{\flat}^{\alpha}}%
^{\operatorname*{flat}}\left(  \sigma,\omega\right)  \left\vert Q\right\vert
_{\omega}\ ,
\]
taken uniformly over all linearizations $L$ of $T_{\flat}^{\alpha}$.

Note also that in the weak type characterization (\ref{weak type char}), the
conditions required of the kernel of $T$ are very weak, namely just
(\ref{basic}) below, which consists of the size condition together with a Dini
smoothness condition in the first variable of the kernel. In conclusion we
obtain the following theorem.

\begin{theorem}
\label{flat theorem}Suppose $\sigma$ and $\omega$ are doubling measures on
$\mathbb{R}^{n}$, and that $T^{\alpha}$ is a smooth Stein elliptic
Calder\'{o}n-Zygmund operator on $\mathbb{R}^{n}$. Then,%
\[
\mathfrak{N}_{T^{\alpha}}\left(  \sigma,\omega\right)  \approx\mathfrak{T}%
_{T^{\alpha}}\left(  \sigma,\omega\right)  +\mathfrak{T}_{T_{\flat}%
}^{\operatorname*{flat}}\left(  \sigma,\omega\right)  .
\]

\end{theorem}

\begin{remark}
We showed above that%
\[
\mathfrak{T}_{T^{\ast}}^{\operatorname*{ind}}\left(  \omega,\sigma\right)
\lesssim\mathfrak{N}_{T}^{\operatorname*{weak}}\left(  \sigma,\omega\right)
,
\]
and so we have
\begin{equation}
\mathfrak{T}_{T^{\ast}}^{\operatorname*{ind}}\left(  \omega,\sigma\right)
\lesssim\mathfrak{T}_{T_{\flat}}^{\operatorname*{flat}}\left(  \sigma
,\omega\right)  , \label{ind flat}%
\end{equation}
Thus Theorem \ref{flat theorem} follows from (\ref{ind flat}) and Theorem
\ref{main}. However, we do not know how to obtain (\ref{ind flat}) without
going through the weak type norm $\mathfrak{N}_{T}^{\operatorname*{weak}%
}\left(  \sigma,\omega\right)  $ and using Theorem \ref{weak type} and
\cite[Theorem 1.8 (1)]{LaSaUr1}.
\end{remark}

\begin{remark}
\label{main'}The arguments below will show that for doubling measures $\sigma$
and $\omega$, and any smooth fractional Calder\'{o}n-Zygmund operator
$T^{\alpha}$ on $\mathbb{R}^{n}$, not necessarily Stein elliptic, we have the
inequalities,%
\begin{align*}
\mathfrak{N}_{T^{\alpha}}\left(  \sigma,\omega\right)   &  \lesssim
\mathfrak{T}_{T^{\alpha}}\left(  \sigma,\omega\right)  +\mathfrak{T}%
_{T^{\alpha,\ast}}^{\operatorname*{ind}}\left(  \omega,\sigma\right)
+\sqrt{A_{2}^{\alpha}\left(  \sigma,\omega\right)  },\\
\mathfrak{N}_{T^{\alpha}}^{\operatorname*{weak}}\left(  \sigma,\omega\right)
&  \lesssim\mathfrak{T}_{T_{\flat}}^{\operatorname*{flat}}\left(
\sigma,\omega\right)  +\sqrt{A_{2}^{\alpha}\left(  \sigma,\omega\right)  },
\end{align*}
where (\ref{sizeandsmoothness'}) is assumed in the first line, and
(\ref{basic}) in the second line.
\end{remark}

\section{Preliminaries}

Denote by $\mathcal{P}^{n}$ the collection of cubes in $\mathbb{R}^{n}$ having
sides parallel to the coordinate axes; all cubes mentioned in this paper will
be elements of $\mathcal{P}^{n}$. A positive locally finite Borel measure
$\mu$ on $\mathbb{R}^{n}$ is said to satisfy the\emph{\ doubling condition} if
$\left\vert 2Q\right\vert _{\mu}\leq C_{\operatorname*{doub}}\left\vert
Q\right\vert _{\mu}$ for all cubes $Q\in\mathcal{P}^{n}$. It is well known
(see e.g. the introduction in \cite{SaUr}) that doubling implies reverse
doubling, which means that there exists a positive constant $\theta_{\mu
}^{\operatorname*{rev}}$, called a \emph{reverse doubling exponent}, such that%
\[
\sup_{Q\in\mathcal{P}^{n}}\frac{\left\vert sQ\right\vert _{\mu}}{\left\vert
Q\right\vert _{\mu}}\leq s^{\theta_{\mu}^{\operatorname*{rev}}}%
,\ \ \ \ \ \text{for all sufficiently small }s>0.
\]

\subsection{Standard fractional singular integrals and the norm inequality}

Let $0\leq\alpha<n$ and $\kappa\in\mathbb{N}$. We define a standard $\left(
\kappa+\delta\right)  $-smooth $\alpha$-fractional Calder\'{o}n-Zygmund kernel
$K^{\alpha}(x,y)$ to be a function $K^{\alpha}:\mathbb{R}^{n}\times
\mathbb{R}^{n}\rightarrow\mathbb{R}$ satisfying the following fractional size
and smoothness conditions: for $x\neq y$, and with $\nabla_{1}$ and
$\nabla_{2}$ denoting gradient in the first and second variables
respectively,
\begin{align}
&  \left\vert \nabla_{1}^{j}K^{\alpha}\left(  x,y\right)  \right\vert \leq
C_{CZ}\left\vert x-y\right\vert ^{\alpha-j-n},\ \ \ \ \ 0\leq j\leq
\kappa,\label{sizeandsmoothness'}\\
&  \left\vert \nabla_{1}^{\kappa}K^{\alpha}\left(  x,y\right)  -\nabla
_{1}^{\kappa}K^{\alpha}\left(  x^{\prime},y\right)  \right\vert \leq
C_{CZ}\left(  \frac{\left\vert x-x^{\prime}\right\vert }{\left\vert
x-y\right\vert }\right)  ^{\delta}\left\vert x-y\right\vert ^{\alpha-\kappa
-n},\ \ \ \ \ \frac{\left\vert x-x^{\prime}\right\vert }{\left\vert
x-y\right\vert }\leq\frac{1}{2},\nonumber
\end{align}
and where the same inequalities hold for the adjoint kernel $K^{\alpha,\ast
}\left(  x,y\right)  \equiv K^{\alpha}\left(  y,x\right)  $, in which $x$ and
$y$ are interchanged, and where $\nabla_{1}$ is replaced by $\nabla_{2}$.

\subsubsection{Ellipticity of kernels}

Following \cite[(39) on page 210]{Ste}, we say that an $\alpha$-fractional
Calder\'{o}n-Zygmund kernel $K^{\alpha}$ is \emph{elliptic in the sense of
Stein} if there is a unit vector $\mathbf{u}_{0}\in\mathbb{R}^{n}$ and a
constant $c>0$ such that%
\begin{equation}
\left\vert K^{\alpha}\left(  x,x+t\mathbf{u}_{0}\right)  \right\vert \geq
c\left\vert t\right\vert ^{\alpha-n},\ \ \ \ \ \text{for all }t\in\mathbb{R}.
\label{steinelliptic}%
\end{equation}

\begin{remark}
\label{real}The functions and kernels in the Calder\'{o}n-Zygmund operators
considered here, are assumed to be complex-valued. However, it should be noted
that for all of the sufficiency proofs in this paper, one may assume without
loss of generality that the functions and kernels are real-valued. It is only
in (\ref{steinelliptic}) that both real and imaginary parts might be needed.
\end{remark}

\subsubsection{Defining the norm inequality\label{Subsubsection norm}}

We follow the approach in \cite[see page 314]{SaShUr9}. So we suppose that
$K^{\alpha}$ is a standard $\left(  \kappa+\delta\right)  $-smooth $\alpha
$-fractional Calder\'{o}n-Zygmund kernel, and we introduce a family $\left\{
\eta_{\delta,R}^{\alpha}\right\}  _{0<\delta<R<\infty}$ of nonnegative
functions on $\left[  0,\infty\right)  $ so that the truncated kernels
$K_{\delta,R}^{\alpha}\left(  x,y\right)  =\eta_{\delta,R}^{\alpha}\left(
\left\vert x-y\right\vert \right)  K^{\alpha}\left(  x,y\right)  $ are bounded
with compact support for fixed $x$ or $y$, and uniformly satisfy
(\ref{sizeandsmoothness'}). Then the truncated operators
\[
T_{\sigma,\delta,R}^{\alpha}f\left(  x\right)  \equiv\int_{\mathbb{R}^{n}%
}K_{\delta,R}^{\alpha}\left(  x,y\right)  f\left(  y\right)  d\sigma\left(
y\right)  ,\ \ \ \ \ x\in\mathbb{R}^{n},
\]
are pointwise well-defined when $f$ is bounded with compact support, and we
will refer to the pair $\left(  K^{\alpha},\left\{  \eta_{\delta,R}^{\alpha
}\right\}  _{0<\delta<R<\infty}\right)  $ as an $\alpha$-fractional singular
integral operator, which we typically denote by $T^{\alpha}$, suppressing the
dependence on the truncations.

\begin{definition}
\label{def bounded}We say that an $\alpha$-fractional singular integral
operator $T^{\alpha}=\left(  K^{\alpha},\left\{  \eta_{\delta,R}^{\alpha
}\right\}  _{0<\delta<R<\infty}\right)  $ satisfies the norm inequality%
\begin{equation}
\left\Vert T_{\sigma}^{\alpha}f\right\Vert _{L^{2}\left(  \omega\right)  }%
\leq\mathfrak{N}_{T^{\alpha}}\left\Vert f\right\Vert _{L^{2}\left(
\sigma\right)  },\ \ \ \ \ f\in L^{2}\left(  \sigma\right)  .
\label{two weight'}%
\end{equation}
provided%
\[
\left\Vert T_{\sigma,\delta,R}^{\alpha}f\right\Vert _{L^{2}\left(
\omega\right)  }\leq\mathfrak{N}_{T^{\alpha}}\left(  \sigma,\omega\right)
\left\Vert f\right\Vert _{L^{2}\left(  \sigma\right)  },\ \ \ \ \ f\in
L^{2}\left(  \sigma\right)  ,0<\delta<R<\infty.
\]

\end{definition}

\begin{description}
\item[Independence of Truncations] \label{independence}In the presence of the
classical Muckenhoupt condition $A_{2}^{\alpha}$, the norm inequality
(\ref{two weight'}) is independent of the choice of truncations used,
including \emph{nonsmooth} truncations as well - see \cite{LaSaShUr3}.
However, in dealing with the Energy Lemma \ref{ener} below, where $\kappa
^{th}$ order Taylor approximations are made on the truncated kernels, it is
necessary to use sufficiently smooth truncations. Similar comments apply to
the Cube Testing conditions.
\end{description}

\subsection{$\kappa$-cube testing conditions}

In this subsection we describe a variety of testing conditions that arise in
the course of our proof, but which do not appear in the statement of our main
Theorem \ref{main}, where only the classical testing condition in
(\ref{Muck and test}) is used.

The $\kappa$\emph{-cube testing conditions} associated with an $\alpha
$-fractional singular integral operator $T^{\alpha}$ are given by%
\begin{align}
\left(  \mathfrak{T}_{T^{\alpha}}^{\left(  \kappa\right)  }\left(
\sigma,\omega\right)  \right)  ^{2}  &  \equiv\sup_{Q\in\mathcal{P}^{n}}%
\max_{0\leq\left\vert \beta\right\vert <\kappa}\frac{1}{\left\vert
Q\right\vert _{\sigma}}\int_{Q}\left\vert T_{\sigma}^{\alpha}\left(
\mathbf{1}_{Q}m_{Q}^{\beta}\right)  \right\vert ^{2}\omega<\infty
,\label{def Kappa polynomial'}\\
\left(  \mathfrak{T}_{\left(  T^{\alpha}\right)  ^{\ast}}^{\left(
\kappa\right)  }\left(  \omega,\sigma\right)  \right)  ^{2}  &  \equiv
\sup_{Q\in\mathcal{P}^{n}}\max_{0\leq\left\vert \beta\right\vert <\kappa}%
\frac{1}{\left\vert Q\right\vert _{\omega}}\int_{Q}\left\vert \left(
T^{\alpha,\ast}\right)  _{\omega}\left(  \mathbf{1}_{Q}m_{Q}^{\beta}\right)
\right\vert ^{2}\sigma<\infty,\nonumber
\end{align}
where $\left(  T^{\alpha,\ast}\right)  _{\omega}=\left(  T_{\sigma}^{\alpha
}\right)  ^{\ast}$, with $m_{Q}^{\beta}\left(  x\right)  \equiv\left(
\frac{x-c_{Q}}{\ell\left(  Q\right)  }\right)  ^{\beta}$ for any cube $Q$ and
multiindex $\beta$, where $c_{Q}$ is the center of the cube $Q$, and where we
interpret the right hand sides as holding uniformly over all sufficiently
smooth truncations of $T^{\alpha}$. Equivalently, in the presence of
$A_{2}^{\alpha}$, we can take a single suitable truncation, see Independence
of Truncations in Subsubsection \ref{independence}.

We also use the \emph{triple }$\kappa$\emph{-cube testing conditions} in which
the integrals are over the triple $3Q$ of $Q$:%
\begin{align}
\left(  \mathfrak{TR}_{T^{\alpha}}^{\left(  \kappa\right)  }\left(
\sigma,\omega\right)  \right)  ^{2}  &  \equiv\sup_{Q\in\mathcal{P}^{n}}%
\max_{0\leq\left\vert \beta\right\vert <\kappa}\frac{1}{\left\vert
Q\right\vert _{\sigma}}\int_{3Q}\left\vert T_{\sigma}^{\alpha}\left(
\mathbf{1}_{Q}m_{Q}^{\beta}\right)  \right\vert ^{2}\omega<\infty
,\label{full testing}\\
\left(  \mathfrak{TR}_{\left(  T^{\alpha}\right)  ^{\ast}}^{\left(
\kappa\right)  }\left(  \omega,\sigma\right)  \right)  ^{2}  &  \equiv
\sup_{Q\in\mathcal{P}^{n}}\max_{0\leq\left\vert \beta\right\vert <\kappa}%
\frac{1}{\left\vert Q\right\vert _{\omega}}\int_{3Q}\left\vert \left(
T^{\alpha,\ast}\right)  _{\omega}\left(  \mathbf{1}_{Q}m_{Q}^{\beta}\right)
\right\vert ^{2}\sigma<\infty.\nonumber
\end{align}
The smaller fractional Poisson integrals $\mathrm{P}_{\kappa}^{\alpha}\left(
Q,\mu\right)  $ used here, in \cite{RaSaWi} and elsewhere, are given by
\begin{equation}
\mathrm{P}_{\kappa}^{\alpha}\left(  Q,\mu\right)  =\int_{\mathbb{R}^{n}}%
\frac{\ell\left(  Q\right)  ^{\kappa}}{\left(  \ell\left(  Q\right)
+\left\vert y-c_{Q}\right\vert \right)  ^{n+\kappa-\alpha}}d\mu\left(
y\right)  ,\ \ \ \ \ \kappa\geq1. \label{def kappa Poisson}%
\end{equation}
The following lemma from \cite[Subsection 4.1 on pages 12-13, especially
Remark 15]{Saw6} was the point of departure for freeing the theory from
reliance on energy conditions when the measures are doubling.

\begin{lemma}
[\cite{Saw6}]\label{doub piv}If $\sigma$ is a doubling measure, then for
sufficiently large $\kappa$ depending on the doubling constant of $\sigma$, we
have
\[
\mathrm{P}_{\kappa}^{\alpha}\left(  Q,\sigma\right)  \approx\frac{\left\vert
Q\right\vert _{\sigma}}{\left\vert Q\right\vert ^{1-\frac{\alpha}{n}}}\text{
and }\mathrm{P}_{\kappa}^{\alpha}\left(  Q,\sigma\right)  ^{2}\left\vert
Q\right\vert _{\omega}\leq CA_{2}^{\alpha}\left(  \sigma,\omega\right)
\left\vert Q\right\vert _{\sigma}.
\]

\end{lemma}

\subsection{Weighted Alpert bases for $L^{2}\left(  \mu\right)  $ and
$L^{\infty}$ control of projections\label{Subsection Haar}}

We now recall the construction of weighted Alpert wavelets in \cite{RaSaWi},
and refer also to \cite{AlSaUr} for the correction of a small oversight in
\cite{RaSaWi}. Let $\mu$ be a locally finite positive Borel measure on
$\mathbb{R}^{n}$, and fix $\kappa\in\mathbb{N}$. For each cube $Q$, denote by
$L_{Q;\kappa}^{2}\left(  \mu\right)  $ the finite dimensional subspace of
$L^{2}\left(  \mu\right)  $ that consists of linear combinations of the
indicators of\ the children $\mathfrak{C}\left(  Q\right)  $ of $Q$ multiplied
by polynomials of degree less than $\kappa$, and such that the linear
combinations have vanishing $\mu$-moments on the cube $Q$ up to order
$\kappa-1$:%
\[
L_{Q;\kappa}^{2}\left(  \mu\right)  \equiv\left\{  f=%
{\displaystyle\sum\limits_{Q^{\prime}\in\mathfrak{C}\left(  Q\right)  }}
\mathbf{1}_{Q^{\prime}}p_{Q^{\prime};\kappa}\left(  x\right)  :\int
_{Q}f\left(  x\right)  x^{\beta}d\mu\left(  x\right)  =0,\ \ \ \text{for
}0\leq\left\vert \beta\right\vert <\kappa\right\}  ,
\]
where $p_{Q^{\prime};\kappa}\left(  x\right)  =\sum_{\beta\in\mathbb{Z}%
_{+}^{n}:\left\vert \beta\right\vert \leq\kappa-1\ }a_{Q^{\prime};\beta
}x^{\beta}$ is a polynomial in $\mathbb{R}^{n}$ of degree less than $\kappa$.
Here $x^{\beta}=x_{1}^{\beta_{1}}x_{2}^{\beta_{2}}...x_{n}^{\beta_{n}}$. Let
$d_{Q;\kappa}\equiv\dim L_{Q;\kappa}^{2}\left(  \mu\right)  $ be the dimension
of the finite dimensional linear space $L_{Q;\kappa}^{2}\left(  \mu\right)  $.

Consider an arbitrary dyadic grid $\mathcal{D}$. For $Q\in\mathcal{D}$, let
$\bigtriangleup_{Q;\kappa}^{\mu}$ denote orthogonal projection onto the finite
dimensional subspace $L_{Q;\kappa}^{2}\left(  \mu\right)  $, and let
$\mathbb{E}_{Q;\kappa}^{\mu}$ denote orthogonal projection onto the finite
dimensional subspace%
\[
\mathcal{P}_{Q;\kappa}^{n}\left(  \mu\right)  \equiv
\mathrm{\operatorname*{Span}}\{\mathbf{1}_{Q}x^{\beta}:0\leq\left\vert
\beta\right\vert <\kappa\}.
\]

For a doubling measure $\mu$, it is proved in \cite{RaSaWi}, that we have the
orthonormal decompositions%
\begin{equation}
f=\sum_{Q\in\mathcal{D}}\bigtriangleup_{Q;\kappa}^{\mu}f,\ \ \ \ \ f\in
L_{\mathbb{R}^{n}}^{2}\left(  \mu\right)  ,\ \ \ \ \ \text{where }\left\langle
\bigtriangleup_{P;\kappa}^{\mu}f,\bigtriangleup_{Q;\kappa}^{\mu}f\right\rangle
=0\text{ for }P\neq Q, \label{Alpert expan}%
\end{equation}
where convergence holds both in $L_{\mathbb{R}^{n}}^{2}\left(  \mu\right)  $
norm and pointwise $\mu$-almost everywhere, the telescoping identities%
\begin{equation}
\mathbf{1}_{Q}\sum_{I:\ Q\subsetneqq I\subset P}\bigtriangleup_{I;\kappa}%
^{\mu}=\mathbb{E}_{Q;\kappa}^{\mu}-\mathbf{1}_{Q}\mathbb{E}_{P;\kappa}^{\mu
}\ \text{ \ for }P,Q\in\mathcal{D}\text{ with }Q\subsetneqq P,
\label{telescoping}%
\end{equation}
and the moment vanishing conditions%
\begin{equation}
\int_{\mathbb{R}^{n}}\bigtriangleup_{Q;\kappa}^{\mu}f\left(  x\right)
\ x^{\beta}d\mu\left(  x\right)  =0,\ \ \ \text{for }Q\in\mathcal{D},\text{
}\beta\in\mathbb{Z}_{+}^{n},\ 0\leq\left\vert \beta\right\vert <\kappa\ .
\label{mom con}%
\end{equation}

We have the bound for the Alpert projections $\mathbb{E}_{I;\kappa}^{\mu}$
(\cite[see (4.7) on page 14]{Saw6}):
\begin{equation}
\left\Vert \mathbb{E}_{I;\kappa}^{\mu}f\right\Vert _{L_{I}^{\infty}\left(
\mu\right)  }\lesssim E_{I}^{\mu}\left\vert f\right\vert \leq\sqrt{\frac
{1}{\left\vert I\right\vert _{\mu}}\int_{I}\left\vert f\right\vert ^{2}d\mu
},\ \ \ \ \ \text{for all }f\in L_{\operatorname*{loc}}^{2}\left(  \mu\right)
. \label{analogue}%
\end{equation}
In terms of the Alpert coefficient vectors $\widehat{f}\left(  I\right)
\equiv\left\{  \left\langle f,h_{I;\kappa}^{\mu,a}\right\rangle \right\}
_{a\in\Gamma_{I,n,\kappa}}$ for an orthonormal basis $\left\{  h_{I;\kappa
}^{\mu,a}\right\}  _{a\in\Gamma_{I,n,\kappa}}$ of $L_{I;\kappa}^{2}\left(
\mu\right)  $ where $\Gamma_{I,n,\kappa}$ is a convenient finite index set of
size $d_{Q;\kappa}$, we thus have%
\begin{equation}
\left\vert \widehat{f}\left(  I\right)  \right\vert =\left\Vert \bigtriangleup
_{I;\kappa}^{\sigma}f\right\Vert _{L^{2}\left(  \sigma\right)  }\leq\left\Vert
\bigtriangleup_{I;\kappa}^{\sigma}f\right\Vert _{L^{\infty}\left(
\sigma\right)  }\sqrt{\left\vert I\right\vert _{\sigma}}\leq C\left\Vert
\bigtriangleup_{I;\kappa}^{\sigma}f\right\Vert _{L^{2}\left(  \sigma\right)
}=C\left\vert \widehat{f}\left(  I\right)  \right\vert . \label{analogue'}%
\end{equation}

\subsection{The Pivotal Lemma}

For $0\leq\alpha<n$, let $\mathrm{P}^{\alpha}\left(  J,\mu\right)
\equiv\mathrm{P}_{1}^{\alpha}\left(  J,\mu\right)  $ denote the standard
Poisson integral, where $\mathrm{P}_{m}^{\alpha}\left(  J,\mu\right)  $ is as
defined in (\ref{def kappa Poisson}). The following extension of the `energy
lemma' is due to Rahm, Sawyer and Wick \cite{RaSaWi} in the case the
polynomial $R\left(  x\right)  $ is constant, and this case is proved in
detail in \cite[Lemmas 28 and 29 on pages 27-30]{Saw6}. Note the crucial
assumption below is that $\Psi_{J}$ has at least $2\kappa-1$ vanishing moments
while the polynomial $R$ has degree at most $\kappa-1$.

\begin{lemma}
[\textbf{Pivotal Lemma}]\label{ener}Fix $\kappa\geq1$. Let $J\ $be a cube in
$\mathcal{D}$, and let $\Psi_{J}$ be an $L^{2}\left(  \omega\right)  $
function supported in $J$ with vanishing $\omega$-means of all orders less
than $2\kappa$. Let $R\left(  x\right)  $ be a polynomial of degree less than
$\kappa$ that satisfies $\sup_{x\in J}\left\vert R\left(  x\right)
\right\vert \leq1$. Let $\sigma$ be a positive measure supported in
$\mathbb{R}^{n}\setminus\gamma J$ with $\gamma>1$, and let $T^{\alpha}$ be a
standard $\alpha$-fractional singular integral operator with $0\leq\alpha<n$.
Then we have the `pivotal' bound%
\begin{equation}
\left\vert \left\langle RT^{\alpha}\left(  \varphi\sigma\right)  ,\Psi
_{J}\right\rangle _{L^{2}\left(  \omega\right)  }\right\vert \lesssim
C_{\gamma}\mathrm{P}_{\kappa}^{\alpha}\left(  J,\nu\right)  \sqrt{\left\vert
J\right\vert _{\omega}}\left\Vert \Psi_{J}\right\Vert _{L^{2}\left(
\omega\right)  }\lesssim C_{\gamma}\sqrt{A_{2}^{\alpha}\left(  \sigma
,\omega\right)  }\sqrt{\left\vert J\right\vert _{\nu}}\left\Vert \Psi
_{J}\right\Vert _{L^{2}\left(  \omega\right)  }, \label{piv bound}%
\end{equation}
for any function $\varphi$ with $\left\vert \varphi\right\vert \leq1$.
\end{lemma}

To obtain Lemma \ref{ener} in the case of a general polynomial $R\left(
x\right)  $ of degree at most $\kappa$, note that we \textbf{cannot} simply
replace the Taylor expansion of the function $x\rightarrow K\left(
x,y\right)  $ that arises in the proof in \cite[Lemmas 28 and 29 on pages
27-30]{Saw6}, by the Taylor expansion of the function $x\rightarrow R\left(
x\right)  K\left(  x,y\right)  $, since this latter kernel no longer satisfies
Calder\'{o}n-Zygmund estimates in the $x$ variable. For example, if all of the
derivatives in $\partial_{x}^{\beta}\left[  R\left(  x\right)  K\left(
x,y\right)  \right]  $ hit the polynomial $R\left(  x\right)  =x^{\beta}$,
then we get
\[
\left\vert \left[  \partial_{x}^{\beta}R\left(  x\right)  \right]  K\left(
x,y\right)  \right\vert \sim\beta!\left\vert K\left(  x,y\right)  \right\vert
\lesssim\left\vert x-y\right\vert ^{\alpha-n},
\]
which can be much larger than the Calder\'{o}n-Zygmund bound
$C_{\operatorname*{CZ}}\left\vert x-y\right\vert ^{\alpha-n-\left\vert
\beta\right\vert }$ for $\left\vert x-y\right\vert >1$.

On the other hand, the function $R\left(  x\right)  \Psi_{J}\left(  x\right)
$ is supported in $J$ and has at least $\left(  2\kappa-1\right)  -\left(
\kappa-1\right)  =\kappa$ vanishing moments since the degree of $R\left(
x\right)  $ is at most $\kappa-1$, and there are at least $2\kappa-1$
vanishing moments in $\Psi_{J}\left(  x\right)  $. Indeed,%
\[
\int R\left(  x\right)  \Psi_{J}\left(  x\right)  x^{\beta}dx=\int\Psi
_{J}\left(  x\right)  \left[  R\left(  x\right)  x^{\beta}\right]  dx=0
\]
since $R\left(  x\right)  x^{\beta}$ is a polynomial of degree less than
$2\kappa$. Thus the argument in \cite{Saw6} applies to prove the Pivotal Lemma
\ref{ener}.

We also recall from \cite[Lemma 33]{Saw6} the following Poisson estimate, that
is a straightforward extension of the case $m=1$ due to Nazarov, Treil and
Volberg in \cite{NTV4}. This lemma is the key to exploiting the crucial
reduction to $\operatorname{good}$ cubes $J$ that we use below, see
\cite{NTV4} and \cite{NTV}.

\begin{lemma}
\label{Poisson inequality}Fix $m\geq1$. Suppose that $J\subset I\subset K$ and
that $\operatorname*{dist}\left(  J,\partial I\right)  >2\sqrt{n}\ell\left(
J\right)  ^{\varepsilon}\ell\left(  I\right)  ^{1-\varepsilon}$. Then%
\begin{equation}
\mathrm{P}_{m}^{\alpha}(J,\sigma\mathbf{1}_{K\setminus I})\lesssim\left(
\frac{\ell\left(  J\right)  }{\ell\left(  I\right)  }\right)  ^{m-\varepsilon
\left(  n+m-\alpha\right)  }\mathrm{P}_{m}^{\alpha}(I,\sigma\mathbf{1}%
_{K\setminus I}). \label{e.Jsimeq}%
\end{equation}

\end{lemma}

\section{The Calder\'{o}n-Zygmund corona decomposition}

To set the stage for control of the stopping form below in the absence of the
energy condition, we construct the \emph{Calder\'{o}n-Zygmund} corona
decomposition for a function $f$ in $L^{2}\left(  \sigma\right)  $ that is
supported in a dyadic cube $F_{1}^{0}$. Fix $\Gamma>1$ and define
$\mathcal{G}_{0}=\left\{  F_{1}^{0}\right\}  $ to consist of the single cube
$F_{1}^{0}$, and define the first generation $\mathcal{G}_{1}=\left\{
F_{k}^{1}\right\}  _{k}$ of \emph{CZ stopping children} of $F_{1}^{0}$ to be
the \emph{maximal} dyadic subcubes $I$ of $F_{0}$ satisfying%
\[
E_{I}^{\sigma}\left\vert f\right\vert \geq\Gamma E_{F_{1}^{0}}^{\sigma
}\left\vert f\right\vert .
\]
Then define the second generation $\mathcal{G}_{2}=\left\{  F_{k}^{2}\right\}
_{k}$ of CZ stopping children of $F_{1}^{0}$ to be the \emph{maximal} dyadic
subcubes $I$ of some $F_{k}^{1}\in\mathcal{G}_{1}$ satisfying%
\[
E_{I}^{\sigma}\left\vert f\right\vert \geq\Gamma E_{F_{k}^{1}}^{\sigma
}\left\vert f\right\vert .
\]
Continue by recursion to define $\mathcal{G}_{n}$ for all $n\geq0$, and then
set
\[
\mathcal{F\equiv}%
{\displaystyle\bigcup\limits_{n=0}^{\infty}}
\mathcal{G}_{n}=\left\{  F_{k}^{n}:n\geq0,k\geq1\right\}
\]
to be the set of all CZ $\kappa$-pivotal stopping intervals in $F_{1}^{0}$
obtained in this way.

The $\sigma$-Carleson condition for $\mathcal{F}$ follows as usual from the
first step,%
\[
\sum_{F^{\prime}\in\mathfrak{C}_{\mathcal{F}}\left(  F\right)  }\left\vert
F^{\prime}\right\vert _{\sigma}\leq\frac{1}{\Gamma}\sum_{F^{\prime}%
\in\mathfrak{C}_{\mathcal{F}}\left(  F\right)  }\left\{  \mathrm{P}_{\kappa
}^{\alpha}\left(  F^{\prime},\mathbf{1}_{F}\sigma\right)  ^{2}\left\vert
F^{\prime}\right\vert _{\omega}+\frac{1}{E_{F}^{\sigma}\left\vert f\right\vert
}\int_{F^{\prime}}\left\vert f\right\vert d\sigma\right\}  \leq\frac{1}%
{\Gamma}\left(  A_{2}^{\alpha}\left(  \sigma,\omega\right)  +1\right)
\left\vert F\right\vert _{\sigma}.
\]
Moreover, if we define
\begin{equation}
\alpha_{\mathcal{F}}\left(  F\right)  \equiv\sup_{F^{\prime}\in\mathcal{F}%
:\ F\subset F^{\prime}}E_{F^{\prime}}^{\sigma}\left\vert f\right\vert ,
\label{def alpha}%
\end{equation}
then in each corona
\[
\mathcal{C}_{F}\equiv\left\{  I\in\mathcal{D}:I\subset F\text{ and
}I\not \subset F^{\prime}\text{ for any }F^{\prime}\in\mathcal{F}\text{ with
}F^{\prime}\varsubsetneqq F\right\}  ,
\]
we have, from the definition of the stopping times, the average control%
\begin{equation}
E_{I}^{\sigma}\left\vert f\right\vert <\Gamma\alpha_{\mathcal{F}}\left(
F\right)  ,\ \ \ \ \ I\in\mathcal{C}_{F}\text{ and }F\in\mathcal{F}.
\label{average control}%
\end{equation}
Finally, as in \cite{NTV4}, \cite{LaSaShUr3} and \cite{SaShUr7}, we obtain the
Carleson condition and the quasiorthogonality inequality,%
\begin{equation}
\sum_{F^{\prime}\preceq F}\left\vert F^{\prime}\right\vert _{\sigma}\leq
C_{0}\left\vert F\right\vert _{\sigma}\text{ for all }F\in\mathcal{F}%
;\text{\ and }\sum_{F\in\mathcal{F}}\alpha_{\mathcal{F}}\left(  F\right)
^{2}\left\vert F\right\vert _{\sigma}\mathbf{\leq}C_{0}^{2}\left\Vert
f\right\Vert _{L^{2}\left(  \sigma\right)  }^{2}. \label{Car and quasi}%
\end{equation}

Define the two corona projections%
\[
\mathsf{P}_{\mathcal{C}_{F}}^{\sigma}\equiv\sum_{I\in\mathcal{C}_{F}%
}\bigtriangleup_{I;\kappa}^{\sigma}\text{ and }\mathsf{P}_{\mathcal{C}%
_{F}^{\mathbf{\tau}-\operatorname*{shift}}}^{\omega}\equiv\sum_{J\in
\mathcal{C}_{F}^{\mathbf{\tau}-\operatorname*{shift}}}\bigtriangleup
_{J;\kappa}^{\omega}\ ,
\]
where%
\begin{align}
\mathcal{C}_{F}^{\mathbf{\tau}-\operatorname*{shift}}  &  \equiv\left[
\mathcal{C}_{F}\setminus\mathcal{N}_{\mathcal{D}}^{\mathbf{\tau}}\left(
F\right)  \right]  \cup%
{\displaystyle\bigcup\limits_{F^{\prime}\in\mathfrak{C}_{\mathcal{F}}\left(
F\right)  }}
\left[  \mathcal{N}_{\mathcal{D}}^{\mathbf{\tau}}\left(  F^{\prime}\right)
\setminus\mathcal{N}_{\mathcal{D}}^{\mathbf{\tau}}\left(  F\right)  \right]
;\label{def shift}\\
\text{where }\mathcal{N}_{\mathcal{D}}^{\mathbf{\tau}}\left(  F\right)   &
\equiv\left\{  J\in\mathcal{D}:J\subset F\text{ and }\ell\left(  J\right)
>2^{-\mathbf{\tau}}\ell\left(  F\right)  \right\}  ,\nonumber
\end{align}
and note that $f=\sum_{F\in\mathcal{F}}\mathsf{P}_{\mathcal{C}_{F}}^{\sigma}%
f$. Thus the corona $\mathcal{C}_{F}^{\mathbf{\tau}-\operatorname*{shift}}$
has the top $\mathbf{\tau}$ levels from $\mathcal{C}_{F}$ removed, and
includes the first $\mathbf{\tau}$ levels from each of its $\mathcal{F}%
$-children, except if they have already been removed.

\begin{remark}
\label{pd}The shifted coronas are pairwise disjoint, since if $F^{\prime
\prime},F\in\mathcal{F}$ and $J\in\mathcal{C}_{F}^{\mathbf{\tau}%
-\operatorname*{shift}}\cap\mathcal{C}_{F^{\prime\prime}}^{\mathbf{\tau
}-\operatorname*{shift}}$, then either $F^{\prime\prime}\subset F$ or
$F\subset F^{\prime\prime}$. Thus it suffices to assume $F^{\prime\prime
}\subsetneqq F$ and derive a contradiction. But then $J\subset F^{\prime
\prime}$ and by (\ref{def shift}) and the assumption that $J\in\mathcal{C}%
_{F}^{\mathbf{\tau}-\operatorname*{shift}}$, we have $J\in\mathcal{N}%
_{\mathcal{D}}^{\mathbf{\tau}}\left(  F^{\prime}\right)  \setminus
\mathcal{N}_{\mathcal{D}}^{\mathbf{\tau}}\left(  F\right)  $ for some
$F^{\prime}\in\mathfrak{C}_{\mathcal{F}}\left(  F\right)  $. Thus $J\subset
F^{\prime\prime}\subset F^{\prime}\subset F$, and the assumption that
$J\in\mathcal{C}_{F^{\prime\prime}}^{\mathbf{\tau}-\operatorname*{shift}}$
implies that $J\not \in \mathcal{N}_{\mathcal{D}}^{\mathbf{\tau}}\left(
F^{\prime\prime}\right)  $, contradicting $J\in\mathcal{N}_{\mathcal{D}%
}^{\mathbf{\tau}}\left(  F^{\prime}\right)  $ with $F^{\prime\prime}\subset
F^{\prime}$.
\end{remark}

The main result we need from \cite{Saw6} regarding these coronas is the
Intertwining Proposition.

\begin{proposition}
[{The Intertwining Proposition \cite[see Subsection 6.4]{Saw6}}]%
\label{strongly adapted}Suppose that $\mathcal{F}$ satisfies both%
\[
\sum_{F^{\prime}\in\mathcal{F}:F^{\prime}\subset F}\left\vert F^{\prime
}\right\vert _{\sigma}\leq C_{0}\left\vert F\right\vert _{\sigma}\text{ for
all }F\in\mathcal{F},\text{\ and }\sum_{F\in\mathcal{F}}\alpha_{\mathcal{F}%
}\left(  F\right)  ^{2}\left\vert F\right\vert _{\sigma}\mathbf{\leq}C_{0}%
^{2}\left\Vert f\right\Vert _{L^{2}\left(  \sigma\right)  }^{2},
\]
where $\alpha_{\mathcal{F}}\left(  F\right)  $ is as in (\ref{def alpha}), and
that%
\[
\left\Vert \bigtriangleup_{I;\kappa}^{\sigma}f\right\Vert _{L^{\infty}\left(
\sigma\right)  }\leq C\alpha_{\mathcal{F}}\left(  F\right)  ,\ \ \ \ \ f\in
L^{2}\left(  \sigma\right)  ,\ I\in\mathcal{C}_{F}.
\]
Then for $\operatorname*{good}$ functions $f\in L^{2}\left(  \sigma\right)  $
and $g\in L^{2}\left(  \omega\right)  $, and with $\kappa\geq1$, we have%
\[
\left\vert \sum_{F\in\mathcal{F}}\ \sum_{I:\ I\supsetneqq F}\ \left\langle
T_{\sigma}^{\alpha}\bigtriangleup_{I;\kappa}^{\sigma}f,\mathsf{P}%
_{\mathcal{C}_{F}^{\mathbf{\tau}-\operatorname*{shift}}}^{\omega
}g\right\rangle _{\omega}\right\vert \lesssim\left(  \sqrt{A_{2}^{\alpha}%
}+\mathfrak{T}_{T^{\alpha}}^{\kappa}\right)  \ \left\Vert f\right\Vert
_{L^{2}\left(  \sigma\right)  }\left\Vert g\right\Vert _{L^{2}\left(
\omega\right)  }.
\]

\end{proposition}

\subsection{Parallel corona decompositions}

In this subsection, we recall certain material on parallel corona
decompositions from \cite{Saw6}. Strictly speaking, we will not use any of
this material in our paper, but it is included here as it motivates a key
construction used later on for the `above' stopping form. Let $\left(
C_{0},\mathcal{A},\alpha_{\mathcal{A}}\right)  $ constitute stopping data for
$f\in L^{2}\left(  \sigma\right)  $,\ and let $\left(  C_{0},\mathcal{B}%
,\alpha_{\mathcal{B}}\right)  $ constitute stopping data for $g\in
L^{2}\left(  \omega\right)  $ as in the previous subsubsection. We now
organize the bilinear form,%
\begin{align*}
\left\langle T_{\sigma}^{\alpha}f,g\right\rangle _{\omega}  &  =\left\langle
T_{\sigma}^{\alpha}\left(  \sum_{I\in\mathcal{D}}\bigtriangleup_{I;\kappa_{1}%
}^{\sigma}f\right)  ,\left(  \sum_{J\in\mathcal{D}}\bigtriangleup
_{J;\kappa_{2}}^{\omega}g\right)  \right\rangle _{\omega}=\sum_{I\in
\mathcal{D}\ \text{and }J\in\mathcal{D}}\left\langle T_{\sigma}^{\alpha
}\left(  \bigtriangleup_{I;\kappa_{1}}^{\sigma}f\right)  ,\left(
\bigtriangleup_{J;\kappa_{2}}^{\omega}g\right)  \right\rangle _{\omega}\\
&  =\sum_{\left(  A,B\right)  \in\mathcal{A}\times\mathcal{B}}\sum
_{I\in\mathcal{C}_{\mathcal{A}}\left(  A\right)  \text{ and }J\in
\mathcal{C}_{\mathcal{B}}\left(  B\right)  }\left\langle T_{\sigma}^{\alpha
}\left(  \bigtriangleup_{I;\kappa_{1}}^{\sigma}f\right)  ,\left(
\bigtriangleup_{J;\kappa_{2}}^{\omega}g\right)  \right\rangle _{\omega}%
=\sum_{\left(  A,B\right)  \in\mathcal{A}\times\mathcal{B}}\left\langle
T_{\sigma}^{\alpha}\left(  \mathsf{P}_{\mathcal{C}_{\mathcal{A}}\left(
A\right)  }^{\sigma}f\right)  ,\mathsf{P}_{\mathcal{C}_{\mathcal{B}}\left(
B\right)  }^{\omega}g\right\rangle _{\omega}\ ,
\end{align*}
as a sum over the families of Calder\'{o}n-Zygmund stopping cubes
$\mathcal{A}$ and $\mathcal{B}$, and then decompose this sum by the
\emph{Parallel Corona decomposition},\emph{\ }in which the `diagonal cut' in
the bilinear form is made according to the relative positions of intersecting
coronas, rather than the traditional way of making the `diagonal cut'
according to relative side lengths of cubes. See e.g. \cite{Saw6} for more on
the parallel corona decomposition.

We have%
\begin{align}
&  \left\langle T_{\sigma}^{\alpha}f,g\right\rangle _{\omega}=\sum_{\left(
A,B\right)  \in\mathcal{A}\times\mathcal{B}}\left\langle T_{\sigma}^{\alpha
}\left(  \mathsf{P}_{\mathcal{C}_{\mathcal{A}}\left(  A\right)  }^{\sigma
}f\right)  ,\mathsf{P}_{\mathcal{C}_{\mathcal{B}}\left(  B\right)  }^{\omega
}g\right\rangle _{\omega}\label{parallel corona decomp'}\\
&  =\left\{  \sum_{\left(  A,B\right)  \in\operatorname{Near}\left(
\mathcal{A}\times\mathcal{B}\right)  }+\sum_{\left(  A,B\right)
\in\operatorname{Disjoint}\left(  \mathcal{A}\times\mathcal{B}\right)  }%
+\sum_{\left(  A,B\right)  \in\operatorname{Far}\left(  \mathcal{A}%
\times\mathcal{B}\right)  }\right\}  \left\langle T_{\sigma}^{\alpha}\left(
\mathsf{P}_{\mathcal{C}_{\mathcal{A}}\left(  A\right)  }^{\sigma}f\right)
,\mathsf{P}_{\mathcal{C}_{\mathcal{B}}\left(  B\right)  }^{\omega
}g\right\rangle _{\omega}\nonumber\\
&  \equiv\operatorname{Near}\left(  f,g\right)  +\operatorname{Disjoint}%
\left(  f,g\right)  +\operatorname{Far}\left(  f,g\right)  .\nonumber
\end{align}
Here $\operatorname{Near}\left(  \mathcal{A}\times\mathcal{B}\right)  $ is the
set of pairs $\left(  A,B\right)  \in\mathcal{A}\times\mathcal{B}$ such that
one of $A,B$ is contained in the other, and there is no $A_{1}\in\mathcal{A}$
with $B\subset A_{1}\subsetneqq A$, nor is there $B_{1}\in\mathcal{B}$ with
$A\subset B_{1}\subsetneqq B$. The set $\operatorname{Disjoint}\left(
\mathcal{A}\times\mathcal{B}\right)  $ is the set of pairs $\left(
A,B\right)  \in\mathcal{A}\times\mathcal{B}$ such that $A\cap B=\emptyset$.
The set $\operatorname{Far}\left(  \mathcal{A}\times\mathcal{B}\right)  $ is
the complement of $\operatorname{Near}\left(  \mathcal{A}\times\mathcal{B}%
\right)  \cup\operatorname{Disjoint}\left(  \mathcal{A}\times\mathcal{B}%
\right)  $ in $\mathcal{A}\times\mathcal{B}$:%
\[
\operatorname{Far}\left(  \mathcal{A}\times\mathcal{B}\right)  =\left(
\mathcal{A}\times\mathcal{B}\right)  \setminus\left\{  \operatorname{Near}%
\left(  \mathcal{A}\times\mathcal{B}\right)  \cup\operatorname{Disjoint}%
\left(  \mathcal{A}\times\mathcal{B}\right)  \right\}  .
\]
Note that if $\left(  A,B\right)  \in\operatorname{Far}\left(  \mathcal{A}%
\times\mathcal{B}\right)  $, then \textbf{either} $B\subset A^{\prime}$ for
some $A^{\prime}\in\mathfrak{C}_{\mathcal{A}}\left(  A\right)  $, \textbf{or}
$A\subset B^{\prime}$ for some $B^{\prime}\in\mathfrak{C}_{\mathcal{B}}\left(
B\right)  $.

We further decompose the near form $\operatorname{Near}\left(  f,g\right)  $
into%
\begin{align*}
\operatorname{Near}\left(  f,g\right)   &  =\left\{  \sum_{\substack{\left(
A,B\right)  \in\operatorname{Near}\left(  \mathcal{A}\times\mathcal{B}\right)
\\B\subset A}}+\sum_{\substack{\left(  A,B\right)  \in\operatorname{Near}%
\left(  \mathcal{A}\times\mathcal{B}\right)  \\A\subset B}}\right\}
\left\langle T_{\sigma}^{\alpha}\left(  \mathsf{P}_{\mathcal{C}_{\mathcal{A}%
}\left(  A\right)  }^{\sigma}f\right)  ,\mathsf{P}_{\mathcal{C}_{\mathcal{B}%
}\left(  B\right)  }^{\omega}g\right\rangle _{\omega}\\
&  =\operatorname{Near}_{\operatorname*{below}}\left(  f,g\right)
+\operatorname{Near}_{\operatorname*{above}}\left(  f,g\right)  .
\end{align*}
The $\operatorname{Near}_{\operatorname*{below}}\left(  f,g\right)  $ form can
be controlled by the Indicator/Cube Testing condition (\ref{ind}) if we define
projections%
\[
\mathsf{Q}_{A}^{\omega}g\equiv\sum_{\substack{B\in\mathcal{B}:\ \left(
A,B\right)  \in\operatorname{Near}\left(  \mathcal{A}\times\mathcal{B}\right)
\\B\subset A}}\mathsf{P}_{\mathcal{C}_{\mathcal{B}}\left(  B\right)  }%
^{\omega}g,
\]
and observe that, while the Alpert support of $\mathsf{Q}_{A}^{\omega}$ need
not be contained in the corona $\mathcal{C}_{\mathcal{A}}\left(  A\right)  $,
these projections are nevertheless mutually orthogonal in the index
$A\in\mathcal{A}$, since for $\left(  A,B\right)  \in\operatorname{Near}%
\left(  \mathcal{A}\times\mathcal{B}\right)  $ there is no $A_{1}%
\in\mathcal{A}$ with $B\subset A_{1}\subsetneqq A$. Indeed,
\begin{align}
&  \left\vert \operatorname{Near}_{\operatorname*{below}}\left(  f,g\right)
\right\vert =\sum_{A\in\mathcal{A}}\left\vert \left\langle T_{\sigma}^{\alpha
}\mathsf{P}_{\mathcal{C}_{\mathcal{A}}\left(  A\right)  }^{\sigma}%
f,\mathsf{Q}_{A}^{\omega}g\right\rangle _{\omega}\right\vert
\label{near porism}\\
&  \leq\sum_{A\in\mathcal{A}}\left\Vert T_{\sigma}^{\alpha}\mathsf{P}%
_{\mathcal{C}_{\mathcal{A}}\left(  A\right)  }^{\sigma}f\right\Vert
_{L^{2}\left(  \omega\right)  }\left\Vert \mathsf{Q}_{A}^{\omega}g\right\Vert
_{L^{2}\left(  \omega\right)  }\lesssim\mathfrak{T}_{T^{\alpha}}%
^{\operatorname*{ind}}\left(  \sigma,\omega\right)  \sum_{A\in\mathcal{A}%
}\alpha_{\mathcal{A}}\left(  A\right)  \sqrt{\left\vert A\right\vert _{\sigma
}}\left\Vert \mathsf{Q}_{A}^{\omega}g\right\Vert _{L^{2}\left(  \omega\right)
}\nonumber\\
&  \leq\mathfrak{T}_{T^{\alpha}}^{\operatorname*{ind}}\left(  \sigma
,\omega\right)  \left(  \sum_{A\in\mathcal{A}}\alpha_{\mathcal{A}}\left(
A\right)  ^{2}\left\vert A\right\vert _{\sigma}\right)  ^{\frac{1}{2}}\left(
\sum_{A\in\mathcal{A}}\left\Vert \mathsf{Q}_{A}^{\omega}g\right\Vert
_{L^{2}\left(  \omega\right)  }^{2}\right)  ^{\frac{1}{2}}\lesssim
\mathfrak{T}_{T^{\alpha}}^{\operatorname*{ind}}\left(  \sigma,\omega\right)
\left\Vert f\right\Vert _{L^{2}\left(  \sigma\right)  }\left\Vert g\right\Vert
_{L^{2}\left(  \omega\right)  }\ ,\nonumber
\end{align}
by quasi-orthogonality and the fact that the projections $\mathsf{Q}%
_{A}^{\omega}$ are mutually orthogonal in the index $A\in\mathcal{A}$.

\section{Reduction of the proof to local forms}

The proof of Theorem \ref{main} will require significant new arguments beyond
those in \cite{Saw6}. In particular, we will use the shifted corona
decomposition as in \cite{LaSaShUr3} and \cite{SaShUr7}, instead of the
parallel corona decomposition used in \cite{Saw6}, and construct a paraproduct
- stopping - neighbour decomposition of NTV type for a local Alpert form, but
complicated by the fact that singular integrals do not in general commute with
multiplication by polynomials.

To prove Theorem \ref{main}, we begin by proving the bilinear form bound,%
\[
\left\vert \left\langle T_{\sigma}^{\alpha}f,g\right\rangle _{\omega
}\right\vert \lesssim\left(  \sqrt{A_{2}^{\alpha}\left(  \sigma,\omega\right)
}+\mathfrak{TR}_{T^{\alpha}}^{\left(  \kappa\right)  }\left(  \sigma
,\omega\right)  +\mathfrak{TR}_{\left(  T^{\alpha}\right)  ^{\ast}}^{\left(
\kappa\right)  }\left(  \omega,\sigma\right)  \right)  \ \left\Vert
f\right\Vert _{L^{2}\left(  \sigma\right)  }\left\Vert g\right\Vert
_{L^{2}\left(  \omega\right)  },
\]
for functions $f\in L^{2}\left(  \sigma\right)  $ and $g\in L^{2}\left(
\omega\right)  $. Following the weighted Haar expansions of Nazarov, Treil and
Volberg, we write $f$ and $g$ in weighted Alpert wavelet expansions,%
\begin{equation}
\left\langle T_{\sigma}^{\alpha}f,g\right\rangle _{\omega}=\left\langle
T_{\sigma}^{\alpha}\left(  \sum_{I\in\mathcal{D}}\bigtriangleup_{I;\kappa_{1}%
}^{\sigma}f\right)  ,\left(  \sum_{J\in\mathcal{D}}\bigtriangleup
_{J;\kappa_{2}}^{\omega}g\right)  \right\rangle _{\omega}=\sum_{I\in
\mathcal{D}\ \text{and }J\in\mathcal{D}}\left\langle T_{\sigma}^{\alpha
}\left(  \bigtriangleup_{I;\kappa_{1}}^{\sigma}f\right)  ,\left(
\bigtriangleup_{J;\kappa_{2}}^{\omega}g\right)  \right\rangle _{\omega}\ .
\label{expand}%
\end{equation}
Then the sum is further decomposed by first \emph{Cube Size Splitting}, then
using the \emph{Shifted Corona Decomposition}, according to the
\emph{Canonical Splitting}. We assume the reader is reasonably familiar with
the notation and arguments in the first eight sections of \cite{SaShUr7}. The
$n$-dimensional decompositions used in \cite{SaShUr7} are in spirit the same
as the one-dimensional decompositions in \cite{LaSaShUr3}, as well as the
$n$-dimensional decompositions in \cite{LaWi}, but differ in significant details.

A fundamental result of Nazarov, Treil and Volberg \cite{NTV} is that all the
cubes $I$ and $J$ appearing in the bilinear form above may be assumed to be
$\left(  \mathbf{r},\varepsilon\right)  -\operatorname*{good}$\footnote{See
also \cite[Subsection 3.1]{SaShUr10} for a treatment using finite collections
of grids, in which case the conditional probability arguments are
elementary.}, where a dyadic interval $K$ is $\left(  \mathbf{r}%
,\varepsilon\right)  -\operatorname*{good}$, or simply $\operatorname*{good}$,
if for \emph{every} dyadic supercube $L$ of $K$, it is the case that
\textbf{either} $K$ has side length at least $2^{1-\mathbf{r}}$ times that of
$L$, \textbf{or} $K\Subset_{\left(  \mathbf{r},\varepsilon\right)  }L$. We say
that a dyadic cube $K$ is $\left(  \mathbf{r},\varepsilon\right)
$-\emph{deeply embedded} in a dyadic cube $L$, or simply $\mathbf{r}%
$\emph{-deeply embedded} in $L$, which we write as $K\Subset_{\mathbf{r}%
,\varepsilon}L$, when $K\subset L$ and both
\begin{align}
\ell\left(  K\right)   &  \leq2^{-\mathbf{r}}\ell\left(  L\right)
,\label{def deep embed}\\
\operatorname*{dist}\left(  K,%
{\displaystyle\bigcup\limits_{L^{\prime}\in\mathfrak{C}_{\mathcal{D}}L}}
\partial L^{\prime}\right)   &  \geq2\ell\left(  K\right)  ^{\varepsilon}%
\ell\left(  L\right)  ^{1-\varepsilon}.\nonumber
\end{align}

Here is a brief schematic diagram derived from \cite{SaShUr7} summarizing the
shifted corona decompositions that we will follow using Alpert wavelet
expansions below. We first introduce parameters as in \cite{SaShUr7}. We will
choose $\varepsilon>0$ sufficiently small later in the argument, and then
$\mathbf{r}$ must be chosen sufficiently large depending on $\varepsilon$ in
order to reduce matters to $\left(  \mathbf{r},\varepsilon\right)
-\operatorname{good}$ functions by the NTV argument in \cite{NTV}.

\begin{definition}
\label{def parameters}The parameters $\mathbf{\tau}$ and $\mathbf{\rho}$ are
fixed to satisfy
\[
\mathbf{\tau}>\mathbf{r}\text{ and }\mathbf{\rho}>\mathbf{r}+\mathbf{\tau},
\]
where $\mathbf{r}$ is the goodness parameter already fixed.
\end{definition}

Here is a brief diagram highlighting the various decompositions of the
bilinear form $\left\langle T_{\sigma}^{\alpha}f,g\right\rangle _{\omega}$. We
will treat the below form $\mathsf{B}_{\subset_{\mathbf{\rho},\varepsilon}%
}\left(  f,g\right)  $ in detail under the assumption that $\kappa_{2}%
\geq2\kappa_{1}$, and then turn to the above form $\mathsf{B}_{\supset
_{\mathbf{\rho},\varepsilon}}\left(  f,g\right)  $, which is handled in the
same way except for the diagonal form, which requires the \emph{above}
indicator / cube testing condition due to the asymmetry in the assumption
$\kappa_{2}\geq2\kappa_{1}$.%

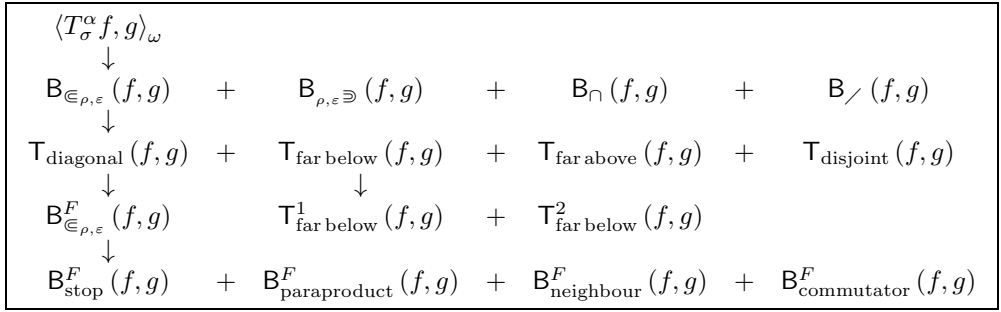
\[
\fbox{$%
\begin{array}
[c]{ccccccc}%
\left\langle T_{\sigma}^{\alpha}f,g\right\rangle _{\omega} &  &  &  &  &  & \\
\downarrow &  &  &  &  &  & \\
\mathsf{B}_{\subset_{\mathbf{\rho},\varepsilon}}\left(  f,g\right)  & + &
\mathsf{B}_{\supset_{\mathbf{\rho},\varepsilon}}\left(  f,g\right)  & + &
\mathsf{B}_{\cap}\left(  f,g\right)  & + & \mathsf{B}_{\diagup}\left(
f,g\right) \\
\downarrow &  &  &  &  &  & \\
\mathsf{T}_{\operatorname*{diagonal}}^{\subset_{\mathbf{\rho},\varepsilon}%
}\left(  f,g\right)  & + & \mathsf{T}_{\operatorname*{far}%
\operatorname*{below}}^{\subset_{\mathbf{\rho},\varepsilon}}\left(  f,g\right)
& + & \mathsf{T}_{\operatorname*{far}\operatorname*{above}}^{\subset
_{\mathbf{\rho},\varepsilon}}\left(  f,g\right)  & + & \mathsf{T}%
_{\operatorname*{disjoint}}^{\subset_{\mathbf{\rho},\varepsilon}}\left(
f,g\right) \\
\downarrow &  & \downarrow &  &  &  & \\
\mathsf{T}_{\operatorname*{diagonal}}^{\subset_{\mathbf{\rho},\varepsilon}%
,F}\left(  f,g\right)  &  & \mathsf{T}_{\operatorname*{far}%
\operatorname*{below}}^{1}\left(  f,g\right)  & + & \mathsf{T}%
_{\operatorname*{far}\operatorname*{below}}^{2}\left(  f,g\right)  &  & \\
\downarrow &  &  &  &  &  & \\
\mathsf{T}_{\operatorname{stop}}^{\subset_{\mathbf{\rho},\varepsilon}%
,F}\left(  f,g\right)  & + & \mathsf{T}_{\operatorname{paraproduct}}%
^{\subset_{\mathbf{\rho},\varepsilon}F}\left(  f,g\right)  & + &
\mathsf{T}_{\operatorname{neighbour}}^{\subset_{\mathbf{\rho},\varepsilon}%
,F}\left(  f,g\right)  & + & \mathsf{B}_{\operatorname*{commutator}}%
^{\subset_{\mathbf{\rho},\varepsilon},F}\left(  f,g\right)
\end{array}
$}.
\]

\subsection{Cube Size Splitting}

The NTV \emph{Cube Size Splitting} of the inner product $\left\langle
T_{\sigma}^{\alpha}f,g\right\rangle _{\omega}$ given in (\ref{expand}) splits
the pairs of cubes $\left(  I,J\right)  $ in a simultaneous Alpert
decomposition of $f$ and $g$ into four groups determined by relative position,
and is given by%
\begin{align}
\left\langle T_{\sigma}^{\alpha}f,g\right\rangle _{\omega}  &  =%
{\displaystyle\sum\limits_{\substack{I,J\in\mathcal{D}\\J\subset
_{\mathbf{\rho},\varepsilon}I}}}
\left\langle T_{\sigma}^{\alpha}\left(  \bigtriangleup_{I;\kappa}^{\sigma
}f\right)  ,\left(  \bigtriangleup_{J;\kappa}^{\omega}g\right)  \right\rangle
_{\omega}+%
{\displaystyle\sum\limits_{\substack{I,J\in\mathcal{D}\\J\supset
_{\mathbf{\rho},\varepsilon}I}}}
\left\langle T_{\sigma}^{\alpha}\left(  \bigtriangleup_{I;\kappa}^{\sigma
}f\right)  ,\left(  \bigtriangleup_{J;\kappa}^{\omega}g\right)  \right\rangle
_{\omega}\label{cube size}\\
&  +%
{\displaystyle\sum\limits_{\substack{I,J\in\mathcal{D}\\J\cap I=\emptyset
\text{ and }\frac{\ell\left(  J\right)  }{\ell\left(  I\right)  }\notin\left[
2^{-\mathbf{\rho}},2^{\mathbf{\rho}}\right]  }}}
\left\langle T_{\sigma}^{\alpha}\left(  \bigtriangleup_{I;\kappa}^{\sigma
}f\right)  ,\left(  \bigtriangleup_{J;\kappa}^{\omega}g\right)  \right\rangle
_{\omega}+%
{\displaystyle\sum\limits_{\substack{I,J\in\mathcal{D}\\2^{-\mathbf{\rho}}%
\leq\frac{\ell\left(  J\right)  }{\ell\left(  I\right)  }\leq2^{\mathbf{\rho}%
}}}}
\left\langle T_{\sigma}^{\alpha}\left(  \bigtriangleup_{I;\kappa}^{\sigma
}f\right)  ,\left(  \bigtriangleup_{J;\kappa}^{\omega}g\right)  \right\rangle
_{\omega}\nonumber\\
&  =\mathsf{B}_{\subset_{\mathbf{\rho},\varepsilon}}\left(  f,g\right)
+\mathsf{B}_{\supset_{\mathbf{\rho},\varepsilon}}\left(  f,g\right)
+\mathsf{B}_{\cap}\left(  f,g\right)  +\mathsf{B}_{\diagup}\left(  f,g\right)
,\nonumber
\end{align}
where the final four forms are referred to as the \emph{below}, \emph{above},
\emph{intersection} and \emph{comparable} forms respectively. The assumption
that the cubes $I$ and $J$ in the Alpert supports of $f$ and $g$ are $\left(
\mathbf{r},\varepsilon\right)  -\operatorname*{good}$ remains in force
throughout the proof. We also define the sublinear \emph{intersection} and
\emph{comparable }forms%
\begin{align*}
\left\vert \mathsf{B}_{\cap}\right\vert \left(  f,g\right)   &  \equiv%
{\displaystyle\sum\limits_{\substack{I,J\in\mathcal{D}\\J\cap I=\emptyset
\text{ and }\frac{\ell\left(  J\right)  }{\ell\left(  I\right)  }\notin\left[
2^{-\mathbf{\rho}},2^{\mathbf{\rho}}\right]  }}}
\left\vert \left\langle T_{\sigma}^{\alpha}\left(  \bigtriangleup_{I;\kappa
}^{\sigma}f\right)  ,\left(  \bigtriangleup_{J;\kappa}^{\omega}g\right)
\right\rangle _{\omega}\right\vert ,\\
\left\vert \mathsf{B}_{\diagup}\right\vert \left(  f,g\right)   &  \equiv%
{\displaystyle\sum\limits_{\substack{I,J\in\mathcal{D}\\2^{-\mathbf{\rho}}%
\leq\frac{\ell\left(  J\right)  }{\ell\left(  I\right)  }\leq2^{\mathbf{\rho}%
}}}}
\left\vert \left\langle T_{\sigma}^{\alpha}\left(  \bigtriangleup_{I;\kappa
}^{\sigma}f\right)  ,\left(  \bigtriangleup_{J;\kappa}^{\omega}g\right)
\right\rangle _{\omega}\right\vert ,
\end{align*}
in which absolute values are placed inside the sum. We have the following
bound for the sublinear \emph{intersection} and \emph{comparable} forms from
\cite[see Lemma 31]{Saw6}, which in turn followed the NTV arguments for Haar
wavelets in \cite[see the proof of Lemma 7.1]{SaShUr7} (see also
\cite{LaSaShUr3}),%
\begin{equation}
\left\vert \mathsf{B}_{\cap}\right\vert \left(  f,g\right)  +\left\vert
\mathsf{B}_{\diagup}\right\vert \left(  f,g\right)  \leq C\left(
\mathfrak{T}_{T^{\alpha}}^{\kappa_{1}}+\mathfrak{T}_{T^{\alpha,\ast}}%
^{\kappa_{2}}+\mathcal{WBP}_{T^{\alpha}}^{\left(  \kappa_{1},\kappa
_{2}\right)  }\left(  \sigma,\omega\right)  +\sqrt{A_{2}^{\alpha}}\right)
\left\Vert f\right\Vert _{L^{2}\left(  \sigma\right)  }\left\Vert g\right\Vert
_{L^{2}\left(  \omega\right)  }, \label{routine}%
\end{equation}
where if $\Omega$ is the set of all dyadic grids,
\[
\mathcal{WBP}_{T^{\alpha}}^{\left(  \kappa_{1},\kappa_{2}\right)  }\left(
\sigma,\omega\right)  \equiv\sup_{\mathcal{D}\in\Omega}\sup
_{\substack{Q,Q^{\prime}\in\mathcal{D}\\Q\subset3Q^{\prime}\setminus
Q^{\prime}\text{ or }Q^{\prime}\subset3Q\setminus Q}}\frac{1}{\sqrt{\left\vert
Q\right\vert _{\sigma}\left\vert Q^{\prime}\right\vert _{\omega}}}%
\sup_{\substack{f\in\left(  \mathcal{P}_{Q}^{\kappa_{1}}\right)
_{\operatorname*{norm}}\left(  \sigma\right)  \\g\in\left(  \mathcal{P}%
_{Q^{\prime}}^{\kappa_{2}}\right)  _{\operatorname*{norm}}\left(
\omega\right)  }}\left\vert \int_{Q^{\prime}}T_{\sigma}^{\alpha}\left(
\mathbf{1}_{Q}f\right)  \ gd\omega\right\vert <\infty
\]
is a weak boundedness constant introduced in \cite{Saw6}. This constant will
be removed in the final section below using the following bound proved in
\cite[see (6.25) in Subsection 6.7 and note that only triple testing is needed
there by choosing $\ell\left(  Q^{\prime}\right)  \leq\ell\left(  Q\right)  $
(using duality and $T^{\alpha,\ast}$ if needed)]{Saw6},
\begin{equation}
\mathcal{WBP}_{T^{\alpha}}^{\left(  \kappa_{1},\kappa_{2}\right)  }\left(
\sigma,\omega\right)  \leq C_{\kappa,\varepsilon}\left(  \mathfrak{TR}%
_{T^{\alpha}}^{\left(  \kappa_{1}\right)  }\left(  \sigma,\omega\right)
+\mathfrak{TR}_{T^{\alpha,\ast}}^{\left(  \kappa_{2}\right)  }\left(
\omega,\sigma\right)  \right)  . \label{unif bound'}%
\end{equation}
Since the \emph{below} and \emph{above} forms $\mathsf{B}_{\subset
_{\mathbf{\rho},\varepsilon}}\left(  f,g\right)  ,\mathsf{B}_{\supset
_{\mathbf{\rho},\varepsilon}}\left(  f,g\right)  $ are symmetric, matters are
reduced to proving%
\[
\left\vert \mathsf{B}_{\subset_{\mathbf{\rho},\varepsilon}}\left(  f,g\right)
\right\vert \lesssim\left(  \mathfrak{T}_{T^{\alpha}}^{\kappa_{1}%
}+\mathfrak{T}_{T^{\alpha,\ast}}^{\kappa_{2}}+\sqrt{A_{2}^{\alpha}}\right)
\left\Vert f\right\Vert _{L^{2}\left(  \sigma\right)  }\left\Vert g\right\Vert
_{L^{2}\left(  \omega\right)  }\ .
\]

\subsection{Shifted Corona Decomposition}

For this we recall the \emph{Shifted Corona Decomposition}, as opposed to the
\emph{parallel} corona decomposition used in \cite{Saw6}, associated with the
Calder\'{o}n-Zygmund $\kappa$-pivotal stopping cubes $\mathcal{F}$ introduced
above. But first we must invoke standard arguments, using the $\kappa$-cube
testing conditions (\ref{full testing}),\ to permit us to assume that $f$ and
$g$ are supported in a finite union of dyadic cubes $F_{0}$ on which they have
vanishing moments of order less than $\kappa$.

\subsubsection{The initial reduction using testing}

For this construction, we will follow the treatment as given in
\cite{SaShUr12}. We first restrict $f$ and $g$ to be supported in a large
common cube $Q_{\infty}$. Then we cover $Q_{\infty}$ with $2^{n}$ pairwise
disjoint cubes $I_{\infty}\in\mathcal{D}$ with $\ell\left(  I_{\infty}\right)
=\ell\left(  Q_{\infty}\right)  $. We now claim we can reduce matters to
consideration of the $2^{2n}$ forms%
\[
\sum_{I\in\mathcal{D}:\ I\subset I_{\infty}}\sum_{J\in\mathcal{D}:\ J\subset
J_{\infty}}\int\left(  T_{\sigma}^{\alpha}\bigtriangleup_{I;\kappa_{1}%
}^{\sigma}f\right)  \bigtriangleup_{J;\kappa_{2}}^{\omega}gd\omega,
\]
as both $I_{\infty}$ and $J_{\infty}$ range over the dyadic cubes as above.
First we note that when $I_{\infty}$ and $J_{\infty}$ are distinct, the
corresponding form is included in the sum $\mathsf{B}_{\cap}\left(
f,g\right)  +\mathsf{B}_{\diagup}\left(  f,g\right)  $, and hence controlled.
Thus it remains to consider the forms with $I_{\infty}=J_{\infty}$ and use the
cubes $I_{\infty}$ as the starting cubes in our corona construction below.
Indeed, we have from (\ref{Alpert expan}) that%
\[
f=\sum_{I\in\mathcal{D}:\ I\subset I_{\infty}}\bigtriangleup_{I;\kappa_{1}%
}^{\sigma}f+\mathbb{E}_{I_{\infty};\kappa_{1}}^{\sigma}f\ \ \ \ \text{ and
\ \ \ \ }g=\sum_{J\in\mathcal{D}:\ J\subset I_{\infty}}\bigtriangleup
_{J;\kappa_{2}}^{\omega}g+\mathbb{E}_{I_{\infty};\kappa_{2}}^{\omega}g,
\]
which can then be used to write the bilinear form $\int\left(  T_{\sigma
}f\right)  gd\omega$ as a sum of the forms%
\begin{align}
&  \ \ \ \ \ \ \ \ \ \ \ \ \ \ \ \int\left(  T_{\sigma}f\right)  gd\omega
=\sum_{I_{\infty}}\left\{  \sum_{I,J\in\mathcal{D}:\ I,J\subset I_{\infty}%
}\int\left(  T_{\sigma}^{\alpha}\bigtriangleup_{I;\kappa_{1}}^{\sigma
}f\right)  \bigtriangleup_{J;\kappa_{2}}^{\omega}gd\omega\right.
\label{sum of forms}\\
&  \left.  +\sum_{I\in\mathcal{D}:\ I\subset I_{\infty}}\int\left(  T_{\sigma
}^{\alpha}\bigtriangleup_{I;\kappa_{1}}^{\sigma}f\right)  \mathbb{E}%
_{I_{\infty};\kappa_{2}}^{\omega}gd\omega+\sum_{J\in\mathcal{D}:\ J\subset
I_{\infty}}\int\left(  T_{\sigma}^{\alpha}\mathbb{E}_{I_{\infty};\kappa_{1}%
}^{\sigma}f\right)  \bigtriangleup_{J;\kappa_{2}}^{\omega}gd\omega+\int\left(
T_{\sigma}^{\alpha}\mathbb{E}_{I_{\infty};\kappa_{1}}^{\sigma}f\right)
\mathbb{E}_{I_{\infty};\kappa_{2}}^{\omega}gd\omega\right\}  ,\nonumber
\end{align}
taken over the $2^{n}$ cubes $I_{\infty}$ above.

The second, third and fourth sums in (\ref{sum of forms}) can be controlled by
the $\kappa$-testing conditions (\ref{full testing}), e.g. using
Cauchy-Schwarz,%
\begin{align}
&  \left\vert \sum_{I\in\mathcal{D}:\ I\subset I_{\infty}}\int\left(
T_{\sigma}^{\alpha}\bigtriangleup_{I;\kappa_{1}}^{\sigma}f\right)
\mathbb{E}_{I_{\infty};\kappa_{2}}^{\omega}gd\omega\right\vert \leq\left\Vert
\sum_{I\in\mathcal{D}:\ I\subset I_{\infty}}\bigtriangleup_{I;\kappa_{1}%
}^{\sigma}f\right\Vert _{L^{2}\left(  \sigma\right)  }\left\Vert
\mathbb{E}_{I_{\infty};\kappa_{2}}^{\omega}g\right\Vert _{L^{\infty}%
}\left\Vert \mathbf{1}_{I_{\infty}}T_{\omega}^{\alpha,\ast}\left(
\frac{\mathbb{E}_{I_{\infty};\kappa_{2}}^{\omega}g}{\left\Vert \mathbb{E}%
_{I_{\infty};\kappa_{2}}^{\omega}g\right\Vert _{L^{\infty}}}\right)
\right\Vert _{L^{2}\left(  \sigma\right)  }\label{top control}\\
&  \ \ \ \ \ \ \ \ \ \ \ \ \ \ \ \lesssim\left\Vert f\right\Vert
_{L^{2}\left(  \sigma\right)  }\frac{\left\vert \widehat{g}\left(  I_{\infty
}\right)  \right\vert }{\sqrt{\left\vert I_{\infty}\right\vert _{\omega}}%
}\mathfrak{T}_{T_{\omega}^{\alpha,\ast}}^{\kappa_{2}}\sqrt{\left\vert
I_{\infty}\right\vert _{\omega}}\leq\mathfrak{T}_{T_{\omega}^{\alpha,\ast}%
}^{\kappa_{2}}\left\Vert f\right\Vert _{L^{2}\left(  \sigma\right)
}\left\Vert g\right\Vert _{L^{2}\left(  \omega\right)  }\ ,\nonumber
\end{align}
and similarly for the third and fourth sum.

\subsubsection{The shifted corona}

Recall the shifted corona $\mathcal{C}_{\mathcal{F}}^{\mathbf{\tau
}-\operatorname*{shift}}\left(  F\right)  $ defined in (\ref{def shift}). A
simple but important property is the fact that the $\mathbf{\tau}$-shifted
coronas $\mathcal{C}_{F}^{\mathbf{\tau}-\operatorname*{shift}}$ have overlap
bounded by $\mathbf{\tau}$:%
\begin{equation}
\sum_{F\in\mathcal{F}}\mathbf{1}_{\mathcal{C}_{F}^{\mathbf{\tau}%
-\operatorname*{shift}}}\left(  J\right)  \leq\mathbf{\tau},\ \ \ \ \ J\in
\mathcal{D}. \label{tau overlap}%
\end{equation}
It is convenient, for use in the canonical splitting below, to introduce the
following shorthand notation for $F,G\in\mathcal{F}$:%
\[
\left\langle T_{\sigma}^{\alpha}\left(  \mathsf{P}_{\mathcal{C}_{F}}^{\sigma
}f\right)  ,\mathsf{P}_{\mathcal{C}_{G}^{\mathbf{\tau}-\operatorname*{shift}}%
}^{\omega}g\right\rangle _{\omega}^{\subset_{\mathbf{\rho},\varepsilon}}%
\equiv\sum_{\substack{I\in\mathcal{C}_{F}\text{ and }J\in\mathcal{C}%
_{G}^{\mathbf{\tau}-\operatorname*{shift}}\\J\subset_{\mathbf{\rho
},\varepsilon}I}}\left\langle T_{\sigma}^{\alpha}\left(  \bigtriangleup
_{I;\kappa_{1}}^{\sigma}f\right)  ,\left(  \bigtriangleup_{J;\kappa_{2}%
}^{\omega}g\right)  \right\rangle _{\omega}\ .
\]

\subsection{Canonical Splitting}

We then proceed with the \emph{Canonical Splitting} of $\mathsf{B}%
_{\subset_{\mathbf{\rho},\varepsilon}}\left(  f,g\right)  $ in
(\ref{cube size}) as in \cite{SaShUr7}, but with Alpert wavelets in place of
Haar wavelets,%
\begin{align*}
\mathsf{B}_{\subset_{\mathbf{\rho},\varepsilon}}\left(  f,g\right)   &
=\sum_{F\in\mathcal{F}}\left\langle T_{\sigma}\left(  \mathsf{P}%
_{\mathcal{C}_{F}}^{\sigma}f\right)  ,\mathsf{P}_{\mathcal{C}_{F}%
^{\mathbf{\tau}-\operatorname*{shift}}}^{\omega}g\right\rangle _{\omega
}^{\subset_{\mathbf{\rho},\varepsilon}}+\sum_{\substack{F,G\in\mathcal{F}%
\\G\subsetneqq F}}\left\langle T_{\sigma}\left(  \mathsf{P}_{\mathcal{C}_{F}%
}^{\sigma}f\right)  ,\mathsf{P}_{\mathcal{C}_{G}^{\mathbf{\tau}%
-\operatorname*{shift}}}^{\omega}g\right\rangle _{\omega}^{\subset
_{\mathbf{\rho},\varepsilon}}\\
&  +\sum_{\substack{F,G\in\mathcal{F}\\G\supsetneqq F}}\left\langle T_{\sigma
}\left(  \mathsf{P}_{\mathcal{C}_{F}}^{\sigma}f\right)  ,\mathsf{P}%
_{\mathcal{C}_{G}^{\mathbf{\tau}-\operatorname*{shift}}}^{\omega
}g\right\rangle _{\omega}^{\subset_{\mathbf{\rho},\varepsilon}}+\sum
_{\substack{F,G\in\mathcal{F}\\F\cap G=\emptyset}}\left\langle T_{\sigma
}\left(  \mathsf{P}_{\mathcal{C}_{F}}^{\sigma}f\right)  ,\mathsf{P}%
_{\mathcal{C}_{G}^{\mathbf{\tau}-\operatorname*{shift}}}^{\omega
}g\right\rangle _{\omega}^{\subset_{\mathbf{\rho},\varepsilon}}\\
&  \equiv\mathsf{T}_{\operatorname*{diagonal}}^{\subset_{\mathbf{\rho
},\varepsilon}}\left(  f,g\right)  +\mathsf{T}_{\operatorname*{far}%
\operatorname*{below}}^{\subset_{\mathbf{\rho},\varepsilon}}\left(
f,g\right)  +\mathsf{T}_{\operatorname*{far}\operatorname*{above}}%
^{\subset_{\mathbf{\rho},\varepsilon}}\left(  f,g\right)  +\mathsf{T}%
_{\operatorname*{disjoint}}^{\subset_{\mathbf{\rho},\varepsilon}}\left(
f,g\right)  .
\end{align*}
The final two forms $\mathsf{T}_{\operatorname*{far}\operatorname*{above}%
}^{\subset_{\mathbf{\rho},\varepsilon}}\left(  f,g\right)  $ and
$\mathsf{T}_{\operatorname*{disjoint}}^{\subset_{\mathbf{\rho},\varepsilon}%
}\left(  f,g\right)  $ each vanish just as in \cite{SaShUr7}, since there are
no pairs $\left(  I,J\right)  \in\mathcal{C}_{F}\times\mathcal{C}%
_{G}^{\mathbf{\tau}-\operatorname*{shift}}$ with both (\textbf{i})
$J\subset_{\mathbf{\rho},\varepsilon}I$ and (\textbf{ii}) either $F\subsetneqq
G$ or $G\cap F=\emptyset$. The \emph{below far below} form $\mathsf{T}%
_{\operatorname*{far}\operatorname*{below}}^{\subset_{\mathbf{\rho
},\varepsilon}}\left(  f,g\right)  $ is then further split into two forms
$\mathsf{T}_{\operatorname*{far}\operatorname*{below}}^{1}\left(  f,g\right)
$ and $\mathsf{T}_{\operatorname*{far}\operatorname*{below}}^{2}\left(
f,g\right)  $ as in \cite{SaShUr7},%
\begin{align*}
&  \ \ \ \ \ \ \ \ \ \ \mathsf{T}_{\operatorname*{far}\operatorname*{below}%
}^{\subset_{\mathbf{\rho},\varepsilon}}\left(  f,g\right)  =\sum
_{G\in\mathcal{F}}\sum_{F\in\mathcal{F}:\ G\subsetneqq F}\sum_{\substack{I\in
\mathcal{C}_{F}\text{ and }J\in\mathcal{C}_{G}^{\mathbf{\tau}%
-\operatorname*{shift}}\\J\subset_{\mathbf{\rho},\varepsilon}I}}\left\langle
T_{\sigma}^{\alpha}\left(  \bigtriangleup_{I;\kappa_{1}}^{\sigma}f\right)
,\left(  \bigtriangleup_{J;\kappa_{2}}^{\omega}g\right)  \right\rangle
_{\omega}\\
&  =\sum_{G\in\mathcal{F}}\sum_{F\in\mathcal{F}:\ G\subsetneqq F}\sum
_{J\in\mathcal{C}_{G}^{\mathbf{\tau}-\operatorname*{shift}}}\sum
_{I\in\mathcal{C}_{F}\text{ and }J\subset I}\left\langle T_{\sigma}^{\alpha
}\left(  \bigtriangleup_{I;\kappa_{1}}^{\sigma}f\right)  ,\left(
\bigtriangleup_{J;\kappa_{2}}^{\omega}g\right)  \right\rangle _{\omega}\\
&  -\sum_{F\in\mathcal{F}}\sum_{G\in\mathcal{F}:\ G\subsetneqq F}\sum
_{J\in\mathcal{C}_{G}^{\mathbf{\tau}-\operatorname*{shift}}}\sum
_{I\in\mathcal{C}_{F}\text{ and }J\subset I\text{ but }J\not \subset
_{\mathbf{\rho},\varepsilon}I}\left\langle T_{\sigma}^{\alpha}\left(
\bigtriangleup_{I;\kappa_{1}}^{\sigma}f\right)  ,\left(  \bigtriangleup
_{J;\kappa_{2}}^{\omega}g\right)  \right\rangle _{\omega}\equiv\mathsf{T}%
_{\operatorname*{far}\operatorname*{below}}^{1}\left(  f,g\right)
-\mathsf{T}_{\operatorname*{far}\operatorname*{below}}^{2}\left(  f,g\right)
.
\end{align*}

\begin{remark}
For the remainder of the proof of Theorem \ref{main}, one should keep in mind
that if $T^{\alpha}$ is a Stein elliptic Calder\'{o}n-Zygmund operator on
$\mathbb{R}^{n}$, then $A_{2}^{\alpha}\left(  \sigma,\omega\right)
\lesssim\operatorname*{weak}\mathfrak{N}_{T^{\alpha}}\left(  \sigma
,\omega\right)  $.
\end{remark}

The second form $\mathsf{T}_{\operatorname*{far}\operatorname*{below}}%
^{2}\left(  f,g\right)  $ is easily seen to satisfy%
\begin{equation}
\left\vert \mathsf{T}_{\operatorname*{far}\operatorname*{below}}^{2}\left(
f,g\right)  \right\vert \lesssim C\left(  \mathfrak{T}_{T^{\alpha}}%
^{\kappa_{1}}+\mathfrak{T}_{T^{\alpha,\ast}}^{\kappa_{2}}+\mathcal{WBP}%
_{T^{\alpha}}^{\left(  \kappa_{1},\kappa_{2}\right)  }\left(  \sigma
,\omega\right)  +\sqrt{A_{2}^{\alpha}}\right)  \left\Vert f\right\Vert
_{L^{2}\left(  \sigma\right)  }\left\Vert g\right\Vert _{L^{2}\left(
\omega\right)  }, \label{second far below}%
\end{equation}
just as for the analogous inequality in \cite{SaShUr7} for Haar wavelets, by
(\ref{routine}). To control the first and main form $\mathsf{T}%
_{\operatorname*{far}\operatorname*{below}}^{1}\left(  f,g\right)  $, we use
the $\kappa$-pivotal Intertwining Proposition \ref{strongly adapted} recalled
from \cite{Saw6} in the earlier section on preliminaries. This proposition
then immediately gives the bound%
\begin{equation}
\left\vert \mathsf{T}_{\operatorname*{far}\operatorname*{below}}^{1}\left(
f,g\right)  \right\vert \lesssim\left(  \mathfrak{T}_{T^{\alpha}}^{\kappa_{1}%
}+\sqrt{A_{2}^{\alpha}}\right)  \left\Vert f\right\Vert _{L^{2}\left(
\sigma\right)  }\left\Vert g\right\Vert _{L^{2}\left(  \omega\right)  }.
\label{first far below}%
\end{equation}

To handle the \emph{below diagonal} form $\mathsf{T}_{\operatorname*{diagonal}%
}^{\subset_{\mathbf{\rho},\varepsilon}}\left(  f,g\right)  $, we decompose
according to the stopping times $\mathcal{F}$,%
\begin{equation}
\mathsf{T}_{\operatorname*{diagonal}}^{\subset_{\mathbf{\rho},\varepsilon}%
}\left(  f,g\right)  =\sum_{F\in\mathcal{F}}\mathsf{T}%
_{\operatorname*{diagonal}}^{\subset_{\mathbf{\rho},\varepsilon},F}\left(
f,g\right)  ,\text{ where }\mathsf{T}_{\operatorname*{diagonal}}%
^{\subset_{\mathbf{\rho},\varepsilon},F}\left(  f,g\right)  \equiv\left\langle
T_{\sigma}^{\alpha}\left(  \mathsf{P}_{\mathcal{C}_{F}}^{\sigma}f\right)
,\mathsf{P}_{\mathcal{C}_{F}^{\mathbf{\tau}-\operatorname*{shift}}}^{\omega
}g\right\rangle _{\omega}^{\subset_{\mathbf{\rho},\varepsilon}},
\label{def block}%
\end{equation}
and it is enough, using Cauchy-Schwarz and quasiorthogonality
(\ref{Car and quasi}) in $f$, together with orthogonality in both $f$ and $g$,
to prove the following bound involving the usual cube testing constant,%
\begin{equation}
\left\vert \mathsf{T}_{\operatorname*{diagonal}}^{\subset_{\mathbf{\rho
},\varepsilon},F}\left(  f,g\right)  \right\vert \lesssim\left(
\mathfrak{T}_{T^{\alpha}}^{\kappa_{1}}+\sqrt{A_{2}^{\alpha}}\right)  \ \left(
\alpha_{\mathcal{F}}\left(  F\right)  \sqrt{\left\vert F\right\vert _{\sigma}%
}+\left\Vert \mathsf{P}_{\mathcal{C}_{F}}^{\sigma}f\right\Vert _{L^{2}\left(
\sigma\right)  }\right)  \ \left\Vert \mathsf{P}_{\mathcal{C}_{F}%
^{\mathbf{\tau}-\operatorname*{shift}}}^{\omega}g\right\Vert _{L^{2}\left(
\omega\right)  }\ . \label{below form bound}%
\end{equation}
Indeed, this then gives the estimate,%
\begin{equation}
\left\vert \mathsf{T}_{\operatorname*{diagonal}}^{\subset_{\mathbf{\rho
},\varepsilon}}\left(  f,g\right)  \right\vert \lesssim\left(  \mathfrak{T}%
_{T^{\alpha}}^{\kappa_{1}}+\sqrt{A_{2}^{\alpha}}\right)  \left\Vert
f\right\Vert _{L^{2}\left(  \sigma\right)  }\left\Vert g\right\Vert
_{L^{2}\left(  \omega\right)  }. \label{diag est}%
\end{equation}
It is important for this below estimate that we choose $\kappa_{2}\geq
2\kappa_{1}$ and $\kappa_{1}$ sufficiently large.

On the other hand, when we turn to bounding the above diagonal form
$\mathsf{T}_{\operatorname*{diagonal}}^{\supset_{\mathbf{\rho},\varepsilon}%
}\left(  f,g\right)  $, we will \textbf{not} have $\kappa_{1}\geq2\kappa_{2}$,
and we will have to argue differently in order to use the dual indicator /
cube testing constant $\mathfrak{T}_{T^{\alpha}}^{\operatorname*{ind},\ast
}\left(  \omega,\sigma\right)  $,%
\begin{equation}
\left\vert \mathsf{T}_{\operatorname*{diagonal}}^{\supset_{\mathbf{\rho
},\varepsilon}}\left(  f,g\right)  \right\vert \lesssim\left(  \mathfrak{T}%
_{T^{\alpha}}^{\operatorname*{ind},\ast}\left(  \omega,\sigma\right)
+\sqrt{A_{2}^{\alpha}}\right)  \ \left(  \alpha_{\mathcal{F}}\left(  F\right)
\sqrt{\left\vert F\right\vert _{\sigma}}+\left\Vert \mathsf{P}_{\mathcal{C}%
_{F}}^{\sigma}f\right\Vert _{L^{2}\left(  \sigma\right)  }\right)
\ \left\Vert \mathsf{P}_{\mathcal{C}_{F}^{\mathbf{\tau}-\operatorname*{shift}%
}}^{\omega}g\right\Vert _{L^{2}\left(  \omega\right)  }. \label{diag est dual}%
\end{equation}

Thus at this point we have reduced the proof of Theorem \ref{main} to

\begin{enumerate}
\item proving (\ref{below form bound}),

\item proving (\ref{diag est dual}),

\item and controlling the triple polynomial testing condition
(\ref{full testing}) by the usual cube testing condition and the classical
Muckenhoupt condition (\ref{Muck and test}).
\end{enumerate}

In the next section we address the first issue by proving the inequality
(\ref{below form bound}) for the \emph{below diagonal} forms $\mathsf{T}%
_{\operatorname*{diagonal}}^{\subset_{\mathbf{\rho},\varepsilon},F}\left(
f,g\right)  $, and in the subsequent section we prove (\ref{diag est dual})
for the \emph{above diagonal} form $\mathsf{T}_{\operatorname*{diagonal}%
}^{\supset_{\mathbf{\rho},\varepsilon},F}\left(  f,g\right)  $. In the final
section, we address the second issue and complete the proof of Theorem
\ref{main} by drawing together all of the estimates.

\section{Below diagonal form and the NTV reach for Alpert wavelets}

In this section we give the main new argument of this paper. It will be
convenient to denote our fractional singular integral operators by
$T^{\lambda}$, $0\leq\lambda<n$, instead of $T^{\alpha}$, as $\alpha$ will
denote a multi-index in $\mathbb{Z}_{+}^{n}$. But first, we note that for a
doubling measure $\mu$, a cube $I$ and a polynomial $P$, we have $\left\Vert
P\mathbf{1}_{I}\right\Vert _{L^{\infty}\left(  \mu\right)  }=\sup_{x\in
I}\left\vert P\left(  x\right)  \right\vert $; in particular, $\left\Vert
P\mathbf{1}_{I}\right\Vert _{L^{\infty}\left(  \sigma\right)  }=\left\Vert
P\mathbf{1}_{I}\right\Vert _{L^{\infty}\left(  \omega\right)  }=\left\Vert
P\mathbf{1}_{I}\right\Vert _{L^{\infty}}$.

We will adapt the classical reach of NTV using Haar wavelet projections
$\bigtriangleup_{I}^{\sigma}$, namely the ingenious `thinking outside the box'
idea of the paraproduct / stopping / neighbour decomposition of Nazarov, Treil
and Volberg \cite{NTV4}. Since we are using weighted Alpert wavelet
projections $\bigtriangleup_{I;\kappa}^{\sigma}$ instead, the projection
$\mathbb{E}_{I^{\prime};\kappa}^{\sigma}\bigtriangleup_{I;\kappa}^{\sigma}f$
onto the child $I^{\prime}\in\mathfrak{C}_{\mathcal{D}}\left(  I\right)  $
equals $M_{I^{\prime};\kappa}\mathbf{1}_{I_{\pm}}$ where $M=M_{I^{\prime
};\kappa}$ is a polynomial of degree less than $\kappa$ restricted to
$I^{\prime}$, as opposed to a constant in the Haar case, and hence no longer
commutes in general with the operator $T_{\sigma}^{\lambda}$. As mentioned in
the introduction, this results in a new commutator form to be bounded, and
complicates bounding the remaining forms as well.

We will treat the \emph{below diagonal} forms $\mathsf{T}%
_{\operatorname*{diagonal}}^{\subset_{\mathbf{\rho},\varepsilon},F}\left(
f,g\right)  $ for $F\in\mathcal{F}$ in detail in this section, and turn to the
analogous \emph{above diagonal} forms $\mathsf{T}_{\operatorname*{diagonal}%
}^{\supset_{\mathbf{\rho},\varepsilon},G}\left(  f,g\right)  $ for
$G\in\mathcal{G}$ in the next section. We have from (\ref{def block}), that
$\mathsf{T}_{\operatorname*{diagonal}}^{\subset_{\mathbf{\rho},\varepsilon}%
,F}\left(  f,g\right)  $ equals%
\begin{align*}
&  \sum_{\substack{I\in\mathcal{C}_{F}\text{ and }J\in\mathcal{C}%
_{F}^{\mathbf{\tau}-\operatorname*{shift}}\\J\subset_{\mathbf{\rho
},\varepsilon}I}}\left\langle T_{\sigma}^{\lambda}\left(  \mathbf{1}_{I_{J}%
}\bigtriangleup_{I;\kappa_{1}}^{\sigma}f\right)  ,\bigtriangleup_{J;\kappa
_{2}}^{\omega}g\right\rangle _{\omega}+\sum_{\substack{I\in\mathcal{C}%
_{F}\text{ and }J\in\mathcal{C}_{F}^{\mathbf{\tau}-\operatorname*{shift}%
}\\J\subset_{\mathbf{\rho},\varepsilon}I}}\sum_{\theta\left(  I_{J}\right)
\in\mathfrak{C}_{\mathcal{D}}\left(  I\right)  \setminus\left\{
I_{J}\right\}  }\left\langle T_{\sigma}^{\lambda}\left(  \mathbf{1}%
_{\theta\left(  I_{J}\right)  }\bigtriangleup_{I;\kappa_{1}}^{\sigma}f\right)
,\bigtriangleup_{J;\kappa_{2}}^{\omega}g\right\rangle _{\omega}\\
&  \equiv\mathsf{T}_{\operatorname{home}}^{\subset_{\mathbf{\rho},\varepsilon
},F}\left(  f,g\right)  +\mathsf{T}_{\operatorname*{neighbour}}^{\subset
_{\mathbf{\rho},\varepsilon},F}\left(  f,g\right)  ,
\end{align*}
where we write $\mathbf{\kappa}=\left(  \kappa_{1},\kappa_{2}\right)  $, and
we further decompose the \emph{below home} form using
\begin{equation}
M_{I^{\prime}}=M_{I^{\prime};\kappa_{1}}\equiv\mathbf{1}_{I^{\prime}%
}\bigtriangleup_{I;\kappa_{1}}^{\sigma}f=\mathbb{E}_{I^{\prime};\kappa_{1}%
}^{\sigma}\bigtriangleup_{I;\kappa_{1}}^{\sigma}f=\mathbb{E}_{I^{\prime
};\kappa_{1}}^{\sigma}\bigtriangleup_{I;\kappa_{1}}^{\sigma}\mathsf{P}%
_{\mathcal{C}_{F}}f, \label{def M}%
\end{equation}
where $\mathsf{P}_{\mathcal{C}_{F}}f\equiv\sum_{I\in\mathcal{C}_{F}%
}\bigtriangleup_{I;\kappa_{1}}^{\sigma}f$, to obtain%
\begin{align*}
&  \mathsf{T}_{\operatorname{home}}^{\subset_{\mathbf{\rho},\varepsilon}%
,F}\left(  f,g\right)  =\sum_{\substack{I\in\mathcal{C}_{F}\text{ and }%
J\in\mathcal{C}_{F}^{\mathbf{\tau}-\operatorname*{shift}}\\J\subset
_{\mathbf{\rho},\varepsilon}I}}\left\langle M_{I_{J}}T_{\sigma}^{\lambda
}\mathbf{1}_{F},\bigtriangleup_{J;\kappa_{2}}^{\omega}g\right\rangle _{\omega
}-\sum_{\substack{I\in\mathcal{C}_{F}\text{ and }J\in\mathcal{C}%
_{F}^{\mathbf{\tau}-\operatorname*{shift}}\\J\subset_{\mathbf{\rho
},\varepsilon}I}}\left\langle M_{I_{J}}T_{\sigma}^{\lambda}\mathbf{1}%
_{F\setminus I_{J}},\bigtriangleup_{J;\kappa_{2}}^{\omega}g\right\rangle
_{\omega}\\
&  +\sum_{\substack{I\in\mathcal{C}_{F}\text{ and }J\in\mathcal{C}%
_{F}^{\mathbf{\tau}-\operatorname*{shift}}\\J\subset_{\mathbf{\rho
},\varepsilon}I}}\left\langle \left[  T_{\sigma}^{\lambda},M_{I_{J}}\right]
\mathbf{1}_{I_{J}},\bigtriangleup_{J;\kappa_{2}}^{\omega}g\right\rangle
_{\omega}\equiv\mathsf{T}_{\operatorname*{paraproduct}}^{\subset
_{\mathbf{\rho},\varepsilon},F}\left(  f,g\right)  +\mathsf{T}%
_{\operatorname*{stop}}^{\subset_{\mathbf{\rho},\varepsilon},F}\left(
f,g\right)  +\mathsf{T}_{\operatorname*{commutator}}^{\subset_{\mathbf{\rho
},\varepsilon},F}\left(  f,g\right)  .
\end{align*}
Altogether then we have the weighted Alpert version of the NTV paraproduct
decomposition\footnote{In \cite[see the end of Section 10 on Concluding
Remarks]{Saw6} it was remarked that one cannot extend a nonconstant
polynomial, normalized to a cube $Q$, to a supercube $F$ without destroying
the normalization in general. This obstacle to the paraproduct decomposition
of NTV is overcome here by controlling the commutator form.},%
\[
\mathsf{T}_{\operatorname*{diagonal}}^{\subset_{\mathbf{\rho},\varepsilon}%
,F}\left(  f,g\right)  =\mathsf{T}_{\operatorname*{paraproduct}}%
^{\subset_{\mathbf{\rho},\varepsilon},F}\left(  f,g\right)  +\mathsf{T}%
_{\operatorname*{stop}}^{\subset_{\mathbf{\rho},\varepsilon},F}\left(
f,g\right)  +\mathsf{T}_{\operatorname*{commutator}}^{\subset_{\mathbf{\rho
},\varepsilon},F}\left(  f,g\right)  +\mathsf{T}_{\operatorname*{neighbour}%
}^{\subset_{\mathbf{\rho},\varepsilon},F}\left(  f,g\right)  .
\]

In fact, we will see that all forms above, except for the paraproduct, are
absolutely convergent with respect to the double sum over the cubes $I$, $J$.

\subsection{The below paraproduct form}

First pigeonhole the sum over pairs $I$ and $J$ according to which child
$I^{\prime}\in\mathfrak{C}_{\mathcal{D}}\left(  I\right)  $ contains $J$ to
get
\[
\mathsf{T}_{\operatorname*{paraproduct}}^{\subset_{\mathbf{\rho},\varepsilon
},F}\left(  f,g\right)  =\sum_{I\in\mathcal{C}_{F}}\sum_{I^{\prime}%
\in\mathfrak{C}_{\mathcal{D}}\left(  I\right)  }\sum_{\substack{J\in
\mathcal{C}_{F}^{\mathbf{\tau}-\operatorname*{shift}}:\ J\subset
_{\mathbf{\rho},\varepsilon}I\\J\subset I^{\prime}}}\left\langle M_{I^{\prime
};\kappa_{1}}T_{\sigma}^{\lambda}\mathbf{1}_{F},\bigtriangleup_{J;\kappa_{2}%
}^{\omega}g\right\rangle _{\omega}.
\]
This form $\mathsf{T}_{\operatorname*{paraproduct}}^{\subset_{\mathbf{\rho
},\varepsilon},F}\left(  f,g\right)  $ can be handled as usual, using the
telescoping property (\ref{telescoping}) to sum the restrictions to a cube
$J\in\mathcal{C}_{F}^{\mathbf{\tau}-\operatorname*{shift}}$ of the polynomials
$M_{I^{\prime};\kappa_{1}}$ over the relevant cubes $I$, to obtain a
restricted polynomial $\mathbf{1}_{J}P_{I^{\prime};\kappa_{1}}$ that is
controlled by $\alpha_{\mathcal{F}}\left(  F\right)  $, and then passing the
polynomial $M_{J;\kappa_{1}}$ over to $\bigtriangleup_{J;\kappa_{2}}^{\omega
}g$. More precisely, for each $J\in\mathcal{C}_{F}^{\mathbf{\tau
}-\operatorname*{shift}}$, let $I_{J}^{\natural}$ denote the smallest
$K\in\mathcal{C}_{F}$ such that $J\subset_{\mathbf{\rho},\varepsilon}K$ (and
as a consequence $J\subset_{\mathbf{\rho},\varepsilon}I$ for all $I\supset
I_{J}^{\natural}$), and let $I_{J}^{\flat}$ denote the $\mathcal{D}$-child of
$I_{J}^{\natural}$ that contains $J$. Then we further consider the two
possibilities where $I_{J}^{\flat}$ is in $\mathcal{C}_{F}$ or not. We have
\begin{equation}
\label{eq:telescope_M_to_poly}\sum_{I\in\mathcal{C}_{F}:\ I_{J}^{\natural
}\subset I}\mathbf{1}_{J}M_{I^{\prime};\kappa_{1}}=\mathbf{1}_{J}\sum
_{I\in\mathcal{C}_{F}:\ I_{J}^{\natural}\subset I}M_{I^{\prime};\kappa_{1}%
}=\mathbf{1}_{J}\left(  \mathbb{E}_{I_{J}^{\flat};\kappa_{1}}^{\sigma
}f-\mathbb{E}_{F;\kappa_{1}}^{\sigma}f\right)  \equiv\mathbf{1}_{J}%
P_{J;\kappa_{1}}%
\end{equation}
and we set $P_{J;\kappa_{1}}\equiv0$ if $I_{J}^{\flat}\not \in \mathcal{C}%
_{F}$ and $I_{J}^{\flat}\in\mathcal{C}_{F}$. We now claim that
\begin{equation}
\left\Vert \mathbf{1}_{J}P_{J;\kappa_{1}}\right\Vert _{L^{\infty}\left(
\sigma\right)  }\leq\left\Vert \mathbb{E}_{I_{J}^{\flat};\kappa_{1}}^{\sigma
}f\right\Vert _{L^{\infty}\left(  \sigma\right)  }+\left\Vert \mathbb{E}%
_{F;\kappa_{1}}^{\sigma}f\right\Vert _{L^{\infty}\left(  \sigma\right)  }
\lesssim\alpha_{\mathcal{F}}\left(  F\right)  \, . \label{unif bdd}%
\end{equation}
Indeed, we note that
\[
\left\Vert \mathbb{E}_{F;\kappa_{1}}^{\sigma}f\right\Vert _{L^{\infty}\left(
\sigma\right)  } \lesssim E_{F} f \lesssim\alpha_{\mathcal{F}} (F) \, .
\]
And as for $\left\Vert \mathbb{E}_{I_{J}^{\flat};\kappa_{1}}^{\sigma
}f\right\Vert _{L^{\infty}\left(  \sigma\right)  }$, there are two cases: if
$I_{J}^{\flat} \in C_{F}$, then
\[
\left\Vert \mathbb{E}_{I_{J}^{\flat};\kappa_{1}}^{\sigma}f\right\Vert
_{L^{\infty}\left(  \sigma\right)  } \lesssim E_{F} f \lesssim\alpha
_{\mathcal{F}} (F)
\]
by (\ref{analogue}) and the definition of the stopping time, and if
$I_{J}^{\flat}\in\mathcal{F}$, then because $\sigma$ is doubling and
$\pi_{\mathcal{D}}F^{\prime}\in\mathcal{C}_{F}$, we get
\[
\left\Vert \mathbb{E}_{I_{J}^{\flat};\kappa_{1}}^{\sigma}f\right\Vert
_{L^{\infty}\left(  \sigma\right)  } \lesssim E_{F^{\prime}}^{\sigma
}\left\vert f\right\vert \lesssim\alpha_{\mathcal{F}}\left(  F\right)  \, .
\]
Thus
\begin{align*}
&  \left\vert \mathsf{T}_{\operatorname*{paraproduct}}^{\subset_{\mathbf{\rho
},\varepsilon},F }\left(  f,g\right)  \right\vert =\left\vert \sum
_{J\in\mathcal{C}_{F}^{\mathbf{\tau}-\operatorname*{shift}}}\left\langle
\mathbf{1}_{J}\left(  \sum_{\substack{I\in\mathcal{C}_{F}:\ \\J\subset
_{\mathbf{\rho},\varepsilon}I}}\sum_{\substack{I^{\prime}\in\mathfrak{C}%
_{\mathcal{D}}\left(  I\right)  \\J\subset I^{\prime}\text{ }}}M_{I^{\prime
};\kappa_{1}}\right)  T_{\sigma}^{\lambda}\mathbf{1}_{F},\bigtriangleup
_{J;\kappa}^{\omega}g\right\rangle _{\omega}\right\vert \\
&  =\left\vert \sum_{J\in\mathcal{C}_{F}^{\mathbf{\tau}-\operatorname*{shift}%
}}\left\langle \mathbf{1}_{J}P_{J;\kappa_{1}}T_{\sigma}^{\lambda}%
\mathbf{1}_{F},\bigtriangleup_{J;\kappa}^{\omega}g\right\rangle _{\omega
}\right\vert =\left\vert \left\langle T_{\sigma}^{\lambda}\mathbf{1}_{F}%
,\sum_{J\in\mathcal{C}_{F}^{\mathbf{\tau}-\operatorname*{shift}}}%
P_{J;\kappa_{1}} \bigtriangleup_{J;\kappa_{2}}^{\omega}g\right\rangle
_{\omega}\right\vert \\
&  \leq\left\Vert T_{\sigma}^{\lambda}\mathbf{1}_{F}\right\Vert _{L^{2}\left(
\omega\right)  }\alpha_{\mathcal{F}}\left(  F\right)  \left\Vert \sum
_{J\in\mathcal{C}_{F}^{\mathbf{\tau}-\operatorname*{shift}}}\frac
{P_{J;\kappa_{1}}}{\alpha_{\mathcal{F}}\left(  F\right)  }\bigtriangleup
_{J;\kappa_{2}}^{\omega}g\right\Vert _{L^{2}\left(  \omega\right)  }%
\leq\mathfrak{T}_{T^{\lambda}}\sqrt{\left\vert F\right\vert _{\sigma}}%
\alpha_{\mathcal{F}}\left(  F\right)  \left\Vert \sum_{J\in\mathcal{C}%
_{F}^{\mathbf{\tau}-\operatorname*{shift}}}\frac{P_{J;\kappa_{1}}}%
{\alpha_{\mathcal{F}}\left(  F\right)  }\bigtriangleup_{J;\kappa_{2}}^{\omega
}g\right\Vert _{L^{2}\left(  \omega\right)  }.
\end{align*}
Now we will use an almost orthogonality argument that exploits the fact that
for $J^{\prime}$ small compared to $J$, and $\kappa_{1}\leq\kappa_{2}$, the
function $M_{J^{\prime};\kappa_{1}}\bigtriangleup_{J^{\prime};\kappa_{2}%
}^{\omega}g$ has vanishing $\omega$-means up to order $\kappa_{2}-\kappa
_{1}+1$, and the polynomial $\mathbf{1}_{J}P_{J;\kappa_{1}}\bigtriangleup
_{J;\kappa_{2}}^{\omega}g$ is relatively smooth at the scale of $J^{\prime}$,
together with the fact that the polynomials $R_{J;\kappa_{1}}\equiv
\frac{P_{J;\kappa_{1}}}{\alpha_{\mathcal{F}}\left(  F\right)  }$ of degree at
most $\kappa_{1}-1$, have $L^{\infty}$ norm uniformly bounded by
(\ref{unif bdd}), to show that
\begin{align}
\left\Vert \sum_{J\in\mathcal{C}_{F}^{\mathbf{\tau}-\operatorname*{shift}}%
}R_{J;\kappa_{1}}\bigtriangleup_{J;\kappa_{2}}^{\omega}g\right\Vert
_{L^{2}\left(  \omega\right)  }^{2}  &  =\sum_{J\in\mathcal{C}_{F}%
^{\mathbf{\tau}-\operatorname*{shift}}}\left\Vert R_{J;\kappa_{1}%
}\bigtriangleup_{J;\kappa_{2}}^{\omega}g\right\Vert _{L^{2}\left(
\omega\right)  }^{2}+\sum_{\substack{J,J^{\prime}\in\mathcal{C}_{F}%
^{\mathbf{\tau}-\operatorname*{shift}}\\J\neq J^{\prime}}}\int\left(
R_{J;\kappa_{1}}\bigtriangleup_{J;\kappa_{2}}^{\omega}g\right)  \left(
R_{J^{\prime};\kappa_{1}}\bigtriangleup_{J^{\prime};\kappa_{2}}^{\omega
}g\right)  d\omega\label{to show that}\\
&  \lesssim\sum_{J\in\mathcal{C}_{F}^{\mathbf{\tau}-\operatorname*{shift}}%
}\left\Vert R_{J;\kappa_{1}}\bigtriangleup_{J;\kappa_{2}}^{\omega}g\right\Vert
_{L^{2}\left(  \omega\right)  }^{2}\lesssim\sum_{J\in\mathcal{C}%
_{F}^{\mathbf{\tau}-\operatorname*{shift}}}\left\Vert \bigtriangleup
_{J;\kappa_{2}}^{\omega}g\right\Vert _{L^{2}\left(  \omega\right)  }%
^{2}=\left\Vert \mathsf{P}_{\mathcal{C}_{F}^{\mathbf{\tau}%
-\operatorname*{shift}}}^{\omega}g\right\Vert _{L^{2}\left(  \omega\right)
}^{2}.\nonumber
\end{align}
Indeed, if $J^{\prime}$ is small compared to $J$, and $J^{\prime}\subset
J_{J^{\prime}}\subset J$, we have%
\begin{align*}
&  \left\vert \int\left(  R_{J;\kappa_{1}}\bigtriangleup_{J;\kappa_{2}%
}^{\omega}g\right)  \left(  R_{J^{\prime};\kappa_{1}}\bigtriangleup
_{J^{\prime};\kappa_{2}}^{\omega}g\right)  d\omega\right\vert \\
&  =\left\Vert R_{J;\kappa_{1}}\bigtriangleup_{J;\kappa_{2}}^{\omega
}g\right\Vert _{L^{\infty}\left(  \omega\right)  }\left\Vert R_{J^{\prime
};\kappa_{1}}\bigtriangleup_{J^{\prime};\kappa_{2}}^{\omega}g\right\Vert
_{L^{\infty}\left(  \omega\right)  }\left\vert \int\left(  \frac
{R_{J;\kappa_{1}}\bigtriangleup_{J;\kappa_{2}}^{\omega}g}{\left\Vert
R_{J;\kappa_{1}}\bigtriangleup_{J;\kappa_{2}}^{\omega}g\right\Vert
_{L^{\infty}\left(  \omega\right)  }}\right)  \left(  \frac{R_{J^{\prime
};\kappa_{1}}\bigtriangleup_{J^{\prime};\kappa_{2}}^{\omega}g}{\left\Vert
R_{J^{\prime};\kappa_{1}}\bigtriangleup_{J^{\prime};\kappa_{2}}^{\omega
}g\right\Vert _{L^{\infty}\left(  \omega\right)  }}\right)  d\omega\right\vert
\\
&  \lesssim\left\Vert R_{J;\kappa_{1}}\bigtriangleup_{J;\kappa_{2}}^{\omega
}g\right\Vert _{L^{\infty}\left(  \omega\right)  }\left\Vert R_{J^{\prime
};\kappa_{1}}\bigtriangleup_{J^{\prime};\kappa_{2}}^{\omega}g\right\Vert
_{L^{\infty}\left(  \omega\right)  }\frac{\ell\left(  J^{\prime}\right)
}{\ell\left(  J\right)  }\sqrt{\frac{\left\vert J^{\prime}\right\vert
_{\omega}}{\left\vert J\right\vert _{\omega}}}\sqrt{\left\vert J^{\prime
}\right\vert _{\omega}\left\vert J\right\vert _{\omega}}\\
&  \lesssim\sqrt{\frac{\left\vert J^{\prime}\right\vert _{\omega}}{\left\vert
J\right\vert _{\omega}}}\frac{\ell\left(  J^{\prime}\right)  }{\ell\left(
J\right)  }\left\Vert \bigtriangleup_{J^{\prime};\kappa_{2}}^{\omega
}g\right\Vert _{L^{2}\left(  \omega\right)  }\left\Vert \bigtriangleup
_{J;\kappa_{2}}^{\omega}g\right\Vert _{L^{2}\left(  \omega\right)  },
\end{align*}
by (\ref{analogue'}), i.e.%
\begin{align*}
\left\Vert R_{J^{\prime};\kappa_{1}}\bigtriangleup_{J^{\prime};\kappa_{2}%
}^{\omega}g\right\Vert _{L^{\infty}\left(  \omega\right)  }\sqrt{\left\vert
J^{\prime}\right\vert _{\omega}}  &  \lesssim\left\Vert \bigtriangleup
_{J^{\prime};\kappa_{2}}^{\omega}g\right\Vert _{L^{\infty}\left(
\omega\right)  }\sqrt{\left\vert J^{\prime}\right\vert _{\omega}}%
\lesssim\left\Vert \bigtriangleup_{J^{\prime};\kappa_{2}}^{\omega}g\right\Vert
_{L^{2}\left(  \omega\right)  },\\
\left\Vert R_{J;\kappa_{1}}\bigtriangleup_{J;\kappa_{2}}^{\omega}g\right\Vert
_{L^{\infty}\left(  \omega\right)  }\sqrt{\left\vert J\right\vert _{\omega}}
&  \lesssim\left\Vert \bigtriangleup_{J;\kappa_{2}}^{\omega}g\right\Vert
_{L^{\infty}\left(  \omega\right)  }\sqrt{\left\vert J\right\vert _{\omega}%
}\lesssim\left\Vert \bigtriangleup_{J;\kappa_{2}}^{\omega}g\right\Vert
_{L^{2}\left(  \omega\right)  }.
\end{align*}
Thus%
\begin{align*}
&  \sum_{\substack{J,J^{\prime}\in\mathcal{C}_{F}^{\mathbf{\tau}%
-\operatorname*{shift}}\\J^{\prime}\subsetneqq J}}\left\vert \int\left(
R_{J;\kappa_{1}}\bigtriangleup_{J;\kappa_{2}}^{\omega}g\right)  \left(
R_{J^{\prime};\kappa_{1}}\bigtriangleup_{J^{\prime};\kappa_{2}}^{\omega
}g\right)  d\omega\right\vert \lesssim\sum_{\substack{J,J^{\prime}%
\in\mathcal{C}_{F}^{\mathbf{\tau}-\operatorname*{shift}}\\J^{\prime
}\subsetneqq J}}\sqrt{\frac{\left\vert J^{\prime}\right\vert _{\omega}%
}{\left\vert J\right\vert _{\omega}}}\frac{\ell\left(  J^{\prime}\right)
}{\ell\left(  J\right)  }\left\Vert \bigtriangleup_{J^{\prime};\kappa_{2}%
}^{\omega}g\right\Vert _{L^{2}\left(  \omega\right)  }\left\Vert
\bigtriangleup_{J;\kappa_{2}}^{\omega}g\right\Vert _{L^{2}\left(
\omega\right)  }\\
&  =\sum_{m=1}^{\infty}2^{-m}\sum_{\substack{J,J^{\prime}\in\mathcal{C}%
_{F}^{\mathbf{\tau}-\operatorname*{shift}}\\\ell\left(  J^{\prime}\right)
=2^{-m}\ell\left(  J\right)  }}\sqrt{\frac{\left\vert J^{\prime}\right\vert
_{\omega}}{\left\vert J\right\vert _{\omega}}}\left\Vert \bigtriangleup
_{J^{\prime};\kappa_{2}}^{\omega}g\right\Vert _{L^{2}\left(  \omega\right)
}\left\Vert \bigtriangleup_{J;\kappa_{2}}^{\omega}g\right\Vert _{L^{2}\left(
\omega\right)  }\\
&  \leq\sum_{m=1}^{\infty}2^{-m}\sqrt{\sum_{\substack{J^{\prime}\in
\mathcal{C}_{F}^{\mathbf{\tau}-\operatorname*{shift}}\\\pi^{\left(  m\right)
}J^{\prime}\in\mathcal{C}_{F}^{\mathbf{\tau}-\operatorname*{shift}}%
}}\left\Vert \bigtriangleup_{J^{\prime};\kappa_{2}}^{\omega}g\right\Vert
_{L^{2}\left(  \omega\right)  }^{2}}\sqrt{\sum_{\substack{J^{\prime}%
\in\mathcal{C}_{F}^{\mathbf{\tau}-\operatorname*{shift}}\\J=\pi^{\left(
m\right)  }J^{\prime}\in\mathcal{C}_{F}^{\mathbf{\tau}-\operatorname*{shift}}%
}}\frac{\left\vert J^{\prime}\right\vert _{\omega}}{\left\vert J\right\vert
_{\omega}}\left\Vert \bigtriangleup_{J;\kappa_{2}}^{\omega}g\right\Vert
_{L^{2}\left(  \omega\right)  }^{2}}\lesssim\sum_{J\in\mathcal{C}%
_{F}^{\mathbf{\tau}-\operatorname*{shift}}}\left\Vert \bigtriangleup
_{J;\kappa_{2}}^{\omega}g\right\Vert _{L^{2}\left(  \omega\right)  }^{2}\ .
\end{align*}
Altogether we have shown%
\begin{equation}
\left\vert \mathsf{T}_{\operatorname*{paraproduct}}^{\subset_{\mathbf{\rho
},\varepsilon},F}\right\vert \lesssim\mathfrak{T}_{T^{\lambda}}\alpha
_{\mathcal{F}}\left(  F\right)  \sqrt{\left\vert F\right\vert _{\sigma}%
}\ \left\Vert \mathsf{P}_{\mathcal{C}_{F}^{\mathbf{\tau}-\operatorname*{shift}%
}}^{\omega}g\right\Vert _{L^{2}\left(  \omega\right)  } \label{para est}%
\end{equation}
as required by (\ref{below form bound}).

\subsection{The below commutator form}

We show that \emph{below commutator} form converges absolutely, in the sense
that
\begin{align}
&  \left\vert \mathsf{T}_{\operatorname*{commutator}}^{\subset_{\rho
,\varepsilon},F}\left(  f,g\right)  \right\vert \leq\sum_{\substack{I\in
\mathcal{C}_{F}\text{ and }J\in\mathcal{C}_{F}^{\mathbf{\tau}%
-\operatorname*{shift}}\\J\subset_{\mathbf{\rho},\varepsilon}I}} \left\vert
\left\langle \left[  T_{\sigma}^{\lambda},M_{I_{J};\kappa}\right]
\mathbf{1}_{I_{_{J}}},\bigtriangleup_{J;\kappa_{2}}^{\omega}g\right\rangle
_{\omega}\right\vert \lesssim\sqrt{A_{2}^{\lambda}}\left\Vert \mathsf{P}%
_{\mathcal{C}_{F}}^{\sigma}f\right\Vert _{L^{2}\left(  \sigma\right)
}\left\Vert \mathsf{P}_{\mathcal{C}_{F}^{\mathbf{\tau}-\operatorname*{shift}}%
}^{\omega}g\right\Vert _{L^{2}\left(  \omega\right)  } \, . \label{comm est}%
\end{align}

If $T=H$ is the Hilbert transform on the real line, and if $P_{\ell}\left(
x\right)  =x^{\ell}$ with $1\leq\ell\leq\kappa$ , then by the moment vanishing
properties of Alpert projections, we get that $H$ commutes with polynomials
$P$ of degree at most $\kappa$ when acting on a function $f$ with vanishing
$\sigma$-means up to order $\kappa-1$, i.e.
\[
\left\langle \left(  H_{\sigma}P-PH_{\sigma}\right)  f,g\right\rangle
_{\omega}=\left\langle 0,g\right\rangle _{\omega}=0.
\]
By duality we also have%
\[
\left\langle \left(  H_{\sigma}P-PH_{\sigma}\right)  f,g\right\rangle
_{\omega}=\left\langle f,\left(  H_{\omega}P-PH_{\omega}\right)
g\right\rangle _{\sigma}=\left\langle f,0\right\rangle _{\sigma}=0,
\]
if $g$ has vanishing $\omega$-means up to order $\kappa-1$. This motivates the
following argument.

\begin{notation}
We will take $\kappa_{2}\geq\kappa_{1}$ throughout this subsection, and for
convenience we write $\kappa=\kappa_{1}$. Then $\bigtriangleup_{J;\kappa_{2}%
}^{\omega}g\left(  x\right)  $ has vanishing $\omega$-means up to order
$\kappa_{2}-\kappa_{1}+1\geq1$. At many points in the arguments below we will
simply use that $\bigtriangleup_{J;\kappa_{2}}^{\omega}g\left(  x\right)  $
has $1$ vanishing $\omega$-moment, and continue with this single vanishing
$\omega$-moment in subsequent estimates, without further reference to the fact
that the estimates could be improved for $\kappa_{2}$ strictly larger than
$\kappa=\kappa_{1}$ (as these improved estimates are not needed for the
commutator form).
\end{notation}

Fix $\kappa\geq1$. Assume that $K^{\lambda}$ is a general standard $\lambda
$-fractional kernel in $\mathbb{R}^{n}$, and $T^{\lambda}$ is the associated
Calder\'{o}n-Zygmund operator, and define
\[
P_{\alpha,a,I^{\prime}}\left(  x\right)  =\left(  \frac{x-a}{\ell\left(
I^{\prime}\right)  }\right)  ^{\alpha}=\left(  \frac{x_{1}-a_{1}}{\ell\left(
I^{\prime}\right)  }\right)  ^{\alpha_{1}}...\left(  \frac{x_{n}-a_{n}}%
{\ell\left(  I^{\prime}\right)  }\right)  ^{\alpha_{n}},
\]
where $1\leq\left\vert \alpha\right\vert \leq\kappa-1$ (since when
$|\alpha|=0$, $P_{\alpha,a,I^{\prime}}$ commutes with $T^{\lambda}$) and
$I^{\prime}\in\mathfrak{C}_{\mathcal{D}}\left(  I\right)  $, $I\in
\mathcal{C}_{F}$.

We consider the renormalization $Q_{I^{\prime};\kappa}^{\mu}$ of the
polynomial $M_{I^{\prime};\kappa}$ introduced earlier, given by
\[
Q_{I^{\prime};\kappa}^{\mu}\equiv\frac{1}{\left\vert \widehat{f}\left(
I\right)  \right\vert }1_{I^{\prime}}\bigtriangleup_{I;\kappa}^{\sigma}%
f=\frac{1}{\left\vert \widehat{f}\left(  I\right)  \right\vert }M_{I^{\prime
};\kappa} \, .
\]
For $c_{J}\in J\subset I^{\prime}$, write%
\[
Q_{I^{\prime};\kappa}^{\sigma}\left(  y\right)  =\sum_{\left\vert
\alpha\right\vert <\kappa}b_{\alpha}\left(  \frac{y-c_{J}}{\ell\left(
I^{\prime}\right)  }\right)  ^{\alpha}=\sum_{\left\vert \alpha\right\vert
<\kappa}b_{\alpha}P_{\alpha,c_{J},I^{\prime}}\left(  y\right)  \, .
\]
By rescaling to the unit cube and invoking the fact that any two norms on a
finite dimensional vector space are equivalent, followed by then noting that
from (\ref{analogue'}) we get $\left\Vert Q_{I^{\prime};\kappa}^{\sigma
}\right\Vert _{\infty}\approx\frac{1}{\sqrt{\left\vert I\right\vert _{\sigma}%
}}$, then we have
\begin{equation}
\sum_{\left\vert \alpha\right\vert <\kappa}\left\vert b_{\alpha}\right\vert
\approx\left\Vert Q_{I^{\prime};\kappa}^{\sigma}\right\Vert _{\infty}%
\approx\frac{1}{\sqrt{\left\vert I\right\vert _{\sigma}}}. \label{we obtain}%
\end{equation}

We then bound%
\begin{align*}
\left\vert \left\langle \left[  M_{I^{\prime};\kappa},T_{\sigma}^{\lambda
}\right]  \mathbf{1}_{I^{\prime}},\bigtriangleup_{J;\kappa_{2}}^{\omega
}g\right\rangle _{\omega}\right\vert \leq\left|  \hat{f} \left(  I \right)
\right|  \left\vert \left\langle \left[  Q_{I^{\prime};\kappa},T_{\sigma
}^{\lambda}\right]  \mathbf{1}_{I^{\prime}},\bigtriangleup_{J;\kappa_{2}%
}^{\omega}g\right\rangle _{\omega}\right\vert \leq\sum_{\left\vert
\alpha\right\vert <\kappa}\left\vert b_{\alpha}\left\langle \left[
P_{\alpha,c_{J},I^{\prime}},T_{\sigma}^{\lambda}\right]  \mathbf{1}%
_{I^{\prime}},\bigtriangleup_{J;\kappa_{2}}^{\omega}g\right\rangle _{\omega
}\right\vert
\end{align*}
\[
\lesssim\frac{ \left|  \hat{f} \left(  I \right)  \right|  }{\sqrt{\left\vert
I\right\vert _{\sigma}}}\max_{\left\vert \alpha\right\vert <\kappa}\left\vert
\left\langle \left[  P_{\alpha,c_{J},I^{\prime}},T_{\sigma}^{\lambda}\right]
\mathbf{1}_{I^{\prime}},\bigtriangleup_{J;\kappa_{2}}^{\omega}g\right\rangle
_{\omega}\right\vert \, ,
\]
so we turn to estimating
\[
\left\vert \left\langle \left[  P_{\alpha,c_{J},I^{\prime}},T_{\sigma
}^{\lambda}\right]  \mathbf{1}_{I^{\prime}},\bigtriangleup_{J;\kappa_{2}%
}^{\omega}g\right\rangle _{\omega}\right\vert
\]
uniformly in $\alpha$.

Taking $J\subset I^{\prime}$, we begin by writing
\begin{align}
&  \ \ \ \ \ \ \ \ \ \ \ \ \ \ \ \left\langle \left[  P_{\alpha,a,I^{\prime}%
}\ ,T_{\sigma}^{\lambda}\right]  \mathbf{1}_{I^{\prime}},\bigtriangleup
_{J;\kappa_{2}}^{\omega}g\right\rangle _{\omega}=\int\left[  P_{\alpha
,a,I^{\prime}}\ ,T_{\sigma}^{\lambda}\right]  \mathbf{1}_{I^{\prime}}\left(
x\right)  \bigtriangleup_{J;\kappa_{2}}^{\omega}g\left(  x\right)
d\omega\left(  x\right) \label{two pieces}\\
&  =\int\left[  P_{\alpha,a,I^{\prime}}\ ,T_{\sigma}^{\lambda}\right]
\mathbf{1}_{I^{\prime}\setminus2J}\left(  x\right)  \bigtriangleup
_{J;\kappa_{2}}^{\omega}g\left(  x\right)  d\omega\left(  x\right)
+\int\left[  P_{\alpha,a,I^{\prime}}\ ,T_{\sigma}^{\lambda}\right]
\mathbf{1}_{2J}\left(  x\right)  \bigtriangleup_{J;\kappa_{2}}^{\omega
}g\left(  x\right)  d\omega\left(  x\right) \nonumber\\
&  \equiv\operatorname{Int}^{\lambda,\natural}\left(  J\right)
+\operatorname{Int}^{\lambda,\flat}\left(  J\right)  ,\nonumber
\end{align}
where we are suppressing the dependence on both $\alpha$ and $\kappa$.

Let us address the first term. We use the known identity
\[
x^{\alpha}-y^{\alpha}=\sum_{k=1}^{n}\left(  x_{k}-y_{k}\right)
{\displaystyle\sum\limits_{\left\vert \beta\right\vert +\left\vert
\gamma\right\vert =\left\vert \alpha\right\vert -1}}
c_{\alpha,\beta,\gamma}x^{\beta}y^{\gamma},
\]
to write the pointwise equality
\begin{align*}
&  \mathbf{1}_{I^{\prime}}\left(  x\right)  \left[  P_{\alpha,a,I^{\prime}%
}\ ,T_{\sigma}^{\lambda}\right]  \mathbf{1}_{I^{\prime}}\left(  x\right)
=\mathbf{1}_{I^{\prime}}\left(  x\right)  \int K^{\lambda}\left(  x-y\right)
\left\{  P_{\alpha,a,I^{\prime}}\left(  x\right)  -P_{\alpha,a,I^{\prime}%
}\left(  y\right)  \right\}  \mathbf{1}_{I^{\prime}}\left(  y\right)
d\sigma\left(  y\right) \\
&  =\mathbf{1}_{I^{\prime}}\left(  x\right)  \int K^{\lambda}\left(
x-y\right)  \left\{  \sum_{k=1}^{n}\left(  \frac{x_{k}-y_{k}}{\ell\left(
I^{\prime}\right)  }\right)
{\displaystyle\sum\limits_{\left\vert \beta\right\vert +\left\vert
\gamma\right\vert =\left\vert \alpha\right\vert -1}}
c_{\alpha,\beta,\gamma}\left(  \frac{x-a}{\ell\left(  I^{\prime}\right)
}\right)  ^{\beta}\left(  \frac{y-a}{\ell\left(  I^{\prime}\right)  }\right)
^{\gamma}\right\}  \mathbf{1}_{I^{\prime}}\left(  y\right)  d\sigma\left(
y\right) \\
&  =\sum_{k=1}^{n}%
{\displaystyle\sum\limits_{\left\vert \beta\right\vert +\left\vert
\gamma\right\vert =\left\vert \alpha\right\vert -1}}
c_{\alpha,\beta,\gamma}\mathbf{1}_{I^{\prime}}\left(  x\right)  \left[
\int\Phi_{k}^{\lambda}\left(  x-y\right)  \left\{  \left(  \frac{y-a}%
{\ell\left(  I^{\prime}\right)  }\right)  ^{\gamma}\right\}  \mathbf{1}%
_{I^{\prime}}\left(  y\right)  d\sigma\left(  y\right)  \right]  \left(
\frac{x-a}{\ell\left(  I^{\prime}\right)  }\right)  ^{\beta},
\end{align*}
where $\Phi_{k}^{\lambda}\left(  x-y\right)  =K^{\lambda}\left(  x-y\right)
\left(  \frac{x_{k}-y_{k}}{\ell\left(  I^{\prime}\right)  }\right)  $.

Integrating the above against $\bigtriangleup_{J; \kappa_{2}} ^{\omega} g$ and
then pulling out the double sum $\sum_{k=1}^{n} \sum\limits_{\left\vert
\beta\right\vert +\left\vert \gamma\right\vert =\left\vert \alpha\right\vert
-1}$ lets us write
\[
\operatorname{Int}^{\lambda,\natural}\left(  J\right)  \equiv\sum_{k=1}^{n}
\sum\limits_{\left\vert \beta\right\vert +\left\vert \gamma\right\vert
=\left\vert \alpha\right\vert -1} c_{\alpha,\beta,\gamma}\operatorname{Int}%
_{k,\beta,\gamma}^{\lambda,\natural}\left(  J\right)  ,
\]
where with the choice $a=c_{J}$ the center of $J$, we define%
\begin{align}
\operatorname{Int}_{k,\beta,\gamma}^{\lambda,\natural}\left(  J\right)   &
\equiv\int_{J}\left[  \int_{I^{\prime}\setminus2J}\Phi_{k}^{\lambda}\left(
x-y\right)  \left(  \frac{y-c_{J}}{\ell\left(  I^{\prime}\right)  }\right)
^{\gamma}d\sigma\left(  y\right)  \right]  \left(  \frac{x-c_{J}}{\ell\left(
I^{\prime}\right)  }\right)  ^{\beta}\bigtriangleup_{J;\kappa_{2}}^{\omega
}g\left(  x\right)  d\omega\left(  x\right) \label{coefficients}\\
&  =\int_{I^{\prime}\setminus2J}\left\{  \int_{J}\Phi_{k}^{\lambda}\left(
x-y\right)  \left(  \frac{x-c_{J}}{\ell\left(  I^{\prime}\right)  }\right)
^{\beta}\bigtriangleup_{J;\kappa_{2}}^{\omega}g\left(  x\right)
d\omega\left(  x\right)  \right\}  \left(  \frac{y-c_{J}}{\ell\left(
I^{\prime}\right)  }\right)  ^{\gamma}d\sigma\left(  y\right)  .\nonumber
\end{align}

Taking
\[
h\left(  x\right)  \equiv\left(  \frac{x-c_{J}}{\ell\left(  I^{\prime}\right)
}\right)  ^{\beta}\bigtriangleup_{J;\kappa_{2}}^{\omega}g\left(  x\right)  \,
,
\]
which has support in $J$ and at $\kappa-|\beta| + 1$ vanishing moments, by
Taylor's formula and the Calderon-Zygmund estimates for $K^{\lambda}$, we have
the inner-most integral has absolute value
\[
\left\vert \int\Phi_{k}^{\lambda}\left(  x-y\right)  h\left(  x\right)
d\omega\left(  x\right)  \right\vert =\left\vert \int\frac{1}{\left(
\kappa-\left\vert \beta\right\vert \right)  !}\left(  \left(  x-c_{J}\right)
\cdot\nabla\right)  ^{\kappa-\left\vert \beta\right\vert }\Phi_{k}^{\lambda
}\left(  \eta_{J}^{\omega}\right)  h\left(  x\right)  d\omega\left(  x\right)
\right\vert
\]
\[
\lesssim\left\Vert h\right\Vert _{L^{1}\left(  \omega\right)  }\frac
{\ell\left(  J\right)  ^{\kappa-\left\vert \beta\right\vert }}{\left[
\ell\left(  J\right)  +\operatorname*{dist}\left(  y,J\right)  \right]
^{\kappa-\left\vert \beta\right\vert +n-\lambda-1}\ell\left(  I^{\prime
}\right)  }\lesssim\left(  \frac{\ell\left(  J\right)  }{\ell\left(
I^{\prime}\right)  }\right)  ^{\left\vert \beta\right\vert }\frac{\ell\left(
J\right)  ^{\kappa-\left\vert \beta\right\vert }}{\left[  \ell\left(
J\right)  +\operatorname*{dist}\left(  y,J\right)  \right]  ^{\kappa
-\left\vert \beta\right\vert +n-\lambda-1}\ell\left(  I^{\prime}\right)
}\sqrt{\left\vert J\right\vert _{\omega}}\left\vert \widehat{g}\left(
J\right)  \right\vert ,
\]
where in the last inequality we used the estimate
\[
\left\Vert h\right\Vert _{L^{1}\left(  \omega\right)  }=\int_{J}\left\vert
\left(  \frac{x-c_{J}}{\ell\left(  I^{\prime}\right)  }\right)  ^{\beta
}\bigtriangleup_{J;\kappa_{2}}^{\omega}g\left(  x\right)  \right\vert
d\omega\left(  x\right)  \leq\left(  \frac{\ell\left(  J\right)  }{\ell\left(
I^{\prime}\right)  }\right)  ^{\left\vert \beta\right\vert }\left\Vert
\bigtriangleup_{J;\kappa_{2}}^{\omega}g\right\Vert _{L^{1}\left(
\omega\right)  }\lesssim\left(  \frac{\ell\left(  J\right)  }{\ell\left(
I^{\prime}\right)  }\right)  ^{\left\vert \beta\right\vert }\sqrt{\left\vert
J\right\vert _{\omega}}\left\vert \widehat{g}\left(  J\right)  \right\vert .
\]
Thus (\ref{coefficients}) yields
\begin{align*}
&  \left\vert \operatorname{Int}_{k,\beta,\gamma}^{\lambda,\natural}\left(
J\right)  \right\vert \leq\int_{I^{\prime}\setminus2J}\left\vert \int\Phi
_{k}^{\lambda}\left(  x-y\right)  \left(  \frac{x-c_{J}}{\ell\left(
I^{\prime}\right)  }\right)  ^{\beta}\bigtriangleup_{J;\kappa_{2}}^{\omega
}g\left(  x\right)  d\omega\left(  x\right)  \right\vert \left\vert \left(
\frac{y-c_{J}}{\ell\left(  I^{\prime}\right)  }\right)  ^{\gamma}\right\vert
d\sigma\left(  y\right) \\
&  \lesssim\int_{I^{\prime}\setminus2J}\left(  \frac{\ell\left(  J\right)
}{\ell\left(  I^{\prime}\right)  }\right)  ^{\left\vert \beta\right\vert
}\frac{\ell\left(  J\right)  ^{\kappa-\left\vert \beta\right\vert }}{\left[
\ell\left(  J\right)  +\operatorname*{dist}\left(  y,J\right)  \right]
^{\kappa-\left\vert \beta\right\vert +n-\lambda-1}\ell\left(  I\right)  }%
\sqrt{\left\vert J\right\vert _{\omega}}\left\vert \widehat{g}\left(
J\right)  \right\vert \left(  \frac{\ell\left(  J\right)
+\operatorname*{dist}\left(  y,J\right)  }{\ell\left(  I^{\prime}\right)
}\right)  ^{\left\vert \gamma\right\vert }d\sigma\left(  y\right) \\
&  =\left(  \frac{\ell\left(  J\right)  }{\ell\left(  I^{\prime}\right)
}\right)  ^{\left\vert \alpha\right\vert -1}\sqrt{\left\vert J\right\vert
_{\omega}}\left\vert \widehat{g}\left(  J\right)  \right\vert \left\{
\int_{I^{\prime}\setminus2J}\left(  \frac{\ell\left(  J\right)  }{\ell\left(
J\right)  +\operatorname*{dist}\left(  y,J\right)  }\right)  ^{\kappa
-\left\vert \alpha\right\vert +1}\frac{1}{\left[  \ell\left(  J\right)
+\operatorname*{dist}\left(  y,J\right)  \right]  ^{n-\lambda-1}\ell\left(
I^{\prime}\right)  }d\sigma\left(  y\right)  \right\}  .
\end{align*}

Now we fix $t\in\mathbb{N}$, and estimate the sum of $\left\vert
\operatorname{Int}^{\lambda,\natural}\left(  J\right)  \right\vert $ over
those $J\subset I^{\prime}$ with $\ell\left(  J\right)  =2^{-t}\ell\left(
I^{\prime}\right)  $ by splitting the integration in $y$ according to the size
of $\ell\left(  J\right)  +\operatorname*{dist}\left(  y,J\right)  $, to
obtain the following bound:%
\begin{align*}
&  \sum_{\substack{J\subset I^{\prime}\\\ell\left(  J\right)  =2^{-t}%
\ell\left(  I\right)  }}\left\vert \operatorname{Int}^{\lambda,\natural
}\left(  J\right)  \right\vert \\
&  \lesssim2^{-t\left(  \left\vert \alpha\right\vert -1\right)  }%
\sum_{\substack{J\subset I^{\prime}\\\ell\left(  J\right)  =2^{-t}\ell\left(
I\right)  }}\sqrt{\left\vert J\right\vert _{\omega}}\left\vert \widehat
{g}\left(  J\right)  \right\vert \left\{  \int_{I^{\prime}\setminus2J}\left(
\frac{\ell\left(  J\right)  }{\ell\left(  J\right)  +\operatorname*{dist}%
\left(  y,J\right)  }\right)  ^{\kappa-\left\vert \alpha\right\vert +1}%
\frac{d\sigma\left(  y\right)  }{\left[  \ell\left(  J\right)
+\operatorname*{dist}\left(  y,J\right)  \right]  ^{n-\lambda-1}\ell\left(
I^{\prime}\right)  }\right\} \\
&  \lesssim2^{-t\left(  \left\vert \alpha\right\vert -1\right)  }%
\sum_{J\subset I^{\prime}:\ \ell\left(  J\right)  =2^{-t}\ell\left(  I\right)
}\sqrt{\left\vert J\right\vert _{\omega}}\left\vert \widehat{g}\left(
J\right)  \right\vert \left\{  \sum_{s=1}^{t}\int_{2^{s+1}J\setminus2^{s}%
J}\left(  2^{-s}\right)  ^{\kappa-\left\vert \alpha\right\vert +1}%
\frac{d\sigma\left(  y\right)  }{\left(  2^{s}\ell\left(  J\right)  \right)
^{n-\lambda-1}\ell\left(  I^{\prime}\right)  }\right\} \\
&  \lesssim2^{-t\left\vert \alpha\right\vert }\sum_{J\subset I^{\prime}%
:\ \ell\left(  J\right)  =2^{-t}\ell\left(  I\right)  }\sqrt{\left\vert
J\right\vert _{\omega}}\left\vert \widehat{g}\left(  J\right)  \right\vert
\sum_{s=1}^{t}\left(  2^{-s}\right)  ^{\kappa-\left\vert \alpha\right\vert
+1}2^{-s\left(  n-\lambda-1\right)  }\frac{\left\vert 2^{s}J\right\vert
_{\sigma}}{\ell\left(  J\right)  ^{n-\lambda}},
\end{align*}
which, upon pigeonholing the sum in $J$ according to membership in the
grandchildren of $I$ at depth $t-s$, gives:
\begin{align*}
&  \sum_{J\in\mathfrak{C}_{\mathcal{D}}^{\left(  t\right)  }\left(  I^{\prime
}\right)  }\left\vert \operatorname{Int}^{\lambda,\natural}\left(  J\right)
\right\vert \lesssim2^{-t\left\vert \alpha\right\vert }\sum_{J\in
\mathfrak{C}_{\mathcal{D}}^{\left(  t\right)  }\left(  I^{\prime}\right)
}\sqrt{\left\vert J\right\vert _{\omega}}\left\vert \widehat{g}\left(
J\right)  \right\vert \sum_{s=1}^{t}\left(  2^{-s}\right)  ^{\kappa-\left\vert
\alpha\right\vert +n-\lambda}\frac{\left\vert 2^{s}J\right\vert _{\sigma}%
}{\ell\left(  J\right)  ^{n-\lambda}}\\
&  =2^{-t\left\vert \alpha\right\vert }\sum_{s=1}^{t}\left(  2^{-s}\right)
^{\kappa-\left\vert \alpha\right\vert }\sum_{K\in\mathfrak{C}_{\mathcal{D}%
}^{\left(  t-s\right)  }\left(  I^{\prime}\right)  }\sum_{J\in\mathfrak{C}%
_{\mathcal{D}}^{\left(  s\right)  }\left(  K\right)  }\sqrt{\left\vert
J\right\vert _{\omega}}\left\vert \widehat{g}\left(  J\right)  \right\vert
\frac{\left\vert 2^{s}J\right\vert _{\sigma}}{\ell\left(  K\right)
^{n-\lambda}}\\
&  \lesssim2^{-t\left\vert \alpha\right\vert }\sum_{s=1}^{t}\left(
2^{-s}\right)  ^{\kappa-\left\vert \alpha\right\vert }\sum_{K\in
\mathfrak{C}_{\mathcal{D}}^{\left(  t-s\right)  }\left(  I^{\prime}\right)
}\frac{\left\vert 3K\right\vert _{\sigma}}{\ell\left(  K\right)  ^{n-\lambda}%
}\sum_{J\in\mathfrak{C}_{\mathcal{D}}^{\left(  s\right)  }\left(  K\right)
}\sqrt{\left\vert J\right\vert _{\omega}}\left\vert \widehat{g}\left(
J\right)  \right\vert \\
&  \lesssim2^{-t\left\vert \alpha\right\vert }\sum_{s=1}^{t}\left(
2^{-s}\right)  ^{\kappa-\left\vert \alpha\right\vert }\sum_{K\in
\mathfrak{C}_{\mathcal{D}}^{\left(  t-s\right)  }\left(  I^{\prime}\right)
}\frac{\left\vert 3K\right\vert _{\sigma}}{\ell\left(  K\right)  ^{n-\lambda}%
}\sqrt{\left\vert K\right\vert _{\omega}}\sqrt{\sum_{J\in\mathfrak{C}%
_{\mathcal{D}}^{\left(  s\right)  }\left(  K\right)  }\left\vert \widehat
{g}\left(  J\right)  \right\vert ^{2}}\\
&  \lesssim2^{-t\left\vert \alpha\right\vert }\sqrt{A_{2}^{\lambda}}\sum
_{s=1}^{t}\left(  2^{-s}\right)  ^{\kappa-\left\vert \alpha\right\vert }%
\sum_{K\in\mathfrak{C}_{\mathcal{D}}^{\left(  t-s\right)  }\left(  I^{\prime
}\right)  }\sqrt{\left\vert K\right\vert _{\sigma}}\sqrt{\sum_{J\in
\mathfrak{C}_{\mathcal{D}}^{\left(  s\right)  }\left(  K\right)  }\left\vert
\widehat{g}\left(  J\right)  \right\vert ^{2}},
\end{align*}
where we used the $A_{2}^{\lambda}$ condition and doubling for $\sigma$ in the
last inequality. Thus we have%
\begin{align*}
\sum_{J\in\mathfrak{C}_{\mathcal{D}}^{\left(  t\right)  }\left(  I^{\prime
}\right)  }\left\vert \operatorname{Int}^{\lambda,\natural}\left(  J\right)
\right\vert  &  \lesssim2^{-t\left\vert \alpha\right\vert }\sqrt
{A_{2}^{\lambda}}\sum_{s=1}^{t}\left(  2^{-s}\right)  ^{\kappa-\left\vert
\alpha\right\vert }\sqrt{\left\vert I^{\prime}\right\vert _{\sigma}}\sqrt
{\sum_{J\in\mathfrak{C}_{\mathcal{D}}^{\left(  t\right)  }\left(  I^{\prime
}\right)  }\left\vert \widehat{g}\left(  J\right)  \right\vert ^{2}}\\
&  \lesssim2^{-t}\sqrt{A_{2}^{\lambda}}\sqrt{\left\vert I^{\prime}\right\vert
_{\sigma}}\sqrt{\sum_{J\in\mathfrak{C}_{\mathcal{D}}^{\left(  t\right)
}\left(  I^{\prime}\right)  }\left\vert \widehat{g}\left(  J\right)
\right\vert ^{2}},
\end{align*}
since $1\leq\left\vert \alpha\right\vert \leq\kappa-1$.

We now claim the same estimate holds for the sum of $\left\vert
\operatorname{Int}^{\lambda,\flat}\left(  J\right)  \right\vert $ over
$J\subset I^{\prime}$ with $\ell\left(  J\right)  =2^{-t}\ell\left(
I^{\prime}\right)  $. We write
\[
\operatorname{Int}^{\lambda,\flat}\left(  J\right)  =\sum_{k=1}^{n}%
{\displaystyle\sum\limits_{\left\vert \beta\right\vert +\left\vert
\gamma\right\vert =\left\vert \alpha\right\vert -1}}
c_{\alpha,\beta,\gamma}\operatorname{Int}_{k,\beta,\gamma}^{\lambda,\flat
}\left(  J\right)  ,
\]
and estimate
\begin{align*}
&  \left\vert \operatorname{Int}_{k,\beta,\gamma}^{\lambda,\flat}\left(
J\right)  \right\vert \lesssim\left\vert \int_{J}\left(  \int_{2J}\Phi
_{k}^{\lambda}\left(  x-y\right)  \left(  \frac{y-c_{J}}{\ell\left(
I^{\prime}\right)  }\right)  ^{\gamma}d\sigma\left(  y\right)  \right)
\left(  \frac{x-c_{J}}{\ell\left(  I^{\prime}\right)  }\right)  ^{\beta
}\bigtriangleup_{J;\kappa_{2}}^{\omega}g\left(  x\right)  d\omega\left(
x\right)  \right\vert \\
&  \leq\int_{J}\left(  \int_{2J}\frac{1}{\ell\left(  I\right)  \left\vert
x-y\right\vert ^{n-\lambda-1}}\left\vert \frac{y-c_{J}}{\ell\left(  I^{\prime
}\right)  }\right\vert ^{\left\vert \gamma\right\vert }d\sigma\left(
y\right)  \right)  \left\vert \frac{x-c_{J}}{\ell\left(  I^{\prime}\right)
}\right\vert ^{\left\vert \beta\right\vert }\left\vert \bigtriangleup
_{J;\kappa_{2}}^{\omega}g\left(  x\right)  \right\vert d\omega\left(  x\right)
\\
&  \lesssim\left(  \frac{\ell\left(  J\right)  }{\ell\left(  I^{\prime
}\right)  }\right)  ^{\left\vert \gamma\right\vert +\left\vert \beta
\right\vert }\frac{1}{\ell\left(  I^{\prime}\right)  \sqrt{\left\vert
J\right\vert _{\omega}}}\left\vert \widehat{g}\left(  J\right)  \right\vert
\int_{J}\int_{2J}\frac{d\sigma\left(  y\right)  d\omega\left(  x\right)
}{\left\vert x-y\right\vert ^{n-\lambda-1}}\lesssim\sqrt{A_{2}^{\lambda}}%
\frac{\ell\left(  J\right)  }{\ell\left(  I^{\prime}\right)  }\left\vert
\widehat{g}\left(  J\right)  \right\vert \sqrt{\left\vert J\right\vert
_{\sigma}},
\end{align*}
where in the last inequality we used that $\left\vert \beta\right\vert
+\left\vert \gamma\right\vert =\left\vert \alpha\right\vert \geq1$ and the
estimate
\[
\int_{J}\int_{2J}\frac{d\sigma\left(  y\right)  d\omega\left(  x\right)
}{\left\vert x-y\right\vert ^{n-\lambda-1}}\lesssim\sqrt{A_{2}^{\lambda}}%
\ell\left(  J\right)  \sqrt{\left\vert J\right\vert _{\sigma}\left\vert
J\right\vert _{\omega}}.
\]
Indeed, in order to estimate the double integral using the $A_{2}^{\lambda}$
condition we cover the band $\left\vert x-y\right\vert \leq C2^{-m}\ell\left(
J\right)  $ by a collection of cubes $Q\left(  z_{m},C2^{-m}\ell\left(
J\right)  \right)  \times Q\left(  z_{m},C2^{-m}\ell\left(  J\right)  \right)
$ in $CJ\times CJ$ with centers $\left(  z_{m},z_{m}\right)  $ and bounded
overlap. Then we have%
\begin{align*}
&  \int_{J}\int_{2J}\frac{d\sigma\left(  y\right)  d\omega\left(  x\right)
}{\left\vert x-y\right\vert ^{n-\lambda-1}}\leq\sum_{m=0}^{\infty}%
{\displaystyle\iint\limits_{\substack{x,y\in2J\\\left\vert x-y\right\vert
\approx2^{-m}\ell\left(  J\right)  }}}
\frac{d\sigma\left(  y\right)  d\omega\left(  x\right)  }{\left(  2^{-m}%
\ell\left(  J\right)  \right)  ^{n-\lambda-1}}\\
&  \approx\sum_{m=0}^{\infty}\sum_{Q\left(  z_{m},C2^{-m}\ell\left(  J\right)
\right)  \times Q\left(  z_{m},C2^{-m}\ell\left(  J\right)  \right)  }%
\int_{Q\left(  z_{m},C2^{-m}\ell\left(  J\right)  \right)  \times Q\left(
z_{m},C2^{-m}\ell\left(  J\right)  \right)  }\frac{d\sigma\left(  y\right)
d\omega\left(  x\right)  }{\left(  2^{-m}\ell\left(  J\right)  \right)
^{n-\lambda-1}}\\
&  \leq\frac{1}{\ell\left(  J\right)  ^{n-\lambda-1}}\sum_{m=0}^{\infty
}2^{m\left(  n-\lambda-1\right)  }\sum_{z_{m}}\left\vert Q\left(
z_{m},C2^{-m}\ell\left(  J\right)  \right)  \right\vert _{\sigma}\left\vert
Q\left(  z_{m},C2^{-m}\ell\left(  J\right)  \right)  \right\vert _{\omega}\\
&  \leq\frac{1}{\ell\left(  J\right)  ^{n-\lambda-1}}\sum_{m=0}^{\infty
}2^{m\left(  n-\lambda-1\right)  }\sum_{z_{m}}\sqrt{\left\vert Q\left(
z_{m},C2^{-m}\ell\left(  J\right)  \right)  \right\vert _{\sigma}\left\vert
Q\left(  z_{m},C2^{-m}\ell\left(  J\right)  \right)  \right\vert _{\omega}%
}\sqrt{A_{2}^{\lambda}}\left(  2^{-m}\ell\left(  J\right)  \right)
^{n-\lambda}\\
&  \lesssim\sqrt{A_{2}^{\lambda}}\ell\left(  J\right)  \sum_{m=0}^{\infty
}2^{m\left(  n-\lambda-1\right)  }2^{-m\left(  n-\lambda\right)  }%
\sqrt{\left\vert CJ\right\vert _{\sigma}\left\vert CJ\right\vert _{\omega}%
}\lesssim\sqrt{A_{2}^{\lambda}}\ell\left(  J\right)  \sqrt{\left\vert
J\right\vert _{\sigma}\left\vert J\right\vert _{\omega}}.
\end{align*}
Now%
\[
\sum_{J\in\mathfrak{C}_{\mathcal{D}}^{\left(  t\right)  }\left(  I^{\prime
}\right)  }\left\vert \operatorname{Int}^{\lambda,\flat}\left(  J\right)
\right\vert \lesssim\sum_{J\in\mathfrak{C}_{\mathcal{D}}^{\left(  t\right)
}\left(  I^{\prime}\right)  }\sqrt{A_{2}^{\lambda}}\frac{\ell\left(  J\right)
}{\ell\left(  I\right)  }\left\vert \widehat{g}\left(  J\right)  \right\vert
\sqrt{\left\vert J\right\vert _{\sigma}}\leq2^{-t}\sqrt{A_{2}^{\lambda}}%
\sqrt{\left\vert I^{\prime}\right\vert _{\sigma}}\sqrt{\sum_{J\in
\mathfrak{C}_{\mathcal{D}}^{\left(  t\right)  }\left(  I^{\prime}\right)
}\left\vert \widehat{g}\left(  J\right)  \right\vert ^{2}},
\]
and so altogether we have%
\begin{align*}
&  \sum_{J\in\mathfrak{C}_{\mathcal{D}}^{\left(  t\right)  }\left(  I^{\prime
}\right)  }\left\vert \left\langle \left[  P_{\alpha,c_{J},I^{\prime}%
},T_{\sigma}^{\lambda}\right]  \mathbf{1}_{I^{\prime}},\bigtriangleup
_{J;\kappa}^{\omega}g\right\rangle _{\omega}\right\vert =\sum_{J\in
\mathfrak{C}_{\mathcal{D}}^{\left(  t\right)  }\left(  I^{\prime}\right)
}\left\vert \operatorname{Int}^{\lambda}\left(  J\right)  \right\vert \\
&  \leq\sum_{J\in\mathfrak{C}_{\mathcal{D}}^{\left(  t\right)  }\left(
I^{\prime}\right)  }\left\vert \operatorname{Int}^{\lambda,\natural}\left(
J\right)  \right\vert +\sum_{J\in\mathfrak{C}_{\mathcal{D}}^{\left(  t\right)
}\left(  I^{\prime}\right)  }\left\vert \operatorname{Int}^{\lambda,\flat
}\left(  J\right)  \right\vert \lesssim2^{-t}\sqrt{A_{2}^{\lambda}}%
\sqrt{\left\vert I^{\prime}\right\vert _{\sigma}}\sqrt{\sum_{J\in
\mathfrak{C}_{\mathcal{D}}^{\left(  t\right)  }\left(  I^{\prime}\right)
}\left\vert \widehat{g}\left(  J\right)  \right\vert ^{2}}.
\end{align*}
Finally then we obtain from this and (\ref{we obtain}),%
\begin{align*}
\sum_{J\in\mathfrak{C}_{\mathcal{D}}^{\left(  t\right)  }\left(  I^{\prime
}\right)  }\left\vert \left\langle \left[  Q_{I^{\prime};\kappa},T_{\sigma
}^{\lambda}\right]  \mathbf{1}_{I^{\prime}},\bigtriangleup_{J;\kappa}^{\omega
}g\right\rangle _{\omega}\right\vert  &  =\sum_{J\in\mathfrak{C}_{\mathcal{D}%
}^{\left(  t\right)  }\left(  I^{\prime}\right)  }\left\vert \sum_{\left\vert
\alpha\right\vert \leq\kappa-1}b_{\alpha}\left\langle \left[  P_{\alpha
,c_{J},I^{\prime}},T_{\sigma}^{\lambda}\right]  \mathbf{1}_{I^{\prime}%
},\bigtriangleup_{J;\kappa}^{\omega}g\right\rangle _{\omega}\right\vert \\
&  \lesssim2^{-t}\sqrt{A_{2}^{\lambda}}\sqrt{\sum_{J\in\mathfrak{C}%
_{\mathcal{D}}^{\left(  t\right)  }\left(  I^{\prime}\right)  }\left\vert
\widehat{g}\left(  J\right)  \right\vert ^{2}}.
\end{align*}

Now using $M_{I_{J};\kappa}=\left\vert \widehat{f}\left(  I\right)
\right\vert Q_{I_{J};\kappa}$, and applying the above estimates with
$I^{\prime}=I_{J}$, we can sum over $t$ and $I\in\mathcal{C}_{F}$ to obtain
the following estimate for the \emph{below commutator} form,%
\begin{align}
&  \left\vert \mathsf{T}_{\operatorname*{commutator}}^{\subset_{\rho
,\varepsilon},F}\left(  f,g\right)  \right\vert \leq\sum_{\substack{I\in
\mathcal{C}_{F}\text{ and }J\in\mathcal{C}_{F}^{\mathbf{\tau}%
-\operatorname*{shift}}\\J\subset_{\mathbf{\rho},\varepsilon}I}}\left\vert
\widehat{f}\left(  I\right)  \right\vert \left\vert \left\langle \left[
T_{\sigma}^{\lambda},Q_{I_{J};\kappa}\right]  \mathbf{1}_{I_{_{J}}%
},\bigtriangleup_{J;\kappa_{2}}^{\omega}g\right\rangle _{\omega}\right\vert
\nonumber\\
&  \lesssim\sum_{t=\mathbf{r}}^{\infty}\sum_{I\in\mathcal{C}_{F}}2^{-t}%
\sqrt{A_{2}^{\lambda}}\left\vert \widehat{f}\left(  I\right)  \right\vert
\sqrt{\sum_{\substack{J\in\mathfrak{C}_{\mathcal{D}}^{\left(  t\right)
}\left(  I_{J}\right)  \text{ }\\J\in\mathcal{C}_{F}^{\mathbf{\tau
}-\operatorname*{shift}}}}\left\vert \widehat{g}\left(  J\right)  \right\vert
^{2}}\lesssim\sqrt{A_{2}^{\lambda}}\sum_{t=\mathbf{r}}^{\infty}2^{-t}%
\sqrt{\sum_{I\in\mathcal{C}_{F}}\left\vert \widehat{f}\left(  I\right)
\right\vert ^{2}}\sqrt{\sum_{I\in\mathcal{C}_{F}}\sum_{\substack{J\in
\mathfrak{C}_{\mathcal{D}}^{\left(  t\right)  }\left(  I_{J}\right)
\\J\in\mathcal{C}_{F}^{\mathbf{\tau}-\operatorname*{shift}}}}\left\vert
\widehat{g}\left(  J\right)  \right\vert ^{2}}\nonumber\\
&  \lesssim\sqrt{A_{2}^{\lambda}}\sum_{t=\mathbf{r}}^{\infty}2^{-t}\left\Vert
\mathsf{P}_{\mathcal{C}_{F}}^{\sigma}f\right\Vert _{L^{2}\left(
\sigma\right)  }\left\Vert \mathsf{P}_{\mathcal{C}_{F}^{\mathbf{\tau
}-\operatorname*{shift}}}^{\omega}g\right\Vert _{L^{2}\left(  \omega\right)
}\lesssim\sqrt{A_{2}^{\lambda}}\left\Vert \mathsf{P}_{\mathcal{C}_{F}}%
^{\sigma}f\right\Vert _{L^{2}\left(  \sigma\right)  }\left\Vert \mathsf{P}%
_{\mathcal{C}_{F}^{\mathbf{\tau}-\operatorname*{shift}}}^{\omega}g\right\Vert
_{L^{2}\left(  \omega\right)  }.\nonumber
\end{align}

\subsection{The below neighbour form}

We show the neighbour form converges absolutely, in the sense that the
neighbour form is bounded above by the sublinear form
\begin{align*}
\left\vert \mathsf{T}_{\operatorname*{neighbour}}^{\subset_{\rho,\varepsilon
},F} \right\vert \left(  f,g\right)   &  \equiv\sum_{\substack{I\in
\mathcal{C}_{F}\text{ and }J\in\mathcal{C}_{F}^{\mathbf{\tau}%
-\operatorname*{shift}}\\J\subset_{\mathbf{\rho},\varepsilon}I}}\sum
_{\theta\left(  I_{J}\right)  \in\mathfrak{C}_{\mathcal{D}}\left(  I\right)
\setminus\left\{  I_{J}\right\}  }\left\vert \left\langle T_{\sigma}^{\alpha
}\left(  \mathbf{1}_{\theta\left(  I_{J}\right)  }\bigtriangleup_{I;\kappa
_{1}}^{\sigma}f\right)  ,\bigtriangleup_{J;\kappa_{2}}^{\omega}g\right\rangle
_{\omega}\right\vert
\end{align*}
which satisfies
\begin{align}
\label{neigh est}\left\vert \mathsf{T}_{\operatorname*{neighbour}}%
^{\subset_{\rho,\varepsilon},F} \right\vert \left(  f,g\right)   &
\lesssim\sqrt{A_{2} ^{\alpha}} \left\Vert \mathsf{P}_{\mathcal{C}_{F}}%
^{\sigma}f\right\Vert _{L^{2}\left(  \sigma\right)  }\left\Vert \mathsf{P}%
_{\mathcal{C}_{F}^{\mathbf{\tau}-\operatorname*{shift}}}^{\omega}g\right\Vert
_{L^{2}\left(  \omega\right)  } \, .
\end{align}

We revert to using $\kappa_{1}$ and $\kappa_{2}$ in treating the \emph{below
neighbour} form, and we also revert to using $\alpha$ instead of $\lambda$ as
was done in the previous subsection. In the neighbour form we obtain the
required bound by taking absolute values inside the sum, and then arguing as
in the case of Haar wavelets in \cite[end of Subsection 8.4]{SaShUr7}. We
begin with $M_{I^{\prime};\kappa_{1}}=\mathbf{1}_{I^{\prime}}\bigtriangleup
_{I;\kappa}^{\sigma}f$ as in (\ref{def M}) to obtain
\begin{align*}
\left\vert \mathsf{T}_{\operatorname*{neighbour}}^{\subset_{\rho,\varepsilon
},F}\right\vert \left(  f,g\right)   &  =\sum_{\substack{I\in\mathcal{C}%
_{F}\text{ and }J\in\mathcal{C}_{F}^{\mathbf{\tau}-\operatorname*{shift}%
}\\J\subset_{\mathbf{\rho},\varepsilon}I}}\sum_{\theta\left(  I_{J}\right)
\in\mathfrak{C}_{\mathcal{D}}\left(  I\right)  \setminus\left\{
I_{J}\right\}  }\left\vert \left\langle T_{\sigma}^{\alpha}\left(
\mathbf{1}_{\theta\left(  I_{J}\right)  }\bigtriangleup_{I;\kappa_{1}}%
^{\sigma}f\right)  ,\bigtriangleup_{J;\kappa_{2}}^{\omega}g\right\rangle
_{\omega}\right\vert \\
&  \leq\sum_{\substack{I\in\mathcal{C}_{F}\text{ and }J\in\mathcal{C}%
_{F}^{\mathbf{\tau}-\operatorname*{shift}}\\J\subset_{\mathbf{\rho
},\varepsilon}I}}\sum_{I^{\prime}\equiv\theta\left(  I_{J}\right)
\in\mathfrak{C}_{\mathcal{D}}\left(  I\right)  \setminus\left\{
I_{J}\right\}  }\left\vert \left\langle T_{\sigma}^{\alpha}\left(
M_{I^{\prime};\kappa_{1}}\mathbf{1}_{I^{\prime}}\right)  ,\bigtriangleup
_{J;\kappa_{2}}^{\omega}g\right\rangle _{\omega}\right\vert
\end{align*}
Using the pivotal bound (\ref{piv bound}) on the inner product with
$\nu=\left\Vert M_{I^{\prime};\kappa_{1}}\right\Vert _{L^{\infty}\left(
\sigma\right)  }\mathbf{1}_{I^{\prime}}d\sigma$, and then estimating by the
usual Poisson kernel,%
\begin{align*}
\left\vert \left\langle T_{\sigma}^{\alpha}\left(  M_{I^{\prime};\kappa_{1}%
}\mathbf{1}_{I^{\prime}}\right)  ,\bigtriangleup_{J;\kappa_{2}}^{\omega
}g\right\rangle _{\omega}\right\vert  &  \lesssim\mathrm{P}_{\kappa_{2}%
}^{\alpha}\left(  J,\left\Vert M_{I^{\prime};\kappa_{1}}\right\Vert
_{L^{\infty}\left(  \sigma\right)  }\mathbf{1}_{I^{\prime}}\sigma\right)
\sqrt{\left\vert J\right\vert _{\omega}}\left\Vert \bigtriangleup
_{J;\kappa_{2}}^{\omega}g\right\Vert _{L^{2}\left(  \omega\right)  }\\
&  \leq\left\Vert M_{I^{\prime};\kappa_{1}}\right\Vert _{L^{\infty}\left(
\sigma\right)  }\mathrm{P}_{\kappa_{2}}^{\alpha}\left(  J,\mathbf{1}%
_{I^{\prime}}\sigma\right)  \sqrt{\left\vert J\right\vert _{\omega}}\left\Vert
\bigtriangleup_{J;\kappa_{2}}^{\omega}g\right\Vert _{L^{2}\left(
\omega\right)  },
\end{align*}
and the estimate $\left\Vert M_{I^{\prime};\kappa_{1}}\right\Vert _{L^{\infty
}\left(  \sigma\right)  }\approx\frac{1}{\sqrt{\left\vert I^{\prime
}\right\vert _{\sigma}}}\left\vert \widehat{f}\left(  I\right)  \right\vert $
from (\ref{analogue'}), to obtain%
\begin{align*}
&  \left\vert \mathsf{T}_{\operatorname*{neighbour}}^{\subset_{\rho
,\varepsilon},F}\right\vert \left(  f,g\right)  \lesssim\sum_{\substack{I\in
\mathcal{C}_{F}\text{ and }J\in\mathcal{C}_{F}^{\mathbf{\tau}%
-\operatorname*{shift}}\\J\subset_{\mathbf{\rho},\varepsilon}I}}\sum
_{I^{\prime}\equiv\theta\left(  I_{J}\right)  \in\mathfrak{C}_{\mathcal{D}%
}\left(  I\right)  \setminus\left\{  I_{J}\right\}  }\frac{\left\vert
\widehat{f}\left(  I\right)  \right\vert }{\sqrt{\left\vert I^{\prime
}\right\vert _{\sigma}}}\mathrm{P}_{\kappa_{2}}^{\alpha}\left(  J,\mathbf{1}%
_{I^{\prime}}\sigma\right)  \sqrt{\left\vert J\right\vert _{\omega}}\left\Vert
\bigtriangleup_{J;\kappa_{2}}^{\omega}g\right\Vert _{L^{2}\left(
\omega\right)  }\\
&  =\sum_{I\in\mathcal{C}_{F}}\sum_{\substack{I_{0},I_{\theta}\in
\mathfrak{C}_{\mathcal{D}}\left(  I\right)  \\I_{0}\neq I_{\theta}}%
}\sum_{\substack{J\in\mathcal{C}_{F}^{\mathbf{\tau}-\operatorname*{shift}%
}\\J\subset_{\mathbf{\rho},\varepsilon}I\text{ and }J\subset I_{0}}%
}\frac{\left\vert \widehat{f}\left(  I\right)  \right\vert }{\sqrt{\left\vert
I_{\theta}\right\vert _{\sigma}}}\mathrm{P}_{\kappa_{2}}^{\alpha}\left(
J,\mathbf{1}_{I_{\theta}}\sigma\right)  \sqrt{\left\vert J\right\vert
_{\omega}}\left\Vert \bigtriangleup_{J;\kappa_{2}}^{\omega}g\right\Vert
_{L^{2}\left(  \omega\right)  }\\
&  =\sum_{s=\mathbf{r}}^{\infty}\sum_{I\in\mathcal{C}_{F}}\sum
_{\substack{I_{0},I_{\theta}\in\mathfrak{C}_{\mathcal{D}}\left(  I\right)
\\I_{0}\neq I_{\theta}}}\sum_{\substack{J\in\mathcal{C}_{F}^{\mathbf{\tau
}-\operatorname*{shift}}\text{ and }\ell\left(  J\right)  =2^{-s}\ell\left(
I\right)  \\J\subset_{\mathbf{\rho},\varepsilon}I\text{ and }J\subset I_{0}%
}}\frac{\left\vert \widehat{f}\left(  I\right)  \right\vert }{\sqrt{\left\vert
I_{\theta}\right\vert _{\sigma}}}\mathrm{P}_{\kappa_{2}}^{\alpha}\left(
J,\mathbf{1}_{I_{\theta}}\sigma\right)  \sqrt{\left\vert J\right\vert
_{\omega}}\left\Vert \bigtriangleup_{J;\kappa_{2}}^{\omega}g\right\Vert
_{L^{2}\left(  \omega\right)  }\\
&  \lesssim\sum_{s=\mathbf{r}}^{\infty}\sum_{I\in\mathcal{C}_{F}}%
\sum_{\substack{I_{0},I_{\theta}\in\mathfrak{C}_{\mathcal{D}}\left(  I\right)
\\I_{0}\neq I_{\theta}}}\sum_{\substack{J\in\mathcal{C}_{F}^{\mathbf{\tau
}-\operatorname*{shift}}\text{ and }\ell\left(  J\right)  =2^{-s}\ell\left(
I\right)  \\J\subset_{\mathbf{\rho},\varepsilon}I\text{ and }J\subset I_{0}%
}}\frac{\left\vert \widehat{f}\left(  I\right)  \right\vert }{\sqrt{\left\vert
I_{\theta}\right\vert _{\sigma}}}\left\{  \left(  2^{-s}\right)
^{1-\varepsilon\left(  n+1-\lambda\right)  }\mathrm{P}_{\kappa_{2}}^{\alpha
}\left(  I_{0},\mathbf{1}_{I_{\theta}}\sigma\right)  \right\}  \sqrt
{\left\vert J\right\vert _{\omega}}\left\Vert \bigtriangleup_{J;\kappa_{2}%
}^{\omega}g\right\Vert _{L^{2}\left(  \omega\right)  }\,.
\end{align*}
By Lemma \ref{doub piv} and the Muckenhoupt condition, this is
\[
\lesssim\sqrt{A_{2}^{\alpha}(\sigma,\omega)}\sum_{s=\mathbf{r}}^{\infty}%
\sum_{I\in\mathcal{C}_{F}}\sum_{\substack{I_{0},I_{\theta}\in\mathfrak{C}%
_{\mathcal{D}}\left(  I\right)  \\I_{0}\neq I_{\theta}}}\sum_{\substack{J\in
\mathcal{C}_{F}^{\mathbf{\tau}-\operatorname*{shift}}\text{ and }\ell\left(
J\right)  =2^{-s}\ell\left(  I\right)  \\J\subset_{\mathbf{\rho},\varepsilon
}I\text{ and }J\subset I_{0}}}\left\vert \widehat{f}\left(  I\right)
\right\vert \left\{  \left(  2^{-s}\right)  ^{1-\varepsilon\left(
n+1-\lambda\right)  }\frac{\sqrt{\left\vert J\right\vert _{\omega}}}%
{\sqrt{\left\vert I\right\vert _{\omega}}}\right\}  \left\Vert \bigtriangleup
_{J;\kappa_{2}}^{\omega}g\right\Vert _{L^{2}\left(  \omega\right)  }\,,
\]
and by Cauchy-Schwarz, this is at most
\begin{align*}
\sqrt{A_{2}^{\alpha}(\sigma,\omega)}  &  \sum_{s=\mathbf{r}}^{\infty}\left(
2^{-s}\right)  ^{1-\varepsilon\left(  n+1-\lambda\right)  }\left(  \sum
_{I\in\mathcal{C}_{F}}\sum_{\substack{I_{0},I_{\theta}\in\mathfrak{C}%
_{\mathcal{D}}\left(  I\right)  \\I_{0}\neq I_{\theta}}}\sum_{\substack{J\in
\mathcal{C}_{F}^{\mathbf{\tau}-\operatorname*{shift}}\text{ and }\ell\left(
J\right)  =2^{-s}\ell\left(  I\right)  \\J\subset_{\mathbf{\rho},\varepsilon
}I\text{ and }J\subset I_{0}}}\left\vert \widehat{f}\left(  I\right)
\right\vert ^{2}\frac{\left\vert J\right\vert _{\omega}}{\left\vert
I\right\vert _{\omega}}\right)  ^{\frac{1}{2}}\\
&  \times\left(  \sum_{I\in\mathcal{C}_{F}}\sum_{\substack{I_{0},I_{\theta}%
\in\mathfrak{C}_{\mathcal{D}}\left(  I\right)  \\I_{0}\neq I_{\theta}}%
}\sum_{\substack{J\in\mathcal{C}_{F}^{\mathbf{\tau}-\operatorname*{shift}%
}\text{ and }\ell\left(  J\right)  =2^{-s}\ell\left(  I\right)  \\J\subset
_{\mathbf{\rho},\varepsilon}I\text{ and }J\subset I_{0}}}\left\Vert
\bigtriangleup_{J;\kappa_{2}}^{\omega}g\right\Vert _{L^{2}\left(
\omega\right)  }\right)  ^{\frac{1}{2}}\,.
\end{align*}

Now we note that%
\[
\sum_{I\in\mathcal{C}_{F}}\sum_{\substack{I_{0},I_{\theta}\in\mathfrak{C}%
_{\mathcal{D}}\left(  I\right)  \\I_{0}\neq I_{\theta}}}\sum_{\substack{J\in
\mathcal{C}_{F}^{\mathbf{\tau}-\operatorname*{shift}}\text{ and }\ell\left(
J\right)  =2^{-s}\ell\left(  I\right)  \\J\subset_{\mathbf{\rho},\varepsilon
}I\text{ and }J\subset I_{0}}} \left\Vert \bigtriangleup_{J;\kappa_{2}%
}^{\omega}g\right\Vert _{L^{2}\left(  \omega\right)  }^{2} \lesssim\left\Vert
\mathsf{P}_{\mathcal{C}_{F}^{\mathbf{\tau}-\operatorname*{shift}}}^{\omega
}g\right\Vert _{L^{2}\left(  \omega\right)  }^{2}\ ,
\]
and
\[
\sum_{\substack{J\in\mathcal{C}_{F}^{\mathbf{\tau}-\operatorname*{shift}%
}\text{ and }\ell\left(  J\right)  =2^{-s}\ell\left(  I\right)  \\J\subset
_{\mathbf{\rho},\varepsilon}I\text{ and }J\subset I_{0}}}\frac{\left\vert
J\right\vert _{\omega}}{\left\vert I\right\vert _{\omega}}\lesssim1 \, ,
\]
and use $\left\Vert \mathsf{P}_{\mathcal{C}_{F}}^{\sigma}f\right\Vert
_{L^{2}\left(  \sigma\right)  }^{2}=\sum_{I\in\mathcal{C}_{F}}\left\vert
\widehat{f}\left(  I\right)  \right\vert ^{2}$, to obtain%

\[
\left\vert \mathsf{T}_{\operatorname*{neighbour}}^{\subset_{\rho,\varepsilon
},F}\right\vert \left(  f,g\right)  \lesssim\sqrt{A_{2}^{\alpha}}%
\sum_{s=\mathbf{r}}^{\infty}\left(  2^{-s}\right)  ^{1-\varepsilon\left(
n+1-\lambda\right)  }\left\Vert \mathsf{P}_{\mathcal{C}_{F}}^{\sigma
}\right\Vert _{L^{2}\left(  \sigma\right)  }\left\Vert \mathsf{P}%
_{\mathcal{C}_{F}^{\mathbf{\tau}-\operatorname*{shift}}}^{\omega}\right\Vert
_{L^{2}\left(  \omega\right)  }\lesssim\sqrt{A_{2}^{\alpha}}\left\Vert
\mathsf{P}_{\mathcal{C}_{F}}^{\sigma}\right\Vert _{L^{2}\left(  \sigma\right)
}\left\Vert \mathsf{P}_{\mathcal{C}_{F}^{\mathbf{\tau}-\operatorname*{shift}}%
}^{\omega}\right\Vert _{L^{2}\left(  \omega\right)  }\,.
\]

\subsection{The below stopping form}

We also show the \emph{below stopping} form converges absolutely, in the sense
that the below stopping form is bounded by the sublinear form
\begin{equation}
\left\vert \mathsf{T}_{\operatorname*{stop}}^{\subset_{\rho,\varepsilon}%
,F}\right\vert \left(  f,g\right)  \equiv\sum_{I\in\mathcal{C}_{F}}%
\sum_{I^{\prime}\in\mathfrak{C}_{\mathcal{D}}\left(  I\right)  }%
\sum_{\substack{J\in\mathcal{C}_{F}^{\mathbf{\tau}-\operatorname*{shift}%
}\\J\subset I^{\prime}\text{ and }J\subset_{\mathbf{\rho},\varepsilon}%
I}}\left\vert \left\langle M_{I^{\prime};\kappa_{1}}T_{\sigma}^{\alpha
}\mathbf{1}_{F\setminus I^{\prime}},\bigtriangleup_{J;\kappa_{2}}^{\omega
}g\right\rangle _{\omega}\right\vert \lesssim\sqrt{A_{2}^{\alpha}}\left\Vert
\mathsf{P}_{\mathcal{C}_{F}}^{\sigma}\right\Vert _{L^{2}\left(  \sigma\right)
}\left\Vert \mathsf{P}_{\mathcal{C}_{F}^{\mathbf{\tau}-\operatorname*{shift}}%
}^{\omega}\right\Vert _{L^{2}\left(  \omega\right)  }\,, \label{stop est}%
\end{equation}
using nearly the same proof as for (\ref{neigh est}), except that here we need
to use the extra vanishing moments of $\bigtriangleup_{J;\kappa_{2}}^{\omega
}g$ to penetrate past the polynomial $M_{I^{\prime};\kappa_{1}}$ of degree
$\kappa_{1}$ and reach the kernel $T_{\sigma}^{\alpha}$.

To bound the \emph{below stopping} form $\mathsf{T}_{\operatorname*{stop}%
}^{\subset_{\rho,\varepsilon},F}\left(  f,g\right)  $, we will use the
$\kappa_{1}$-pivotal condition together with a variant of the Haar stopping
form argument due to Nazarov, Treil and Volberg \cite{NTV4}. Most importantly
we will assume in this subsection that%
\begin{equation}
\kappa_{2}\geq2\kappa_{1}\text{.} \label{double}%
\end{equation}

Recall that%
\[
M_{I^{\prime}; \kappa} = \left\vert \widehat{f}\left(  I\right)  \right\vert
Q_{I^{\prime};\kappa}=\mathbb{E}_{I^{\prime};\kappa}^{\sigma}f-\mathbf{1}%
_{I^{\prime}}\mathbb{E}_{I;\kappa}^{\sigma}f.
\]
We begin the proof by pigeonholing the ratio of side lengths of $I$ and $J$ in
the stopping form:%
\begin{align*}
\left|  \mathsf{T}_{\operatorname*{stop}}^{\subset_{\rho,\varepsilon},F}
\right|  \left(  f,g\right)   &  = \sum_{I\in\mathcal{C}_{F}}\sum_{I^{\prime
}\in\mathfrak{C}_{\mathcal{D}}\left(  I\right)  }\sum_{\substack{J\in
\mathcal{C}_{F}^{\mathbf{\tau}-\operatorname*{shift}}\\J\subset I^{\prime
}\text{ and }J\subset_{\mathbf{\rho},\varepsilon}I}}\left\vert \widehat
{f}\left(  I\right)  \right\vert \left|  \left\langle Q_{I^{\prime};\kappa
_{1}}T_{\sigma}^{\alpha}\mathbf{1}_{F\setminus I^{\prime}},\bigtriangleup
_{J;\kappa_{2}}^{\omega}g\right\rangle _{\omega} \right| \\
&  =\sum_{s=0}^{\infty}\sum_{I\in\mathcal{C}_{F}}\sum_{I^{\prime}%
\in\mathfrak{C}_{\mathcal{D}}\left(  I\right)  }\sum_{\substack{J\in
\mathcal{C}_{F}^{\mathbf{\tau}-\operatorname*{shift}}\text{and }\ell\left(
J\right)  =2^{-s}\ell\left(  I\right)  \\J\subset I^{\prime}\text{ and
}J\subset_{\mathbf{\rho},\varepsilon}I}}\left\vert \widehat{f}\left(
I\right)  \right\vert \left|  \left\langle Q_{I^{\prime};\kappa_{1}}T_{\sigma
}^{\alpha}\mathbf{1}_{F\setminus I^{\prime}},\bigtriangleup_{J;\kappa_{2}%
}^{\omega}g\right\rangle _{\omega} \right|  \equiv\sum_{s=0}^{\infty} \left|
\mathsf{T}_{\operatorname*{stop},s}^{\subset_{\rho,\varepsilon},F} \right|
\left(  f,g\right)  \ .
\end{align*}
By (\ref{piv bound}) with $R=\frac{Q_{I^{\prime};\kappa_{1}}}{\left\Vert
Q_{I^{\prime};\kappa_{1}}\right\Vert _{\infty}}$, and then using $\left\Vert
Q_{I^{\prime};\kappa_{1}}\right\Vert _{\infty}\lesssim\frac{1}{\sqrt
{\left\vert I^{\prime}\right\vert _{\sigma}}}$, the Poisson inequality
(\ref{e.Jsimeq}), and the crucial assumption (\ref{double}), we obtain the
last display is at most a constant times
\begin{align}
&  \sum_{I\in\mathcal{C}_{F}}\sum_{I^{\prime}\in\mathfrak{C}_{\mathcal{D}%
}\left(  I\right)  }\sum_{\substack{J\in\mathcal{C}_{F}^{\mathbf{\tau
}-\operatorname*{shift}}\text{and }\ell\left(  J\right)  =2^{-s}\ell\left(
I\right)  \\J\subset I^{\prime}\text{ and }J\subset_{\mathbf{\rho}%
,\varepsilon}I}}\left\vert \widehat{f}\left(  I\right)  \right\vert \left\Vert
Q_{I^{\prime};\kappa_{1}}\right\Vert _{\infty}\ \mathrm{P}_{\kappa_{1}%
}^{\alpha}\left(  J,\mathbf{1}_{F\setminus I^{\prime}}\sigma\right)
\sqrt{\left\vert J\right\vert _{\omega}}\left\Vert \bigtriangleup
_{J;\kappa_{2}}^{\omega}g\right\Vert _{L^{2}\left(  \omega\right)
}\label{using crucial}\\
&  \lesssim\sum_{I\in\mathcal{C}_{F}}\sum_{I^{\prime}\in\mathfrak{C}%
_{\mathcal{D}}\left(  I\right)  }\sum_{\substack{J\in\mathcal{C}%
_{F}^{\mathbf{\tau}-\operatorname*{shift}}\text{and }\ell\left(  J\right)
=2^{-s}\ell\left(  I\right)  \\J\subset I^{\prime}\text{ and }J\subset
_{\mathbf{\rho},\varepsilon}I}}\left\vert \widehat{f}\left(  I\right)
\right\vert \frac{1}{\sqrt{\left\vert I^{\prime}\right\vert _{\sigma}}%
}\ \left(  2^{-s}\right)  ^{\kappa_{1}-\varepsilon\left(  n+\kappa_{1}%
-\lambda\right)  }\mathrm{P}_{\kappa_{1}}^{\alpha}\left(  I,\mathbf{1}%
_{F\setminus I^{\prime}}\sigma\right)  \sqrt{\left\vert J\right\vert _{\omega
}}\left\Vert \bigtriangleup_{J;\kappa_{2}}^{\omega}g\right\Vert _{L^{2}\left(
\omega\right)  }\nonumber\\
&  \lesssim\sum_{I\in\mathcal{C}_{F}}\sum_{I^{\prime}\in\mathfrak{C}%
_{\mathcal{D}}\left(  I\right)  }\sum_{\substack{J\in\mathcal{C}%
_{F}^{\mathbf{\tau}-\operatorname*{shift}}\text{and }\ell\left(  J\right)
=2^{-s}\ell\left(  I\right)  \\J\subset I^{\prime}\text{ and }J\subset
_{\mathbf{\rho},\varepsilon}I}}\left\vert \widehat{f}\left(  I\right)
\right\vert \frac{1}{\sqrt{\left\vert I^{\prime}\right\vert _{\sigma}}%
}\ \left(  2^{-s}\right)  ^{\kappa_{1}-\varepsilon\left(  n+\kappa_{1}%
-\lambda\right)  }\frac{\left\vert I\right\vert _{\sigma}}{\left\vert
I\right\vert ^{1-\frac{\alpha}{n}}}\sqrt{\left\vert J\right\vert _{\omega}%
}\left\Vert \bigtriangleup_{J;\kappa_{2}}^{\omega}g\right\Vert _{L^{2}\left(
\omega\right)  }\, .\nonumber
\end{align}
By Lemma \ref{doub piv} and the Muckenhoupt condition, this is
\begin{align}
&  \lesssim\sqrt{A_{2}^{\alpha}\left(  \sigma,\omega\right)  }\sum
_{I\in\mathcal{C}_{F}}\sum_{I^{\prime}\in\mathfrak{C}_{\mathcal{D}}\left(
I\right)  }\sum_{\substack{J\in\mathcal{C}_{F}^{\mathbf{\tau}%
-\operatorname*{shift}}\text{and }\ell\left(  J\right)  =2^{-s}\ell\left(
I\right)  \\J\subset I^{\prime}\text{ and }J\subset_{\mathbf{\rho}%
,\varepsilon}I}}\left\vert \widehat{f}\left(  I\right)  \right\vert
\frac{\sqrt{\left\vert J\right\vert _{\omega}}}{\sqrt{\left\vert I^{\prime
}\right\vert _{\omega}}}\ \left(  2^{-s}\right)  ^{\kappa_{1}-\varepsilon
\left(  n+\kappa_{1}-\lambda\right)  }\left\Vert \bigtriangleup_{J;\kappa_{2}%
}^{\omega}g\right\Vert _{L^{2}\left(  \omega\right)  }\nonumber\\
&  \lesssim\left(  2^{-s}\right)  ^{\kappa_{1}-\varepsilon\left(  n+\kappa
_{1}-\lambda\right)  }\sqrt{A_{2}^{\alpha}\left(  \sigma,\omega\right)  }%
\sqrt{\sum_{I\in\mathcal{C}_{F}}\sum_{I^{\prime}\in\mathfrak{C}_{\mathcal{D}%
}\left(  I\right)  }\sum_{\substack{J\in\mathcal{C}_{F}^{\mathbf{\tau
}-\operatorname*{shift}}\text{and }\ell\left(  J\right)  =2^{-s}\ell\left(
I\right)  \\J\subset I^{\prime}\text{ and }J\subset_{\mathbf{\rho}%
,\varepsilon}I}}\left\vert \widehat{f}\left(  I\right)  \right\vert ^{2}%
\frac{\left\vert J\right\vert _{\omega}}{\left\vert I^{\prime}\right\vert
_{\omega}}}\nonumber\\
&
\ \ \ \ \ \ \ \ \ \ \ \ \ \ \ \ \ \ \ \ \ \ \ \ \ \ \ \ \ \ \ \ \ \ \ \times
\sqrt{\sum_{I\in\mathcal{C}_{F}}\sum_{I^{\prime}\in\mathfrak{C}_{\mathcal{D}%
}\left(  I\right)  }\sum_{\substack{J\in\mathcal{C}_{F}^{\mathbf{\tau
}-\operatorname*{shift}}\text{and }\ell\left(  J\right)  =2^{-s}\ell\left(
I\right)  \\J\subset I^{\prime}\text{ and }J\subset_{\mathbf{\rho}%
,\varepsilon}I}}\left\Vert \bigtriangleup_{J;\kappa_{2}}^{\omega}g\right\Vert
_{L^{2}\left(  \omega\right)  }^{2}}\ .\nonumber
\end{align}
Now we note that%
\[
\sum_{I\in\mathcal{C}_{F}}\sum_{I^{\prime}\in\mathfrak{C}_{\mathcal{D}}\left(
I\right)  }\sum_{\substack{J\in\mathcal{C}_{F}^{\mathbf{\tau}%
-\operatorname*{shift}}\text{and }\ell\left(  J\right)  =2^{-s}\ell\left(
I\right)  \\J\subset I^{\prime}\text{ and }J\subset_{\mathbf{\rho}%
,\varepsilon}I}}\left\Vert \bigtriangleup_{J;\kappa_{2}}^{\omega}g\right\Vert
_{L^{2}\left(  \omega\right)  }^{2}\lesssim\left\Vert \mathsf{P}%
_{\mathcal{C}_{F}^{\mathbf{\tau}-\operatorname*{shift}}}^{\omega}g\right\Vert
_{L^{2}\left(  \omega\right)  }^{2}\ ,
\]
and use $\left\Vert \mathsf{P}_{\mathcal{C}_{F}}^{\sigma}f\right\Vert
_{L^{2}\left(  \sigma\right)  }^{2}=\sum_{I\in\mathcal{C}_{F}}\left\vert
\widehat{f}\left(  I\right)  \right\vert ^{2}$, and $\sum_{\substack{J\in
\mathcal{C}_{F}^{\mathbf{\tau}-\operatorname*{shift}}\text{and }\ell\left(
J\right)  =2^{-s}\ell\left(  I\right)  \\J\subset I^{\prime}\text{ and
}J\subset_{\mathbf{\rho},\varepsilon}I}}\frac{\left\vert J\right\vert
_{\omega}}{\left\vert I^{\prime}\right\vert _{\omega}}\leq1$, to obtain%
\[
\left\vert \mathsf{T}_{\operatorname*{stop},s}^{\subset_{\rho,\varepsilon},F}
\right\vert \left(  f,g\right)  \lesssim\left(  2^{-s}\right)  ^{\kappa
_{1}-\varepsilon\left(  n+\kappa_{1}-\alpha\right)  }\sqrt{A_{2}^{\alpha
}\left(  \sigma,\omega\right)  }\left\Vert \mathsf{P}_{\mathcal{C}_{F}%
}^{\sigma}f\right\Vert _{L^{2}\left(  \sigma\right)  }\left\Vert
\mathsf{P}_{\mathcal{C}_{F}^{\mathbf{\tau}-\operatorname*{shift}}}^{\omega
}g\right\Vert _{L^{2}\left(  \omega\right)  }.
\]
Finally then we sum in $s$ to obtain%
\begin{align}
&  \ \ \ \ \ \ \ \ \ \ \ \ \ \ \ \left\vert \mathsf{T}_{\operatorname*{stop}%
}^{\subset_{\rho,\varepsilon},F} \right\vert \left(  f,g\right)  \leq
\sqrt{A_{2}^{\alpha}\left(  \sigma,\omega\right)  }\sum_{s=0}^{\infty
}\left\vert \mathsf{T}_{\operatorname*{stop},s}^{\subset_{\rho,\varepsilon}%
,F}\left(  f,g\right)  \right\vert \lesssim\sqrt{A_{2}^{\alpha}\left(
\sigma,\omega\right)  }\left\Vert \mathsf{P}_{\mathcal{C}_{F}}^{\sigma
}f\right\Vert _{L^{2}\left(  \sigma\right)  }\left\Vert \mathsf{P}%
_{\mathcal{C}_{F}^{\mathbf{\tau}-\operatorname*{shift}}}^{\omega}g\right\Vert
_{L^{2}\left(  \omega\right)  },\nonumber
\end{align}
if we take $0<\varepsilon<\frac{\kappa_{1}}{n+\kappa_{1}-\lambda}$.

\section{Above diagonal form and the parallel corona}

In order to treat the \emph{above} form $\mathsf{B}_{\supset_{\mathbf{\rho
},\varepsilon}}\left(  f,g\right)  $, we again start with the \emph{Canonical
Splitting} of $\mathsf{B}_{\supset_{\mathbf{\rho},\varepsilon}}\left(
f,g\right)  $ in (\ref{cube size}) just as we did for the \emph{below} form
$\mathsf{B}_{\subset_{\mathbf{\rho},\varepsilon}}\left(  f,g\right)  $,%
\begin{align*}
\mathsf{B}_{\supset_{\mathbf{\rho},\varepsilon}}\left(  f,g\right)   &
=\sum_{G\in\mathcal{G}}\left\langle T_{\sigma}\left(  \mathsf{P}%
_{\mathcal{C}_{G}}^{\sigma}f\right)  ,\mathsf{P}_{\mathcal{C}_{G}%
^{\mathbf{\tau}-\operatorname*{shift}}}^{\omega}g\right\rangle _{\omega
}^{\supset_{\mathbf{\rho},\varepsilon}}+\sum_{\substack{F,G\in\mathcal{G}%
\\F\subsetneqq G}}\left\langle T_{\sigma}\left(  \mathsf{P}_{\mathcal{C}_{F}%
}^{\sigma}f\right)  ,\mathsf{P}_{\mathcal{C}_{G}^{\mathbf{\tau}%
-\operatorname*{shift}}}^{\omega}g\right\rangle _{\omega}^{\supset
_{\mathbf{\rho},\varepsilon}}\\
&  +\sum_{\substack{F,G\in\mathcal{G}\\F\supsetneqq G}}\left\langle T_{\sigma
}\left(  \mathsf{P}_{\mathcal{C}_{F}}^{\sigma}f\right)  ,\mathsf{P}%
_{\mathcal{C}_{G}^{\mathbf{\tau}-\operatorname*{shift}}}^{\omega
}g\right\rangle _{\omega}^{\supset_{\mathbf{\rho},\varepsilon}}+\sum
_{\substack{F,G\in\mathcal{G}\\F\cap G=\emptyset}}\left\langle T_{\sigma
}\left(  \mathsf{P}_{\mathcal{C}_{F}}^{\sigma}f\right)  ,\mathsf{P}%
_{\mathcal{C}_{G}^{\mathbf{\tau}-\operatorname*{shift}}}^{\omega
}g\right\rangle _{\omega}^{\supset_{\mathbf{\rho},\varepsilon}}\\
&  \equiv\mathsf{T}_{\operatorname*{diagonal}}^{\supset_{\mathbf{\rho
},\varepsilon}}\left(  f,g\right)  +\mathsf{T}_{\operatorname*{far}%
\operatorname*{below}}^{\supset_{\mathbf{\rho},\varepsilon}}\left(
f,g\right)  +\mathsf{T}_{\operatorname*{far}\operatorname*{above}}%
^{\supset_{\mathbf{\rho},\varepsilon}}\left(  f,g\right)  +\mathsf{T}%
_{\operatorname*{disjoint}}^{\supset_{\mathbf{\rho},\varepsilon}}\left(
f,g\right)  ,
\end{align*}
where the underlying Alpert projections defining $\mathsf{P}_{\mathcal{C}_{F}%
}^{\sigma}f$ and $\mathsf{P}_{\mathcal{C}_{F}^{\mathbf{\tau}%
-\operatorname*{shift}}}^{\omega}g$ are $\bigtriangleup_{I;\kappa_{1}}%
^{\sigma}f$ and $\bigtriangleup_{J;\kappa_{2}}^{\sigma}f$ respectively with
$\kappa_{2}\geq2\kappa_{1}$ and $\kappa_{1}\geq\kappa$. Since the treatment of
the forms $\mathsf{T}_{\operatorname*{far}\operatorname*{below}}%
^{\subset_{\mathbf{\rho},\varepsilon}}\left(  f,g\right)  $, $\mathsf{T}%
_{\operatorname*{far}\operatorname*{above}}^{\subset_{\mathbf{\rho
},\varepsilon}}\left(  f,g\right)  $ and $\mathsf{T}_{\operatorname*{disjoint}%
}^{\subset_{\mathbf{\rho},\varepsilon}}\left(  f,g\right)  $ in the case of
the below form $\mathsf{B}_{\subset_{\mathbf{\rho},\varepsilon}}\left(
f,g\right)  $ did not use the crucial assumption (\ref{double}), but only
$\kappa_{1},\kappa_{2}\geq\kappa$, the same arguments used there bound the
forms $\mathsf{T}_{\operatorname*{far}\operatorname*{below}}^{\supset
_{\mathbf{\rho},\varepsilon}}\left(  f,g\right)  $, $\mathsf{T}%
_{\operatorname*{far}\operatorname*{above}}^{\supset_{\mathbf{\rho
},\varepsilon}}\left(  f,g\right)  $ and $\mathsf{T}_{\operatorname*{disjoint}%
}^{\supset_{\mathbf{\rho},\varepsilon}}\left(  f,g\right)  $ here using only
that $\kappa_{1},\kappa_{2}\geq\kappa$. This leaves just the \emph{above
diagonal} form $\mathsf{T}_{\operatorname*{diagonal}}^{\supset_{\mathbf{\rho
},\varepsilon}}\left(  f,g\right)  $ to be bounded, where we can no longer use
the NTV reach since the stopping form associated with $\mathsf{T}%
_{\operatorname*{diagonal}}^{\supset_{\mathbf{\rho},\varepsilon}}\left(
f,g\right)  $ requires the crucial assumption $\kappa_{1}\geq2\kappa_{2}$,
which we do \emph{not} have because of (\ref{double}). Instead, we must
manipulate the \emph{above diagonal} form $\mathsf{T}%
_{\operatorname*{diagonal}}^{\supset_{\mathbf{\rho},\varepsilon}}\left(
f,g\right)  $ by adding `pieces' of the \emph{below} form $\mathsf{B}%
_{\operatorname{below}}^{\subset_{\mathbf{\rho},\varepsilon}}\left(
f,g\right)  $ to create a parallel corona like decomposition to which we can
apply the indicator / cube testing condition.

Here is a brief schematic diagram of the decompositions used for the
\emph{above diagonal} form $\mathsf{T}_{\operatorname*{diagonal}}%
^{\supset_{\mathbf{\rho},\varepsilon}}\left(  f,g\right)  $,%
\[
\fbox{$%
\begin{array}
[c]{ccccc}%
\mathsf{T}_{\operatorname*{diagonal}}^{\supset_{\mathbf{\rho},\varepsilon}%
}\left(  f,g\right)  &  &  &  & \\
\downarrow &  &  &  & \\
\mathsf{T}_{\operatorname*{diagonal}}^{\supset_{\mathbf{\rho},\varepsilon
},\operatorname*{parallel}}\left(  f,g\right)  & - & \mathsf{T}%
_{\operatorname*{diagonal}}^{\subset_{\mathbf{\rho},\varepsilon}%
,\operatorname*{alt}}\left(  f,g\right)  & - & \operatorname{Error}\left(
f,g\right) \\
&  & \downarrow &  & \\
&  & \mathsf{T}_{\operatorname*{diagonal}}^{\subset_{\mathbf{\rho}%
,\varepsilon}}\left(  f,g\right)  & - & \mathsf{T}_{\operatorname*{diagonal}%
}^{\subset_{\mathbf{\rho},\varepsilon},\operatorname{diff}}\left(  f,g\right)
\end{array}
$}.
\]
The decoupled\emph{ parallel above} form $\mathsf{T}_{\operatorname*{diagonal}%
}^{\supset_{\mathbf{\rho},\varepsilon},\operatorname*{parallel}}\left(
f,g\right)  $ will be controlled by the dual indicator / cube testing
condition. The \emph{diagonal below }form $\mathsf{T}%
_{\operatorname*{diagonal}}^{\subset_{\mathbf{\rho},\varepsilon}}\left(
f,g\right)  $ has already been bounded in the previous section, and the
\emph{difference diagonal below} form $\mathsf{T}_{\operatorname*{diagonal}%
}^{\subset_{\mathbf{\rho},\varepsilon},\operatorname{diff}}\left(  f,g\right)
$ will be handled here in a\ virtually identical way. Finally, the error form
$\operatorname{Error1}\left(  f,g\right)  $ will be controlled by sublinear
variants of the comparable and disjoint forms.

\subsection{Bounding the above diagonal form}

We write,%
\[
\mathsf{T}_{\operatorname*{diagonal}}^{\supset_{\mathbf{\rho},\varepsilon}%
}\left(  f,g\right)  =\sum_{G\in\mathcal{G}}\left\langle T_{\sigma}^{\alpha
}\left(  \mathsf{P}_{\mathcal{C}_{G}}^{\sigma}f\right)  ,\mathsf{P}%
_{\mathcal{C}_{G}^{\mathbf{\tau}-\operatorname*{shift}}}^{\omega
}g\right\rangle _{\omega}^{\supset_{\mathbf{\rho},\varepsilon}}=\sum
_{J\in\mathcal{C}_{G}}\sum_{I\in\mathcal{C}_{G}^{\mathbf{\tau}%
-\operatorname*{shift}}\text{ and }I\subset_{\rho,\varepsilon}J}\left\langle
T_{\omega}^{\alpha,\ast}\bigtriangleup_{J;\kappa_{2}}^{\omega}g,\bigtriangleup
_{I;\kappa_{1}}^{\sigma}f\right\rangle _{\sigma}\ ,
\]
where we now have $\kappa_{1}\leq\frac{1}{2}\kappa_{2}<\kappa_{2}$, which
prevents us from bounding the stopping form if we were to use the NTV reach.
Instead, for each $G\in\mathcal{G}$, we will add the missing inner products
$\left\langle T_{\omega}^{\ast}\bigtriangleup_{J;\kappa_{2}}^{\omega
}g,\bigtriangleup_{I;\kappa_{1}}^{\sigma}f\right\rangle _{\sigma}$ to
$\mathsf{T}_{\operatorname*{diagonal}}^{\supset_{\mathbf{\rho},\varepsilon}%
,G}\left(  f,g\right)  $ that are needed to result in the \emph{decoupled
parallel} form,%
\begin{align*}
\mathsf{T}_{\operatorname*{diagonal}}^{\supset_{\mathbf{\rho},\varepsilon
},\operatorname*{parallel}}\left(  f,g\right)   &  \equiv\sum_{G\in
\mathcal{G}}\sum_{J\in\mathcal{C}_{G}}\sum_{I\in\mathcal{C}_{G}^{\mathbf{\tau
}-\operatorname*{shift}}}\left\langle T_{\omega}^{\alpha,\ast}\bigtriangleup
_{J;\kappa_{2}}^{\omega}g,\bigtriangleup_{I;\kappa_{1}}^{\sigma}f\right\rangle
_{\sigma},\\
&  =\left\langle T_{\omega}^{\alpha,\ast}\sum_{J\in\mathcal{C}_{G}%
}\bigtriangleup_{J;\kappa_{2}}^{\omega}g,\sum_{I\in\mathcal{C}_{G}%
^{\mathbf{\tau}-\operatorname*{shift}}}\bigtriangleup_{I;\kappa_{1}}^{\sigma
}f\right\rangle _{\sigma}=\left\langle T_{\omega}^{\alpha,\ast}\mathsf{P}%
_{\mathcal{C}_{G}}^{\omega}g,\mathsf{P}_{\mathcal{C}_{G}^{\mathbf{\tau
}-\operatorname*{shift}}}^{\sigma}f\right\rangle _{\sigma}.
\end{align*}

However, we compute%
\begin{align*}
&  \mathsf{T}_{\operatorname*{diagonal}}^{\supset_{\mathbf{\rho},\varepsilon
},\operatorname*{parallel}}\left(  f,g\right)  -\mathsf{T}%
_{\operatorname*{diagonal}}^{\supset_{\mathbf{\rho},\varepsilon}}\left(
f,g\right)  =\sum_{G\in\mathcal{G}}\sum_{J\in\mathcal{C}_{G}}\sum
_{I\in\mathcal{C}_{G}^{\mathbf{\tau}-\operatorname*{shift}}\text{ and
}I\not \subset _{\rho,\varepsilon}J}\left\langle T_{\omega}^{\alpha,\ast
}\bigtriangleup_{J;\kappa_{2}}^{\omega}g,\bigtriangleup_{I;\kappa_{1}}%
^{\sigma}f\right\rangle _{\sigma}\\
&  =\sum_{G\in\mathcal{G}}\sum_{J\in\mathcal{C}_{G}}\sum_{\substack{I\in
\mathcal{C}_{G}^{\mathbf{\tau}-\operatorname*{shift}}\\2^{-\rho}\leq\frac
{\ell\left(  J\right)  }{\ell\left(  I\right)  }\leq2^{\rho}\text{ or }J\cap
I=\emptyset}}\left\langle T_{\sigma}^{\alpha}\left(  \bigtriangleup
_{I;\kappa_{1}}^{\sigma}f\right)  ,\bigtriangleup_{J;\kappa_{2}}^{\omega
}g\right\rangle _{\omega}+\sum_{G\in\mathcal{G}}\sum_{J\in\mathcal{C}_{G}}%
\sum_{I\in\mathcal{C}_{G}^{\mathbf{\tau}-\operatorname*{shift}}\text{ and
}J\subset_{\rho,\varepsilon}I}\left\langle T_{\sigma}^{\alpha}\left(
\bigtriangleup_{I;\kappa_{1}}^{\sigma}f\right)  ,\bigtriangleup_{J;\kappa_{2}%
}^{\omega}g\right\rangle _{\omega}\\
&  \equiv\operatorname{Error}\left(  f,g\right)  +\mathsf{T}%
_{\operatorname*{diagonal}}^{\subset_{\mathbf{\rho},\varepsilon}%
,\operatorname*{alt}}\left(  f,g\right)  \ ,
\end{align*}
where the sublinear form%
\[
\left\vert \operatorname{Error}\right\vert \left(  f,g\right)  \equiv
\sum_{G\in\mathcal{G}}\sum_{J\in\mathcal{C}_{G}}\sum_{\substack{I\in
\mathcal{C}_{G}^{\mathbf{\tau}-\operatorname*{shift}}\\2^{-\rho}\leq\frac
{\ell\left(  J\right)  }{\ell\left(  I\right)  }\leq2^{\rho}\text{ or }J\cap
I=\emptyset}}\left\vert \left\langle T_{\sigma}^{\alpha}\left(  \bigtriangleup
_{I;\kappa_{1}}^{\sigma}f\right)  ,\bigtriangleup_{J;\kappa_{2}}^{\omega
}g\right\rangle _{\omega}\right\vert
\]
is controlled by the sublinear \emph{comparable} form $\left\vert
\mathsf{B}_{\diagup}\right\vert \left(  f,g\right)  $ with absolute values
inside, together with the sublinear disjoint form $\left\vert \mathsf{B}%
_{\cap}\right\vert \left(  f,g\right)  $ with absolute values inside.

The \emph{alternate below diagonal} form%
\begin{align*}
\mathsf{T}_{\operatorname*{diagonal}}^{\subset_{\mathbf{\rho},\varepsilon
},\operatorname*{alt}}\left(  f,g\right)   &  \equiv\sum_{G\in\mathcal{G}}%
\sum_{J\in\mathcal{C}_{G}}\sum_{I\in\mathcal{C}_{G}^{\mathbf{\tau
}-\operatorname*{shift}}\text{ and }J\subset_{\rho,\varepsilon}I}\left\langle
T_{\sigma}^{\alpha}\left(  \bigtriangleup_{I;\kappa_{1}}^{\sigma}f\right)
,\bigtriangleup_{J;\kappa_{2}}^{\omega}g\right\rangle _{\omega}\\
&  =\sum_{G\in\mathcal{G}}\sum_{I\in\mathcal{C}_{G}^{\mathbf{\tau
}-\operatorname*{shift}}}\sum_{J\in\mathcal{C}_{G}\text{ and }J\subset
_{\rho,\varepsilon}I}\left\langle T_{\sigma}^{\alpha}\left(  \bigtriangleup
_{I;\kappa_{1}}^{\sigma}f\right)  ,\bigtriangleup_{J;\kappa_{2}}^{\omega
}g\right\rangle _{\omega}\ ,
\end{align*}
was essentially treated in the previous section on the below form by breaking
it up using the NTV reach plus error terms. To see this we write
\[
\mathcal{C}_{F}^{\mathbf{\tau}-\operatorname*{shift}}=\left(  \mathcal{C}%
_{\pi_{\mathcal{G}}F}\cap\mathcal{C}_{F}^{\mathbf{\tau}-\operatorname*{shift}%
}\right)  \bigcup\left(  \bigcup_{\substack{G\in\mathcal{G}:\ \left(
F,G\right)  \in\operatorname{Near}\left(  \mathcal{F}\times\mathcal{G}\right)
\\G\subset F}}\mathcal{C}_{G}\cap\mathcal{C}_{F}^{\mathbf{\tau}%
-\operatorname*{shift}}\bigcup_{G\in\mathcal{G}:\ G\in\bigcup_{F^{\prime}%
\in\mathfrak{C}_{\mathcal{F}}\left(  F\right)  }\mathcal{N}^{\tau}\left(
F^{\prime}\right)  }\mathcal{C}_{G}\cap\mathcal{C}_{F}^{\mathbf{\tau
}-\operatorname*{shift}}\right)
\]
where $\pi_{\mathcal{G}}F$ is the smallest cube $G\in\mathcal{G}$ containing
$F$. Then fixing a pair $\left(  F,J\right)  $ with $F\in\mathcal{F}$ and
$J\in\mathcal{C}_{F}^{\mathbf{\tau}-\operatorname*{shift}}$, and noting Remark
\ref{pd}, the set of cubes $I$ arising in the sum for $\mathsf{T}%
_{\operatorname*{diagonal}}^{\subset_{\mathbf{\rho},\varepsilon}%
,\operatorname*{alt}}\left(  f,g\right)  $ are precisely those $I\in
\mathcal{C}_{\pi_{\mathcal{G}}J}^{\mathbf{\tau}-\operatorname*{shift}}$
satisfying $J\subset_{\rho,\varepsilon}I$, i.e.
\[
\pi_{\mathcal{D}}^{\left(  \rho\right)  }J\subset I\subset\pi_{\mathcal{G}%
}^{\left[  \tau\right]  }J,
\]
where $\pi_{\mathcal{G}}^{\left[  \tau\right]  }J$ is the $\tau$-grandchild of
$\pi_{\mathcal{G}}J$ that contains $J$, or more precisely, is the unique cube
$K$ in $\mathcal{D}$ such that $J\subset K$ and $\pi_{\mathcal{G}}^{\left(
\tau\right)  }K=\pi_{\mathcal{G}}J$. Using Remark \ref{pd} we have%
\begin{align*}
\mathsf{T}_{\operatorname*{diagonal}}^{\subset_{\mathbf{\rho},\varepsilon
},\operatorname*{alt}}\left(  f,g\right)   &  =\sum_{G\in\mathcal{G}}%
\sum_{J\in\mathcal{C}_{G}}\sum_{I\in\mathcal{C}_{G}^{\mathbf{\tau
}-\operatorname*{shift}}\text{ and }J\subset_{\rho,\varepsilon}I}\left\langle
T_{\sigma}^{\alpha}\left(  \bigtriangleup_{I;\kappa_{1}}^{\sigma}f\right)
,\bigtriangleup_{J;\kappa_{2}}^{\omega}g\right\rangle _{\omega}\\
&  =\sum_{F\in\mathcal{F}}\sum_{J\in\mathcal{C}_{F}^{\mathbf{\tau
}-\operatorname*{shift}}}\left\langle T_{\sigma}^{\alpha}\left(  \sum
_{I\in\left[  \pi_{\mathcal{D}}^{\left(  \rho\right)  }J,\pi_{\mathcal{G}%
}^{\left[  \tau\right]  }J\right]  }\bigtriangleup_{I;\kappa_{1}}^{\sigma
}f\right)  ,\bigtriangleup_{J;\kappa_{2}}^{\omega}g\right\rangle _{\omega},
\end{align*}
where if $K\supset L$ then $\sum_{I\in\left[  K,L\right]  }\bigtriangleup
_{I;\kappa_{1}}^{\sigma}f=0$.

The form $\mathsf{T}_{\operatorname*{diagonal}}^{\subset_{\mathbf{\rho
},\varepsilon},\operatorname*{alt}}\left(  f,g\right)  $ matches the form,
\begin{align*}
&  \mathsf{T}_{\operatorname*{diagonal}}^{\subset_{\mathbf{\rho},\varepsilon}%
}\left(  f,g\right)  =\sum_{F\in\mathcal{F}}\left\langle T_{\sigma}^{\alpha
}\left(  \mathsf{P}_{\mathcal{C}_{F}}^{\sigma}f\right)  ,\mathsf{P}%
_{\mathcal{C}_{F}^{\mathbf{\tau}-\operatorname*{shift}}}^{\omega
}g\right\rangle _{\omega}^{\subset_{\mathbf{\rho},\varepsilon}}=\sum
_{F\in\mathcal{F}}\sum_{I\in\mathcal{C}_{F}}\sum_{J\in\mathcal{C}%
_{F}^{\mathbf{\tau}-\operatorname*{shift}}\text{ and }J\subset_{\rho
,\varepsilon}I}\left\langle T_{\sigma}^{\alpha}\left(  \bigtriangleup
_{I;\kappa_{1}}^{\sigma}f\right)  ,\bigtriangleup_{J;\kappa_{2}}^{\omega
}g\right\rangle _{\omega}\\
&  =\sum_{F\in\mathcal{F}}\sum_{J\in\mathcal{C}_{F}^{\mathbf{\tau
}-\operatorname*{shift}}}\left\langle T_{\sigma}^{\alpha}\left(  \sum
_{I\in\mathcal{C}_{F}\text{ and }J\subset_{\rho,\varepsilon}I}\bigtriangleup
_{I;\kappa_{1}}^{\sigma}f\right)  ,\bigtriangleup_{J;\kappa_{2}}^{\omega
}g\right\rangle _{\omega}=\sum_{F\in\mathcal{F}}\sum_{J\in\mathcal{C}%
_{F}^{\mathbf{\tau}-\operatorname*{shift}}}\left\langle T_{\sigma}^{\alpha
}\left(  \sum_{I\in\left[  \pi_{\mathcal{D}}^{\left(  \rho\right)
}J,F\right]  }\bigtriangleup_{I;\kappa_{1}}^{\sigma}f\right)  ,\bigtriangleup
_{J;\kappa_{2}}^{\omega}g\right\rangle _{\omega},
\end{align*}
considered in the previous section,\ with the only exception that the sum in
$I$ stops at $\pi_{\mathcal{G}}^{\left[  \tau\right]  }J$ in $\mathsf{T}%
_{\operatorname*{diagonal}}^{\subset_{\mathbf{\rho},\varepsilon}%
,\operatorname*{alt}}\left(  f,g\right)  $, instead of at $F$ as it does in
$\mathsf{T}_{\operatorname*{diagonal}}^{\subset_{\mathbf{\rho},\varepsilon}%
}\left(  f,g\right)  $.

This suggests we decompose the form $\mathsf{T}_{\operatorname*{diagonal}%
}^{\subset_{\mathbf{\rho},\varepsilon},\operatorname*{alt}}\left(  f,g\right)
$\ as%
\begin{align}
\mathsf{T}_{\operatorname*{diagonal}}^{\subset_{\mathbf{\rho},\varepsilon
},\operatorname*{alt}}\left(  f,g\right)   &  =\sum_{F\in\mathcal{F}}%
\sum_{J\in\mathcal{C}_{F}^{\mathbf{\tau}-\operatorname*{shift}}}\left\langle
T_{\sigma}^{\alpha}\left(  \sum_{I\in\left[  \pi_{\mathcal{D}}^{\left(
\rho\right)  }J,\pi_{\mathcal{G}}^{\left[  \tau\right]  }J\right]
}\bigtriangleup_{I;\kappa_{1}}^{\sigma}f\right)  ,\bigtriangleup_{J;\kappa
_{2}}^{\omega}g\right\rangle _{\omega}\label{suggests}\\
&  =\sum_{F\in\mathcal{F}}\sum_{J\in\mathcal{C}_{F}^{\mathbf{\tau
}-\operatorname*{shift}}}\left\langle T_{\sigma}^{\alpha}\left(  \sum
_{I\in\left[  \pi_{\mathcal{D}}^{\left(  \rho\right)  }J,F\right]  \text{ if
}\pi_{\mathcal{G}}^{\left[  \tau\right]  }J\subset F}\bigtriangleup
_{I;\kappa_{1}}^{\sigma}f\right)  ,\bigtriangleup_{J;\kappa_{2}}^{\omega
}g\right\rangle _{\omega}\nonumber\\
&  -\sum_{F\in\mathcal{F}}\sum_{J\in\mathcal{C}_{F}^{\mathbf{\tau
}-\operatorname*{shift}}}\left\langle T_{\sigma}^{\alpha}\left(
\sum_{\substack{I\in\left[  \pi_{\mathcal{G}}^{\left[  \tau\right]
}J,F\right]  \text{ and }J\subset_{\rho,\varepsilon}I\\\pi_{\mathcal{G}%
}^{\left[  \tau\right]  }J\subset F}}\bigtriangleup_{I;\kappa_{1}}^{\sigma
}f\right)  ,\bigtriangleup_{J;\kappa_{2}}^{\omega}g\right\rangle _{\omega
}\nonumber\\
&  +\sum_{F\in\mathcal{F}}\sum_{J\in\mathcal{C}_{F}^{\mathbf{\tau
}-\operatorname*{shift}}}\left\langle T_{\sigma}^{\alpha}\left(
\sum_{\substack{I\in\left[  F,\pi_{\mathcal{G}}^{\left[  \tau\right]
}J\right]  \text{ and }J\subset_{\rho,\varepsilon}I\\F\subsetneqq
\pi_{\mathcal{G}}^{\left[  \tau\right]  }J}}\bigtriangleup_{I;\kappa_{1}%
}^{\sigma}f\right)  ,\bigtriangleup_{J;\kappa_{2}}^{\omega}g\right\rangle
_{\omega}\nonumber\\
&  \equiv\mathsf{T}_{\operatorname*{diagonal}}^{\subset_{\mathbf{\rho
},\varepsilon}}\left(  f,g\right)  -\mathsf{T}_{\operatorname*{diagonal}%
}^{\subset_{\mathbf{\rho},\varepsilon},\operatorname{diff}%
\operatorname{bottom}}\left(  f,g\right)  +\mathsf{T}%
_{\operatorname*{diagonal}}^{\subset_{\mathbf{\rho},\varepsilon}%
,\operatorname{diff}\operatorname{top}}\left(  f,g\right)  .\nonumber
\end{align}
We now claim that the \emph{ bottom difference} form $\mathsf{T}%
_{\operatorname*{diagonal}}^{\subset_{\mathbf{\rho},\varepsilon}%
,\operatorname{diff}\operatorname{bottom}}\left(  f,g\right)  $ can be treated
using the NTV reach in \emph{the same way} that the form $\mathsf{T}%
_{\operatorname*{diagonal}}^{\subset_{\mathbf{\rho},\varepsilon}}\left(
f,g\right)  $ was treated using the crucial assumption (\ref{double}). Indeed,
by the absolute convergence of the commutator, stopping and neighbor forms
given by (\ref{comm est}), (\ref{stop est}) and (\ref{neigh est}), it suffices
to bound the resulting paraproduct form for $\mathsf{T}%
_{\operatorname*{diagonal}}^{\subset_{\mathbf{\rho},\varepsilon}%
,\operatorname{diff}\operatorname{bottom}}\left(  f,g\right)  $ given by
\[
\sum_{J\in\mathcal{C}_{F}^{\mathbf{\tau}-\operatorname*{shift}}}
\sum_{\substack{I\in\left[  \pi_{\mathcal{G}}^{\left[  \tau\right]
}J,F\right]  \text{ and }J\subset_{\rho,\varepsilon}I\\\pi_{\mathcal{G}%
}^{\left[  \tau\right]  }J\subset F}} \left\langle M_{I_{J} ; \kappa_{1}}
T_{\sigma}^{\alpha} ,\bigtriangleup_{J;\kappa_{2}}^{\omega}g\right\rangle
_{\omega} \, .
\]
Because the sum in $I$ is along a tower of dyadic intervals, then as in
(\ref{eq:telescope_M_to_poly}), we may write
\[
\sum_{\substack{I\in\left[  \pi_{\mathcal{G}}^{\left[  \tau\right]
}J,F\right]  \cap\mathcal{C}_{F} \text{ and }J\subset_{\rho,\varepsilon}%
I\\\pi_{\mathcal{G}}^{\left[  \tau\right]  }J\subset F}} \mathbf{1}_{J}
M_{I^{\prime}; \kappa_{1}} \equiv\mathbf{1}_{J} P_{J ; \kappa_{1}} \ ,
\]
for some polynomial $P_{J ; \kappa_{1}}$ of degree strictly less than
$\kappa_{1}$. Then noting that $\pi_{\mathcal{G}} ^{[\tau]} J$ is either in
$\mathcal{C}_{F}$, or is in the corona associated to an $\mathcal{F}%
$-descendant of $F$ a bounded number of generations below, then one can check
that \eqref{unif bdd} still holds. Then the rest of the proof the paraproduct
term follows verbatim.

We now further decompose the \emph{top difference} form as a sum of a `plug'
form and a `hole' form,%
\begin{align*}
\mathsf{T}_{\operatorname*{diagonal}}^{\subset_{\mathbf{\rho},\varepsilon
},\operatorname{diff}\operatorname{top}}\left(  f,g\right)   &  =\sum
_{F\in\mathcal{F}}\sum_{J\in\mathcal{C}_{F}^{\mathbf{\tau}%
-\operatorname*{shift}}}\left\langle T_{\sigma}^{\alpha}\left(  \mathbf{1}%
_{F}\sum_{\substack{I\in\left[  F,\pi_{\mathcal{G}}^{\left[  \tau\right]
}J\right]  \text{ and }J\subset_{\rho,\varepsilon}I\\F\subsetneqq
\pi_{\mathcal{G}}^{\left[  \tau\right]  }J}}\bigtriangleup_{I;\kappa_{1}%
}^{\sigma}f\right)  ,\bigtriangleup_{J;\kappa_{2}}^{\omega}g\right\rangle
_{\omega}\\
&  +\sum_{F\in\mathcal{F}}\sum_{J\in\mathcal{C}_{F}^{\mathbf{\tau
}-\operatorname*{shift}}}\left\langle T_{\sigma}^{\alpha}\left(
\mathbf{1}_{\mathbb{R}^{n}\setminus F}\sum_{\substack{I\in\left[
F,\pi_{\mathcal{G}}^{\left[  \tau\right]  }J\right]  \text{ and }%
J\subset_{\rho,\varepsilon}I\\F\subsetneqq\pi_{\mathcal{G}}^{\left[
\tau\right]  }J}}\bigtriangleup_{I;\kappa_{1}}^{\sigma}f\right)
,\bigtriangleup_{J;\kappa_{2}}^{\omega}g\right\rangle _{\omega}\\
&  \equiv\mathsf{T}_{\operatorname*{diagonal}}^{\subset_{\mathbf{\rho
},\varepsilon},\operatorname{diff}\operatorname{top}\operatorname{plug}%
}\left(  f,g\right)  +\mathsf{T}_{\operatorname*{diagonal}}^{\subset
_{\mathbf{\rho},\varepsilon},\operatorname{diff}\operatorname{top}%
\operatorname{hole}}\left(  f,g\right)  ,
\end{align*}
where the `hole' form is handled in the same way as the \emph{below stopping
form} was handled above, since the $\omega$-vanishing moments up to order less
than $\kappa_{2}$ of the projection $\bigtriangleup_{J;\kappa_{2}}^{\omega}$
are here applied directly to the kernel of $T_{\sigma}^{\alpha}$. Indeed, we
now use,
\[
\mathrm{P}_{\kappa_{1}}^{\alpha}\left(  J,\mathbf{1}_{\mathbb{R}^{n}\setminus
F}\sigma\right)  \lesssim\left(  \frac{\ell\left(  J\right)  }{\ell\left(
F\right)  }\right)  ^{\varepsilon}\mathrm{P}_{\kappa_{1}}^{\alpha}\left(
F,\mathbf{1}_{\mathbb{R}^{n}\setminus F}\sigma\right)  \lesssim\left(
\frac{\ell\left(  J\right)  }{\ell\left(  F\right)  }\right)  ^{\varepsilon
}\frac{\left\vert F\right\vert _{\sigma}}{\left\vert F\right\vert
^{1-\frac{\alpha}{n}}},
\]
together with the telescoping identity,%
\[
\left\vert \sum_{\substack{I\in\left(  F,\pi_{\mathcal{G}}^{\left[
\tau\right]  }J\right]  \text{ and }J\subset_{\rho,\varepsilon}I\\F\subsetneqq
\pi_{\mathcal{G}}^{\left[  \tau\right]  }J}}\bigtriangleup_{I;\kappa_{1}%
}^{\sigma}f\right\vert =\left\vert \mathbb{E}_{F}^{\sigma}f-\mathbb{E}%
_{\pi_{\mathcal{G}}^{\left[  \tau\right]  }J}^{\sigma}f\right\vert \lesssim
E_{F}^{\sigma}\left\vert f\right\vert ,
\]
to obtain,%
\begin{align*}
&  \left\vert \mathsf{T}_{\operatorname*{diagonal}}^{\subset_{\mathbf{\rho
},\varepsilon},\operatorname{diff}\operatorname{top}\operatorname{hole}%
}\left(  f,g\right)  \right\vert \lesssim\sum_{F\in\mathcal{F}}\sum
_{J\in\mathcal{C}_{F}^{\mathbf{\tau}-\operatorname*{shift}}}\left\vert
\left\langle T_{\sigma}^{\alpha}\left(  \mathbf{1}_{\mathbb{R}^{n}\setminus
F}\sum_{\substack{I\in\left[  F,\pi_{\mathcal{G}}^{\left[  \tau\right]
}J\right]  \text{ and }J\subset_{\rho,\varepsilon}I\\F\subsetneqq
\pi_{\mathcal{G}}^{\left[  \tau\right]  }J}}\bigtriangleup_{I;\kappa_{1}%
}^{\sigma}f\right)  ,\bigtriangleup_{J;\kappa_{2}}^{\omega}g\right\rangle
_{\omega}\right\vert \\
&  \lesssim\sum_{F\in\mathcal{F}}\sum_{J\in\mathcal{C}_{F}^{\mathbf{\tau
}-\operatorname*{shift}}}\mathrm{P}_{\kappa_{1}}^{\alpha}\left(
J,\mathbf{1}_{\mathbb{R}^{n}\setminus F}\left(  E_{F}^{\sigma}\left\vert
f\right\vert \right)  \sigma\right)  \sqrt{\left\vert J\right\vert _{\omega}%
}\left\Vert \bigtriangleup_{J;\kappa_{2}}^{\omega}g\right\Vert _{L^{2}\left(
\omega\right)  }\\
&  \lesssim\sum_{F\in\mathcal{F}}\left(  E_{F}^{\sigma}\left\vert f\right\vert
\right)  \frac{\left\vert F\right\vert _{\sigma}}{\left\vert F\right\vert
^{1-\frac{\alpha}{n}}}\sum_{J\in\mathcal{C}_{F}^{\mathbf{\tau}%
-\operatorname*{shift}}}\left(  \frac{\ell\left(  J\right)  }{\ell\left(
F\right)  }\right)  ^{\kappa_{1}-\varepsilon\left(  \kappa_{1}+n-\alpha
\right)  }\sqrt{\left\vert J\right\vert _{\omega}}\left\Vert \bigtriangleup
_{J;\kappa_{2}}^{\omega}g\right\Vert _{L^{2}\left(  \omega\right)  }\\
&  \lesssim\sum_{F\in\mathcal{F}}\left(  E_{F}^{\sigma}\left\vert f\right\vert
\right)  \frac{\left\vert F\right\vert _{\sigma}}{\left\vert F\right\vert
^{1-\frac{\alpha}{n}}}\sqrt{\sum_{J\in\mathcal{C}_{F}^{\mathbf{\tau
}-\operatorname*{shift}}}\left(  \frac{\ell\left(  J\right)  }{\ell\left(
F\right)  }\right)  ^{2\left[  \kappa_{1}-\varepsilon\left(  \kappa
_{1}+n-\alpha\right)  \right]  }\left\vert J\right\vert _{\omega}}\sqrt
{\sum_{J\in\mathcal{C}_{F}^{\mathbf{\tau}-\operatorname*{shift}}}\left\Vert
\bigtriangleup_{J;\kappa_{2}}^{\omega}g\right\Vert _{L^{2}\left(
\omega\right)  }^{2}}\\
&  \lesssim\sum_{F\in\mathcal{F}}\left(  E_{F}^{\sigma}\left\vert f\right\vert
\right)  \frac{\left\vert F\right\vert _{\sigma}}{\left\vert F\right\vert
^{1-\frac{\alpha}{n}}}\sqrt{\left\vert F\right\vert _{\omega}}\left\Vert
\mathsf{P}_{\mathcal{C}_{F}^{\mathbf{\tau}-\operatorname*{shift}}}g\right\Vert
_{L^{2}\left(  \omega\right)  }\text{ by pigeonholing }\frac{\ell\left(
J\right)  }{\ell\left(  F\right)  }\text{ and using }\kappa_{1}>\varepsilon
\left(  \kappa_{1}+n-\alpha\right) \\
&  \lesssim\sqrt{A_{2}^{\alpha}\left(  \sigma,\omega\right)  }\sum
_{F\in\mathcal{F}}\left(  E_{F}^{\sigma}\left\vert f\right\vert \right)
\sqrt{\left\vert F\right\vert _{\sigma}}\left\Vert \mathsf{P}_{\mathcal{C}%
_{F}^{\mathbf{\tau}-\operatorname*{shift}}}g\right\Vert _{L^{2}\left(
\omega\right)  }\lesssim\sqrt{A_{2}^{\alpha}\left(  \sigma,\omega\right)
}\left\Vert f\right\Vert _{L^{2}\left(  \sigma\right)  }\left\Vert
g\right\Vert _{L^{2}\left(  \omega\right)  }.
\end{align*}

So we finally turn to analyzing the `plug form' above. We again use the
telescoping identity
\[
\mathbf{1}_{F}\sum_{I\in\left[  F,\pi_{\mathcal{G}}^{\left[  \tau\right]
}J\right]  }\bigtriangleup_{I;\kappa_{1}}^{\sigma}f=\mathbb{E}_{F;\kappa_{1}%
}^{\sigma}f-\mathbf{1}_{F}\mathbb{E}_{\pi_{\mathcal{G}}^{\left[  \tau\right]
}J;\kappa_{1}}^{\sigma}f,\ \ \ \ \ \text{for }F\subsetneqq\pi_{\mathcal{G}%
}^{\left[  \tau\right]  }J
\]
to write the \emph{top difference plug} form as,%
\begin{align*}
\mathsf{T}_{\operatorname*{diagonal}}^{\subset_{\mathbf{\rho},\varepsilon
},\operatorname{diff}\operatorname{top}\operatorname{plug}}\left(  f,g\right)
&  =\sum_{F\in\mathcal{F}}\sum_{J\in\mathcal{C}_{F}^{\mathbf{\tau
}-\operatorname*{shift}}}\left\langle T_{\sigma}^{\alpha}\left(
\mathbf{1}_{F}\sum_{I\in\left[  F,\pi_{\mathcal{G}}^{\left[  \tau\right]
}J\right]  \text{ if }F\subsetneqq\pi_{\mathcal{G}}^{\left[  \tau\right]  }%
J}\bigtriangleup_{I;\kappa_{1}}^{\sigma}f\right)  ,\bigtriangleup
_{J;\kappa_{2}}^{\omega}g\right\rangle _{\omega}\\
&  =\sum_{F\in\mathcal{F}}\sum_{J\in\mathcal{C}_{F}^{\mathbf{\tau
}-\operatorname*{shift}}:\ F\subsetneqq\pi_{\mathcal{G}}^{\left[  \tau\right]
}J}\left\langle T_{\sigma}^{\alpha}\left(  \mathbb{E}_{F;\kappa_{1}}^{\sigma
}f\right)  ,\bigtriangleup_{J;\kappa_{2}}^{\omega}g\right\rangle _{\omega}\\
&  -\sum_{F\in\mathcal{F}}\sum_{J\in\mathcal{C}_{F}^{\mathbf{\tau
}-\operatorname*{shift}}:\ F\subsetneqq\pi_{\mathcal{G}}^{\left[  \tau\right]
}J}\left\langle T_{\sigma}^{\alpha}\left(  \mathbf{1}_{F}\mathbb{E}%
_{\pi_{\mathcal{G}}^{\left[  \tau\right]  }J;\kappa_{1}}^{\sigma}f\right)
,\bigtriangleup_{J;\kappa_{2}}^{\omega}g\right\rangle _{\omega}\\
&  \equiv\mathsf{T}_{\operatorname*{diagonal}}^{\subset_{\mathbf{\rho
},\varepsilon},\operatorname{diff}\operatorname{top}\operatorname{plug}%
;1}\left(  f,g\right)  -\mathsf{T}_{\operatorname*{diagonal}}^{\subset
_{\mathbf{\rho},\varepsilon},\operatorname{diff}\operatorname{top}%
\operatorname{plug};2}\left(  f,g\right)  .
\end{align*}
The first form above is controlled by the local testing condition,%
\begin{align*}
&  \left\vert \mathsf{T}_{\operatorname*{diagonal}}^{\subset_{\mathbf{\rho
},\varepsilon},\operatorname{diff}\operatorname{top}\operatorname{plug}%
;1}\left(  f,g\right)  \right\vert =\left\vert \sum_{F\in\mathcal{F}%
}\left\langle T_{\sigma}^{\alpha}\left(  \mathbb{E}_{F;\kappa_{1}}^{\sigma
}f\right)  ,\sum_{J\in\mathcal{C}_{F}^{\mathbf{\tau}-\operatorname*{shift}%
}:\ F\subsetneqq\pi_{\mathcal{G}}^{\left[  \tau\right]  }J}\bigtriangleup
_{J;\kappa_{2}}^{\omega}g\right\rangle _{\omega}\right\vert \\
&  \leq\sum_{F\in\mathcal{F}}\mathfrak{T}_{T^{\alpha}}^{\kappa_{1}}\left(
\sigma,\omega\right)  \left\Vert \mathbb{E}_{F;\kappa_{1}}^{\sigma
}f\right\Vert _{L^{2}\left(  \sigma\right)  }\left\Vert \sum_{J\in
\mathcal{C}_{F}^{\mathbf{\tau}-\operatorname*{shift}}:\ F\subsetneqq
\pi_{\mathcal{G}}^{\left[  \tau\right]  }J}\bigtriangleup_{J;\kappa_{2}%
}^{\omega}g\right\Vert _{L^{2}\left(  \omega\right)  }\\
&  \leq\sum_{F\in\mathcal{F}}\mathfrak{T}_{T^{\alpha}}^{\kappa_{1}}\left(
\sigma,\omega\right)  \left(  E_{F}^{\sigma}\left\vert f\right\vert \right)
\sqrt{\left\vert F\right\vert _{\sigma}}\left\Vert P_{\mathcal{C}%
_{F}^{\mathbf{\tau}-\operatorname*{shift}}}^{\omega}g\right\Vert
_{L^{2}\left(  \omega\right)  }\\
&  \leq\mathfrak{T}_{T^{\alpha}}^{\kappa_{1}}\left(  \sigma,\omega\right)
\sqrt{\sum_{F\in\mathcal{F}}\left(  E_{F}^{\sigma}\left\vert f\right\vert
\right)  ^{2}\left\vert F\right\vert _{\sigma}}\sqrt{\sum_{F\in\mathcal{F}%
}\left\Vert P_{\mathcal{C}_{F}^{\mathbf{\tau}-\operatorname*{shift}}}^{\omega
}g\right\Vert _{L^{2}\left(  \omega\right)  }^{2}}\lesssim\mathfrak{T}%
_{T^{\alpha}}^{\kappa_{1}}\left(  \sigma,\omega\right)  \left\Vert
f\right\Vert _{L^{2}\left(  \sigma\right)  }\left\Vert g\right\Vert
_{L^{2}\left(  \omega\right)  }.
\end{align*}

To control the second form above we note that if $F\subsetneqq\pi
_{\mathcal{G}}^{\left[  \tau\right]  }J$, then we must have $\pi_{\mathcal{G}%
}J=\pi_{\mathcal{G}}F$ and so,%
\begin{align*}
\mathsf{T}_{\operatorname*{diagonal}}^{\subset_{\mathbf{\rho},\varepsilon
},\operatorname{diff}\operatorname{top};2}\left(  f,g\right)   &  =\sum
_{F\in\mathcal{F}}\sum_{J\in\mathcal{C}_{F}^{\mathbf{\tau}%
-\operatorname*{shift}}:\ F\subsetneqq\pi_{\mathcal{G}}^{\left[  \tau\right]
}J}\left\langle T_{\sigma}^{\alpha}\left(  \mathbf{1}_{F}\mathbb{E}%
_{\pi_{\mathcal{G}}^{\left[  \tau\right]  }J;\kappa_{1}}^{\sigma}f\right)
,\bigtriangleup_{J;\kappa_{2}}^{\omega}g\right\rangle _{\omega}\\
&  =\sum_{F\in\mathcal{F}}\sum_{J\in\mathcal{C}_{F}^{\mathbf{\tau
}-\operatorname*{shift}}:\ \pi_{\mathcal{G}}J=\pi_{\mathcal{G}}F\text{ and
}F\subsetneqq\pi_{\mathcal{G}}^{\left[  \tau\right]  }J}\left\langle
T_{\sigma}^{\alpha}\left(  \mathbf{1}_{F}\mathbb{E}_{\pi_{\mathcal{G}%
}^{\left[  \tau\right]  }F;\kappa_{1}}^{\sigma}f\right)  ,\bigtriangleup
_{J;\kappa_{2}}^{\omega}g\right\rangle _{\omega}\\
&  =\sum_{F\in\mathcal{F}}\sum_{J\in\mathcal{C}_{F}^{\mathbf{\tau
}-\operatorname*{shift}}:\ \pi_{\mathcal{G}}J=\pi_{\mathcal{G}}F\text{ and
}F\subsetneqq\pi_{\mathcal{G}}^{\left[  \tau\right]  }J}\left\langle
T_{\sigma}^{\alpha}\left(  \mathbf{1}_{F}\mathbb{E}_{\pi_{\mathcal{G}%
}^{\left[  \tau\right]  }F;\kappa_{1}}^{\sigma}f\right)  ,\bigtriangleup
_{J;\kappa_{2}}^{\omega}g\right\rangle _{\omega}\\
&  =\sum_{F\in\mathcal{F}}\left\langle T_{\sigma}^{\alpha}\left(
\mathbf{1}_{F}\mathbb{E}_{\pi_{\mathcal{G}}^{\left[  \tau\right]  }%
F;\kappa_{1}}^{\sigma}f\right)  ,\sum_{J\in\mathcal{C}_{F}^{\mathbf{\tau
}-\operatorname*{shift}}:\ \pi_{\mathcal{G}}J=\pi_{\mathcal{G}}F\text{ and
}F\subsetneqq\pi_{\mathcal{G}}^{\left[  \tau\right]  }J}\bigtriangleup
_{J;\kappa_{2}}^{\omega}g\right\rangle _{\omega}%
\end{align*}
whose modulus is at most,%
\begin{align*}
&  \sum_{F\in\mathcal{F}}\left\vert \left\langle T_{\sigma}^{\alpha}\left(
\mathbf{1}_{F}\mathbb{E}_{\pi_{\mathcal{G}}^{\left[  \tau\right]  }%
F;\kappa_{1}}^{\sigma}f\right)  ,\sum_{J\in\mathcal{C}_{F}^{\mathbf{\tau
}-\operatorname*{shift}}:\ \pi_{\mathcal{G}}J=\pi_{\mathcal{G}}F\text{ and
}F\subsetneqq\pi_{\mathcal{G}}^{\left[  \tau\right]  }J}\bigtriangleup
_{J;\kappa_{2}}^{\omega}g\right\rangle _{\omega}\right\vert \\
&  =\sum_{F\in\mathcal{F}}\left\Vert \mathbf{1}_{F}\mathbb{E}_{\pi
_{\mathcal{G}}^{\left[  \tau\right]  }F;\kappa_{1}}^{\sigma}f\right\Vert
_{\infty}\left\vert \left\langle T_{\sigma}^{\alpha}\left(  \mathbf{1}%
_{F}\frac{\mathbb{E}_{\pi_{\mathcal{G}}^{\left[  \tau\right]  }F;\kappa_{1}%
}^{\sigma}f}{\left\Vert \mathbf{1}_{F}\mathbb{E}_{\pi_{\mathcal{G}}^{\left[
\tau\right]  }F;\kappa_{1}}^{\sigma}f\right\Vert _{\infty}}\right)
,\sum_{J\in\mathcal{C}_{F}^{\mathbf{\tau}-\operatorname*{shift}}%
:\ \pi_{\mathcal{G}}J=\pi_{\mathcal{G}}F\text{ and }F\subsetneqq
\pi_{\mathcal{G}}^{\left[  \tau\right]  }J}\bigtriangleup_{J;\kappa_{2}%
}^{\omega}g\right\rangle _{\omega}\right\vert \\
&  \lesssim\mathfrak{T}_{T^{\alpha}}^{\kappa_{1}}\left(  \sigma,\omega\right)
\sum_{F\in\mathcal{F}}\left\Vert \mathbf{1}_{F}\mathbb{E}_{\pi_{\mathcal{G}%
}^{\left[  \tau\right]  }F;\kappa_{1}}^{\sigma}f\right\Vert _{\infty
}\left\Vert \mathbf{1}_{F}\right\Vert _{L^{2}\left(  \sigma\right)
}\left\Vert \sum_{J\in\mathcal{C}_{F}^{\mathbf{\tau}-\operatorname*{shift}%
}:\ \pi_{\mathcal{G}}J=\pi_{\mathcal{G}}F\text{ and }F\subsetneqq
\pi_{\mathcal{G}}^{\left[  \tau\right]  }J}\bigtriangleup_{J;\kappa_{2}%
}^{\omega}g\right\Vert _{L^{2}\left(  \omega\right)  },
\end{align*}
and we can now continue with Cauchy-Schwarz in $F^{\prime}$, noting that the
projections
\[
\left\{  \sum_{F\in\mathcal{F}}\sum_{J\in\mathcal{C}_{F}^{\mathbf{\tau
}-\operatorname*{shift}}:\ \pi_{\mathcal{G}}J=\pi_{\mathcal{G}}F\text{ and
}F\subsetneqq\pi_{\mathcal{G}}^{\left[  \tau\right]  }J}\bigtriangleup
_{J;\kappa_{2}}^{\omega}\right\}  _{F^{\prime}\in\mathcal{F}}%
\]
have pairwise disjoint Alpert supports, and that
\[
\left\Vert \mathbf{1}_{F}\mathbb{E}_{\pi_{\mathcal{G}}^{\left[  \tau\right]
}F;\kappa_{1}}^{\sigma}f\right\Vert _{\infty}\left\Vert \mathbf{1}%
_{F}\right\Vert _{L^{2}\left(  \sigma\right)  }\lesssim\left(  E_{F}^{\sigma
}\left\vert f\right\vert \right)  \sqrt{\left\vert F\right\vert _{\sigma}},
\]
by the construction of the Calder\'{o}n-Zygmund corona in which averages at
the tops of coronas increase. Thus we obtain that%
\[
\left\vert \mathsf{T}_{\operatorname*{diagonal}}^{\subset_{\mathbf{\rho
},\varepsilon},\operatorname{diff}\operatorname{top};2}\left(  f,g\right)
\right\vert \lesssim\mathfrak{T}_{T^{\alpha}}^{\kappa_{1}}\left(
\sigma,\omega\right)  \left\Vert f\right\Vert _{L^{2}\left(  \sigma\right)
}\left\Vert g\right\Vert _{L^{2}\left(  \omega\right)  }.
\]

Finally then we conclude that the \emph{alternate below diagonal} form
$\mathsf{T}_{\operatorname*{diagonal}}^{\subset_{\mathbf{\rho},\varepsilon
},\operatorname*{alt}}\left(  f,g\right)  $ is controlled by
\[
\left\vert \mathsf{T}_{\operatorname*{diagonal}}^{\subset_{\mathbf{\rho
},\varepsilon},\operatorname*{alt}}\left(  f,g\right)  \right\vert
\lesssim\left(  \mathfrak{T}_{T^{\alpha}}^{\kappa_{1}}+\mathfrak{T}%
_{T^{\alpha,\ast}}^{\kappa_{2}}+\mathcal{WBP}_{T^{\alpha}}^{\left(  \kappa
_{1},\kappa_{2}\right)  }+\sqrt{A_{2}^{\alpha}}\right)  \left\Vert
f\right\Vert _{L^{2}\left(  \sigma\right)  }\left\Vert g\right\Vert
_{L^{2}\left(  \omega\right)  }.
\]

\subsubsection{The parallel form}

Thus it remains to bound the decoupled \emph{parallel} form $\mathsf{T}%
_{\operatorname*{diagonal}}^{\supset_{\mathbf{\rho},\varepsilon}%
,\operatorname*{parallel},G}\left(  f,g\right)  $, which we do here using the
dual indicator / cube testing condition, defined by%
\[
\mathfrak{T}_{T^{\ast}}^{\operatorname*{ind}}\left(  \omega,\sigma\right)
\equiv\sup_{Q}\left(  \frac{1}{\left\vert Q\right\vert _{\sigma}}%
\sup_{E\subset Q}\int_{Q}\left\vert T_{\omega}^{\ast}\left(  \chi_{E}\right)
\right\vert ^{p}\omega\right)  ^{\frac{1}{p}}.
\]
For this we recall the following elementary observation from \cite{LaSaUr1}.
For any linear operators $T$, which we may assume have real-valued kernel by
Remark \ref{real}, we have with $g_{h,Q}=\frac{T^{\ast}\left(  \chi_{Q}%
h\omega\right)  }{\left\vert T^{\ast}\left(  \chi_{Q}h\omega\right)
\right\vert }$,%
\begin{align}
&  \sup_{\left\vert g\right\vert \leq1}\frac{1}{\left\vert Q\right\vert
_{\sigma}}\int_{Q}\left\vert T\left(  \chi_{Q}g\sigma\right)  \right\vert
^{p}\omega=\sup_{\left\vert g\right\vert \leq1}\sup_{\left\Vert h\right\Vert
_{L^{p^{\prime}}\left(  \omega\right)  }\leq1}\left\vert \frac{1}{\left\vert
Q\right\vert _{\sigma}}\int_{Q}T\left(  \chi_{Q}g\sigma\right)  h\omega
\right\vert \label{elem}\\
&  =\sup_{\left\Vert h\right\Vert _{L^{p^{\prime}}\left(  \omega\right)  }%
\leq1}\sup_{\left\vert g\right\vert \leq1}\left\vert \frac{1}{\left\vert
Q\right\vert _{\sigma}}\int_{Q}T^{\ast}\left(  \chi_{Q}h\omega\right)
g\sigma\right\vert =\sup_{\left\Vert h\right\Vert _{L^{p^{\prime}}\left(
\omega\right)  }\leq1}\left\vert \frac{1}{\left\vert Q\right\vert _{\sigma}%
}\int_{Q}T^{\ast}\left(  \chi_{Q}h\omega\right)  g_{h,Q}\sigma\right\vert
\nonumber\\
&  =\sup_{\left\Vert h\right\Vert _{L^{p^{\prime}}\left(  \omega\right)  }%
\leq1}\left\vert \frac{1}{\left\vert Q\right\vert _{\sigma}}\int_{Q}T\left(
\chi_{Q}g_{h,Q}\sigma\right)  h\omega\right\vert \leq\sup_{\left\Vert
h\right\Vert _{L^{p^{\prime}}\left(  \omega\right)  }\leq1}\frac{1}{\left\vert
Q\right\vert _{\sigma}}\int_{Q}\left\vert T\left(  \chi_{Q}g_{h,Q}%
\sigma\right)  \right\vert ^{p}\omega\leq2^{p}\mathfrak{T}_{T}%
^{\operatorname*{ind}}\left(  \sigma,\omega\right)  ^{p},\nonumber
\end{align}
upon noting that both of the functions $g$ and $h$ appearing in the suprema
are real-valued, and hence we can write $g_{h,Q}=\mathbf{1}_{A}-\mathbf{1}%
_{B}$.

Then we have,%
\begin{align*}
&  \left\vert \mathsf{T}_{\operatorname*{diagonal}}^{\supset_{\mathbf{\rho
},\varepsilon},\operatorname*{parallel},G}\left(  f,g\right)  \right\vert
=\left\vert \left\langle T_{\omega}^{\alpha,\ast}\mathsf{P}_{\mathcal{C}_{G}%
}^{\omega}g,\mathsf{P}_{\mathcal{C}_{G}^{\mathbf{\tau}-\operatorname*{shift}}%
}^{\sigma}f\right\rangle _{\sigma}\right\vert \leq\left\Vert T_{\omega
}^{\alpha,\ast}\mathsf{P}_{\mathcal{C}_{G}}^{\omega}g\right\Vert
_{L^{2}\left(  \sigma\right)  }\left\Vert \mathsf{P}_{\mathcal{C}%
_{G}^{\mathbf{\tau}-\operatorname*{shift}}}^{\sigma}f\right\Vert
_{L^{2}\left(  \sigma\right)  }\\
&  \leq2\mathfrak{T}_{T^{\alpha,\ast}}^{\operatorname*{ind}}\left(
\omega,\sigma\right)  \left\Vert \mathsf{P}_{\mathcal{C}_{G}}^{\omega
}g\right\Vert _{L^{\infty}}\sqrt{\left\vert G\right\vert _{\omega}}\left\Vert
\mathsf{P}_{\mathcal{C}_{G}^{\mathbf{\tau}-\operatorname*{shift}}}^{\sigma
}f\right\Vert _{L^{2}\left(  \sigma\right)  }\leq2\mathfrak{T}_{T^{\alpha
,\ast}}^{\operatorname*{ind}}\left(  \omega,\sigma\right)  \left(
E_{G}^{\omega}\left\vert g\right\vert \right)  \sqrt{\left\vert G\right\vert
_{\omega}}\left\Vert \mathsf{P}_{\mathcal{C}_{G}^{\mathbf{\tau}%
-\operatorname*{shift}}}^{\sigma}f\right\Vert _{L^{2}\left(  \sigma\right)  },
\end{align*}
and collecting all of the estimates above, we conclude that%
\begin{align}
\left\vert \mathsf{T}_{\operatorname*{diagonal}}^{\supset_{\mathbf{\rho
},\varepsilon}}\left(  f,g\right)  \right\vert  &  \leq\sum_{G\in\mathcal{G}%
}\left\vert \mathsf{T}_{\operatorname*{diagonal}}^{\supset_{\mathbf{\rho
},\varepsilon},G}\left(  f,g\right)  \right\vert \label{followed by}\\
&  \lesssim\left(  \mathfrak{T}_{T^{\alpha}}^{\kappa_{1}}+\mathfrak{T}%
_{T^{\alpha,\ast}}^{\kappa_{2}}+\mathcal{WBP}_{T^{\alpha}}^{\left(  \kappa
_{1},\kappa_{2}\right)  }\left(  \sigma,\omega\right)  +\sqrt{A_{2}^{\alpha}%
}\right)  \left\Vert f\right\Vert _{L^{2}\left(  \sigma\right)  }\left\Vert
g\right\Vert _{L^{2}\left(  \omega\right)  }\nonumber\\
&  +\mathfrak{T}_{T^{\ast}}^{\operatorname*{ind}}\left(  \omega,\sigma\right)
\sqrt{\sum_{G\in\mathcal{G}}\left(  E_{G}^{\omega}\left\vert g\right\vert
\right)  ^{2}\left\vert G\right\vert _{\omega}}\sqrt{\sum_{G\in\mathcal{G}%
}\left\Vert \mathsf{P}_{\mathcal{C}_{G}^{\mathbf{\tau}-\operatorname*{shift}}%
}^{\sigma}f\right\Vert _{L^{2}\left(  \sigma\right)  }^{2}}\nonumber\\
&  \lesssim\left(  \mathfrak{T}_{T^{\alpha}}^{\kappa_{1}}+\mathfrak{T}%
_{T^{\alpha,\ast}}^{\kappa_{2}}+\mathcal{WBP}_{T^{\alpha}}^{\left(  \kappa
_{1},\kappa_{2}\right)  }+\sqrt{A_{2}^{\alpha}}+\mathfrak{T}_{T^{\alpha,\ast}%
}^{\operatorname*{ind}}\left(  \omega,\sigma\right)  \right)  \left\Vert
f\right\Vert _{L^{2}\left(  \sigma\right)  }\left\Vert g\right\Vert
_{L^{2}\left(  \omega\right)  }\ ,\nonumber
\end{align}
since the collection of Alpert projections $\left\{  \mathsf{P}_{\mathcal{C}%
_{G}^{\mathbf{\tau}-\operatorname*{shift}}}^{\sigma}\right\}  _{G\in
\mathcal{G}}$ have pairwise disjoint Alpert support.

\section{Conclusion of the proofs}

Collecting all the estimates proved above, namely (\ref{routine}),
(\ref{top control}), (\ref{second far below}), (\ref{first far below}),
(\ref{unif bound'}), (\ref{diag est}), (\ref{para est}), (\ref{comm est}),
(\ref{stop est}), and (\ref{neigh est}), and the relevant corresponding dual
estimates including (\ref{followed by}) for the dual diagonal form, we obtain
that for any dyadic grid $\mathcal{D}$, and any admissible truncation of
$T^{\alpha}$,%

\begin{align*}
\left\vert \left\langle T_{\sigma}^{\alpha}\mathsf{P}_{\operatorname{good}%
}^{\mathcal{D}}f,\mathsf{P}_{\operatorname{good}}^{\mathcal{D}}g\right\rangle
_{\omega}\right\vert  &  \leq C\left(  \mathfrak{TR}_{T^{\alpha}}^{\left(
\kappa\right)  }\left(  \sigma,\omega\right)  +\mathfrak{TR}_{T^{\alpha,\ast}%
}^{\left(  \kappa\right)  }\left(  \omega,\sigma\right)  +\sqrt{A_{2}^{\alpha
}\left(  \sigma,\omega\right)  }+\mathfrak{T}_{T^{\alpha}}%
^{\operatorname*{ind}}\left(  \sigma,\omega\right)  +\varepsilon
_{3}\mathfrak{N}_{T^{\alpha}}\left(  \sigma,\omega\right)  \right) \\
&  \times\left\Vert \mathsf{P}_{\operatorname{good}}^{\mathcal{D}}f\right\Vert
_{L^{2}\left(  \sigma\right)  }\left\Vert \mathsf{P}_{\operatorname{good}%
}^{\mathcal{D}}g\right\Vert _{L^{2}\left(  \omega\right)  }.
\end{align*}
Thus for any admissible truncation of $T^{\alpha}$ we obtain from
\cite[Sections 7.4 - 7.7]{NTV3}, followed by the previous display,
\begin{align}
\mathfrak{N}_{T^{\alpha}}\left(  \sigma,\omega\right)   &  \leq C\sup
_{\mathcal{D}}\sup_{f\in L^{2}\left(  \sigma\right)  \text{ and }g\in
L^{2}\left(  \omega\right)  }\frac{\left\vert \left\langle T_{\sigma}^{\alpha
}\mathsf{P}_{\operatorname{good}}^{\mathcal{D}}f,\mathsf{P}%
_{\operatorname{good}}^{\mathcal{D}}g\right\rangle _{\omega}\right\vert
}{\left\Vert \mathsf{P}_{\operatorname{good}}^{\mathcal{D}}f\right\Vert
_{L^{2}\left(  \sigma\right)  }\left\Vert \mathsf{P}_{\operatorname{good}%
}^{\mathcal{D}}g\right\Vert _{L^{2}\left(  \omega\right)  }}\label{absorb}\\
&  \leq C\left(  \mathfrak{TR}_{T^{\alpha}}^{\left(  \kappa\right)  }\left(
\sigma,\omega\right)  +\mathfrak{TR}_{T^{\alpha,\ast}}^{\left(  \kappa\right)
}\left(  \omega,\sigma\right)  +\sqrt{A_{2}^{\alpha}\left(  \sigma
,\omega\right)  }+\mathfrak{T}_{T^{\alpha}}^{\operatorname*{ind}}\left(
\sigma,\omega\right)  \right)  +C\varepsilon_{3}\mathfrak{N}_{T^{\alpha}%
}\left(  \sigma,\omega\right)  .\nonumber
\end{align}
Our next task is to use the doubling hypothesis to replace the triple $\kappa
$-testing constants by the usual cube testing constants, and we start with a lemma.

\begin{lemma}
Let $m\geq1$. Suppose that $\sigma$ and $\omega$ are locally finite positive
Borel measures on $\mathbb{R}^{n}$, with $\sigma$ doubling. If $T^{\alpha}$ is
a bounded operator from $L^{2}\left(  \sigma\right)  $ to $L^{2}\left(
\omega\right)  $, then\ for every grandchild $I^{\prime}\in\mathfrak{C}%
_{\mathcal{D}}^{\left(  m\right)  }\left(  I\right)  $, and each
$0<\varepsilon_{1}<1$, there is a positive constant $C_{m,\varepsilon_{1}}$
such that%
\[
\sqrt{\int_{3I\setminus I^{\prime}}\left\vert T_{\sigma}^{\alpha}%
\mathbf{1}_{I^{\prime}}\right\vert ^{2}d\omega}\leq\left\{  C_{m,\varepsilon
_{1}}\sqrt{A_{2}^{\alpha}\left(  \sigma,\omega\right)  }+\varepsilon
_{1}\mathfrak{N}_{T^{\alpha}}\left(  \sigma,\omega\right)  \right\}
\sqrt{\left\vert I^{\prime}\right\vert _{\sigma}}\ .
\]

\end{lemma}

\begin{proof}
Given $0<\delta<1$, we split the left hand side as usual,%
\begin{align*}
\sqrt{\int_{3I\setminus I^{\prime}}\left\vert T_{\sigma}^{\alpha}%
\mathbf{1}_{I^{\prime}}\right\vert ^{2}d\omega}  &  \lesssim\sqrt
{\int_{3I\setminus I^{\prime}}\left\vert T_{\sigma}^{\alpha}\mathbf{1}%
_{\left(  1-\delta\right)  I^{\prime}}\right\vert ^{2}d\omega}+\sqrt
{\int_{3I\setminus I^{\prime}}\left\vert T_{\sigma}^{\alpha}\mathbf{1}%
_{I^{\prime}\setminus\left(  1-\delta\right)  I^{\prime}}\right\vert
^{2}d\omega}\\
&  \equiv A+B.
\end{align*}
Using the halo estimate in \cite[Lemma 24 on page 24]{Saw6}, we obtain%
\[
B\leq\mathfrak{N}_{T^{\alpha}}\left(  \sigma,\omega\right)  \sqrt{\left\vert
I^{\prime}\setminus\left(  1-\delta\right)  I^{\prime}\right\vert _{\sigma}%
}\leq\sqrt{\frac{C}{\ln\frac{1}{\delta}}}\mathfrak{N}_{T^{\alpha}}\left(
\sigma,\omega\right)  \sqrt{\left\vert I^{\prime}\right\vert _{\sigma}}.
\]
For term $A$ we have%
\begin{align*}
A  &  \lesssim\sqrt{\int_{3I\setminus I^{\prime}}\left\vert \int_{\left(
1-\delta\right)  I^{\prime}}\left[  \delta\ell\left(  I^{\prime}\right)
\right]  ^{\alpha-n}d\sigma\right\vert ^{2}d\omega}=\sqrt{\left(  3\cdot
2^{m}\delta\right)  ^{2\left(  \alpha-n\right)  }\frac{\left\vert \left(
1-\delta\right)  I\right\vert _{\sigma}\left\vert 3I\setminus I^{\prime
}\right\vert _{\omega}}{\ell\left(  3I\right)  ^{2\left(  n-\alpha\right)  }%
}\left\vert \left(  1-\delta\right)  I\right\vert _{\sigma}}\\
&  \lesssim\sqrt{\left(  2^{m}\delta\right)  ^{2\left(  \alpha-n\right)
}\frac{\left\vert 3I\right\vert _{\sigma}\left\vert 3I\right\vert _{\omega}%
}{\ell\left(  3I\right)  ^{2\left(  n-\alpha\right)  }}\left\vert I\right\vert
_{\sigma}}\leq\left(  2^{m}\delta\right)  ^{\alpha-n}A_{2}^{\alpha}\left(
\sigma,\omega\right)  \sqrt{\left\vert I\right\vert _{\sigma}}.
\end{align*}
Now we choose $0<\delta<1$ so that $\sqrt{\frac{C}{\ln\frac{1}{\delta}}}%
\leq\varepsilon_{1}$.
\end{proof}

Recall that the triple $\kappa$-cube testing conditions use the $Q$-normalized
monomials $m_{Q}^{\beta}\left(  x\right)  \equiv\mathbf{1}_{Q}\left(
x\right)  \left(  \frac{x-c_{Q}}{\ell\left(  Q\right)  }\right)  ^{\beta}$,
for which we have $\left\Vert m_{Q}^{\beta}\right\Vert _{L^{\infty}}\approx1$.

\begin{theorem}
\label{no tails}Suppose that $\sigma$ and $\omega$ are locally finite positive
Borel measures on $\mathbb{R}^{n}$, with $\sigma$ doubling, and let $\kappa
\in\mathbb{N}$. If $T^{\alpha}$ is a bounded operator from $L^{2}\left(
\sigma\right)  $ to $L^{2}\left(  \omega\right)  $, then\ for every
$0<\varepsilon_{2}<1$, there is a positive constant $C\left(  \kappa
,\varepsilon_{2}\right)  $ such that
\[
\mathfrak{TR}_{T^{\alpha}}^{\left(  \kappa\right)  }\left(  \sigma
,\omega\right)  \leq C\left(  \kappa,\varepsilon_{2}\right)  \left[
\mathfrak{T}_{T^{\alpha}}\left(  \sigma,\omega\right)  +\sqrt{A_{2}^{\alpha
}\left(  \sigma,\omega\right)  }\right]  +\varepsilon_{2}\mathfrak{N}%
_{T^{\alpha}}\left(  \sigma,\omega\right)  \ ,\ \ \ \ \ \kappa\geq1,
\]
and where the constants $C\left(  \kappa,\varepsilon_{2}\right)  $ depend only
on $\kappa$ and $\varepsilon$, and not on the operator norm $\mathfrak{N}%
_{T^{\alpha}}\left(  \sigma,\omega\right)  $.
\end{theorem}

\begin{proof}
Fix a dyadic cube $I$. If $P$ is an $I$-normalized polynomial of degree less
than $\kappa$ on the cube $I$, i.e. $\left\Vert P\right\Vert _{L^{\infty}%
}\approx1$, then we can approximate $P$ by a step function
\[
S\equiv\sum_{I^{\prime}\in\mathfrak{C}_{\mathcal{D}}^{\left(  m\right)
}\left(  I\right)  }a_{I^{\prime}}\mathbf{1}_{I^{\prime}},
\]
satisfying%
\[
\left\Vert S-\mathbf{1}_{I}P\right\Vert _{L^{\infty}\left(  \sigma\right)
}<\frac{\varepsilon_{2}}{2}\ ,
\]
provided we take $m\geq1$ sufficiently large depending on $n$ and $\kappa$,
but independent of the cube $I$. Then using the above lemma with $C2^{\frac
{m}{2}}\varepsilon_{1}\leq\frac{\varepsilon_{2}}{2}$, and the estimate
$\left\vert a_{I^{\prime}}\right\vert \lesssim\left\Vert P\right\Vert
_{L^{\infty}}\lesssim1$, we have%
\begin{align*}
&  \sqrt{\int_{3I}\left\vert T_{\sigma}^{\alpha}\mathbf{1}_{I}P\right\vert
^{2}d\omega}\leq\sqrt{\int_{3I}\left\vert \sum_{I^{\prime}\in\mathfrak{C}%
_{\mathcal{D}}^{\left(  m\right)  }\left(  I\right)  }a_{I^{\prime}}T_{\sigma
}^{\alpha}\mathbf{1}_{I^{\prime}}\right\vert ^{2}d\omega}+\sqrt{\int
_{3I}\left\vert T_{\sigma}^{\alpha}\left[  \left(  S-P\right)  \mathbf{1}%
_{I}\right]  \right\vert ^{2}d\omega}\\
&  \leq C\sum_{I^{\prime}\in\mathfrak{C}_{\mathcal{D}}^{\left(  m\right)
}\left(  I\right)  }\sqrt{\int_{3I\setminus I^{\prime}}\left\vert
a_{I^{\prime}}T_{\sigma}^{\alpha}\mathbf{1}_{I^{\prime}}\right\vert
^{2}d\omega}+C\sum_{I^{\prime}\in\mathfrak{C}_{\mathcal{D}}^{\left(  m\right)
}\left(  Q\right)  }\sqrt{\int_{I^{\prime}}\left\vert a_{I^{\prime}}T_{\sigma
}^{\alpha}\mathbf{1}_{I^{\prime}}\right\vert ^{2}d\omega}+\frac{\varepsilon
_{2}}{2}\mathfrak{N}_{T^{\alpha}}\left(  \sigma,\omega\right)  \sqrt
{\left\vert I\right\vert _{\sigma}}\\
&  \leq C\sum_{I^{\prime}\in\mathfrak{C}_{\mathcal{D}}^{\left(  m\right)
}\left(  Q\right)  }\left\{  C_{m,\varepsilon_{1}}\sqrt{A_{2}^{\alpha}\left(
\sigma,\omega\right)  }+\varepsilon_{1}\mathfrak{N}_{T^{\alpha}}\left(
\sigma,\omega\right)  +\mathfrak{T}_{T^{\alpha}}\left(  \sigma,\omega\right)
\right\}  \sqrt{\left\vert I^{\prime}\right\vert _{\sigma}}+\frac
{\varepsilon_{2}}{2}\mathfrak{N}_{T^{\alpha}}\left(  \sigma,\omega\right)
\sqrt{\left\vert I\right\vert _{\sigma}}\\
&  \leq C\left\{  \mathfrak{T}_{T^{\alpha}}\left(  \sigma,\omega\right)
+\sqrt{A_{2}^{\alpha}\left(  \sigma,\omega\right)  }\right\}  \sqrt{\left\vert
I\right\vert _{\sigma}}+\left(  C2^{\frac{m}{2}}\varepsilon_{1}+\frac
{\varepsilon_{2}}{2}\right)  \mathfrak{N}_{T^{\alpha}}\left(  \sigma
,\omega\right)  \sqrt{\left\vert I\right\vert _{\sigma}}.
\end{align*}

\end{proof}

Combining this with (\ref{absorb}) we obtain%
\[
\mathfrak{N}_{T^{\alpha}}\left(  \sigma,\omega\right)  \leq C\left(
\mathfrak{T}_{T^{\alpha}}\left(  \sigma,\omega\right)  +\mathfrak{T}%
_{T^{\alpha,\ast}}\left(  \omega,\sigma\right)  +\sqrt{A_{2}^{\alpha}\left(
\sigma,\omega\right)  }+\mathfrak{T}_{T^{\alpha}}^{\operatorname*{ind}}\left(
\sigma,\omega\right)  \right)  +C\left(  \varepsilon_{2}+\varepsilon
_{3}\right)  \mathfrak{N}_{T^{\alpha}}\left(  \sigma,\omega\right)  .
\]
Since $\mathfrak{N}_{T^{\alpha}}\left(  \sigma,\omega\right)  <\infty$ for
each truncation, we may absorb the final summand on the right into the left
hand side provided $C\left(  \varepsilon_{2}+\varepsilon_{3}\right)  <\frac
{1}{2}$, to obtain%
\[
\mathfrak{N}_{T^{\alpha}}\left(  \sigma,\omega\right)  \lesssim\mathfrak{T}%
_{T^{\alpha}}\left(  \sigma,\omega\right)  +\mathfrak{T}_{T^{\alpha,\ast}%
}\left(  \omega,\sigma\right)  +\sqrt{A_{2}^{\alpha}\left(  \sigma
,\omega\right)  }+\mathfrak{T}_{T^{\alpha}}^{\operatorname*{ind}}\left(
\sigma,\omega\right)  ,
\]
Now we eliminate the Muckenhoupt constant from the right hand side using the
following lemma.

\begin{lemma}
If $K^{\alpha}$ is elliptic in the sense of Stein, then
\[
\sqrt{A_{2}^{\alpha}\left(  \sigma,\omega\right)  }\lesssim\mathfrak{T}%
_{T^{\alpha}}^{\operatorname*{ind}}\left(  \sigma,\omega\right)  .
\]

\end{lemma}

\begin{proof}
The argument in \cite[see the proof of Proposition 7 page 210]{Ste2} shows
that if $T^{\alpha}$ satisfies (\ref{steinelliptic}), there is $\varepsilon
>0$, depending only on the constant $C_{CZ}$ in (\ref{sizeandsmoothness'}),
such that
\[
\frac{\left\vert Q\right\vert _{\omega}\left\vert Q^{\prime}\right\vert
_{\sigma}}{\ell\left(  Q\right)  ^{2\left(  n-\alpha\right)  }}\lesssim
\mathfrak{T}_{T^{\alpha}}^{\operatorname*{ind}}\left(  \sigma,\omega\right)
^{2},
\]
when $\operatorname*{dist}\left(  Q,Q^{\prime}\right)  \approx\ell\left(
Q\right)  =\ell\left(  Q^{\prime}\right)  $, and $\left\vert \frac
{x-y}{\left\vert x-y\right\vert }-\mathbf{u}_{0}\right\vert <\varepsilon$
whenever $x\in Q$ and $y\in Q^{\prime}$. Indeed, given such cubes $Q$ and
$Q^{\prime}$ we may assume they are subcubes of a cube $I$ with $\ell\left(
I\right)  \lesssim\ell\left(  Q\right)  $, and then we have%
\[
\mathfrak{T}_{T^{\alpha}}^{\operatorname*{ind}}\left(  \sigma,\omega\right)
^{2}\geq\sup_{E\subset I}\frac{\int_{I}\left\vert T_{\sigma}^{\lambda
}\mathbf{1}_{E}\right\vert ^{2}d\omega}{\left\vert I\right\vert _{\sigma}%
}\gtrsim\frac{1}{\left\vert Q\right\vert _{\sigma}}\left(  \frac{\left\vert
Q^{\prime}\right\vert _{\sigma}\left\vert Q\right\vert _{\omega}}{\ell\left(
I\right)  ^{n-\alpha}}\right)  ^{2}\gtrsim\frac{\left\vert Q\right\vert
_{\omega}\left\vert Q^{\prime}\right\vert _{\sigma}}{\ell\left(  Q\right)
^{2\left(  n-\alpha\right)  }}\approx\frac{\left\vert Q\right\vert _{\omega
}\left\vert Q\right\vert _{\sigma}}{\ell\left(  Q\right)  ^{2\left(
n-\alpha\right)  }},
\]
since $\sigma$ is doubling. The lemma follows upon taking the supremum over
cubes $Q$.
\end{proof}

Since the norm constant obviously bounds the two testing constants, we have
proved%
\[
\mathfrak{N}_{T^{\alpha}}\left(  \sigma,\omega\right)  \approx\mathfrak{T}%
_{T^{\alpha,\ast}}\left(  \omega,\sigma\right)  +\mathfrak{T}_{T^{\alpha}%
}^{\operatorname*{ind}}\left(  \sigma,\omega\right)  ,
\]
and the remaining equivalences in Theorem \ref{main} that do not involve
cancellation conditions $\mathfrak{A}_{T^{\alpha}}\left(  \sigma
,\omega\right)  $ and $\mathfrak{A}_{T^{\alpha.\ast}}\left(  \omega
,\sigma\right)  $ follow by symmetry. By \cite{Saw6}, the testing conditions
are equivalent to the cancellation conditions when the measures are doubling.

\section{A $T_{\operatorname*{translate}}$ theorem for doubling measures}

\label{T translate section}

The arguments used above rely crucially on the nested property of dyadic
cubes, and break down completely even for balls. Here we instead use a fairly
elementary argument to show that testing over indicators of cubes
$\mathbf{1}_{Q}$ in the doubling theorem above can be replaced by testing over
indicators of balls $\mathbf{1}_{B}$. In fact, we can replace cubes or balls
by translates and dilates of any fixed bounded set having positive Lebesgue
measure. The theorem below extends this to translates of fixed bounded
functions $b_{k},b_{k}^{\ast}$ with integral $1$ at each length scale $2^{k}$,
which we refer to as a \textquotedblleft$T_{\operatorname*{translate}}$
theorem\textquotedblright\ \emph{because at any given scale, we test a
function \textquotedblleft b}$_{k}$\emph{\textquotedblright\ and all of its
translates at that scale.}

More precisely, suppose that for each $t\in\left(  0,\infty\right)  $, we are
given a bounded complex-valued function $b_{t}$ on $\mathbb{R}^{n}$ satisfying

\begin{enumerate}
\item $\operatorname*{Supp}b_{t}\subset B\left(  0,t\right)  $,

\item $\left\vert b_{t}\left(  x\right)  \right\vert \leq\frac{C}{\left\vert
Q_{t}\right\vert }$,

\item $\int b_{t}\left(  x\right)  dx=1$.
\end{enumerate}

\begin{definition}
\label{translate testing}Set $\mathcal{B}\equiv\left\{  b_{Q}\right\}
_{Q\in\mathcal{P}^{n}}$ where $b_{Q}\left(  x\right)  \equiv\left\vert
Q\right\vert b_{\ell\left(  Q\right)  }\left(  x-c_{Q}\right)  $ is a
translation and normalization of $b_{\ell\left(  Q\right)  }$ (so that it
satisfies estimates similar to $\mathbf{1}_{Q}$) where the functions $b_{t}$
are as above. For $\sigma$ and $\omega$ locally finite positive Borel measures
on $\mathbb{R}^{n}$, define the $\mathcal{B}$\emph{-testing constant} and the
$\delta$\emph{-full }$\mathcal{B}$\emph{-testing constant} for the operator
$T^{\alpha}$ by%
\begin{align*}
\mathfrak{T}_{T^{\alpha}}^{\mathcal{B}}\left(  \sigma,\omega\right)   &
\equiv\sup_{Q}\frac{\sqrt{\int_{Q}\left\vert T_{\sigma}^{\alpha}b_{Q}\left(
x\right)  \right\vert ^{2}d\omega\left(  x\right)  }}{\sqrt{\int_{Q}\left\vert
b_{Q}\left(  y\right)  \right\vert ^{2}d\sigma\left(  y\right)  }},\\
\mathfrak{F}_{\delta}\mathfrak{T}_{T^{\alpha}}^{\mathcal{B}}\left(
\sigma,\omega\right)   &  \equiv\sup_{Q}\frac{\sqrt{\int_{\frac{2}{\delta}%
Q}\left\vert T_{\sigma}^{\alpha}b_{Q}\left(  x\right)  \right\vert ^{2}%
d\omega\left(  x\right)  }}{\sqrt{\int_{Q}\left\vert b_{Q}\left(  y\right)
\right\vert ^{2}d\sigma\left(  y\right)  }}.
\end{align*}

\end{definition}

If we take $b_{t}(x)=\frac{1}{t^{n}}\mathbf{1}_{B\left(  0,t\right)  }\left(
x\right)  $, then the $\mathcal{B}$-testing constant is simply the
ball-testing constant.

Fix $0<\delta<1$. Given a cube $Q\in\mathcal{P}$, define the $\delta
$-convolution%
\begin{align*}
b_{Q,\delta}\left(  y\right)   &  \equiv\mathbf{1}_{Q}\ast b_{\delta
\ell\left(  Q\right)  }\left(  y\right)  =\int_{\mathbb{R}^{n}}\mathbf{1}%
_{Q}\left(  y-z\right)  b_{\delta\ell\left(  Q\right)  }\left(  z\right)  dz\\
&  =\int_{Q}b_{\delta\ell\left(  Q\right)  }\left(  y-z\right)  dz=\int
_{Q}\tau_{z}b_{\delta\ell\left(  Q\right)  }\left(  y\right)
dz,\ \ \ \ \ x\in\mathbb{R}^{n}.
\end{align*}
Now we compute%
\begin{align*}
&  \sqrt{\int_{Q}\left\vert T_{\sigma}^{\alpha}\mathbf{1}_{Q}\left(  x\right)
\right\vert ^{2}d\omega\left(  x\right)  }\leq\sqrt{\int_{Q}\left\vert
T_{\sigma}^{\alpha}b_{Q,\delta}\left(  x\right)  \right\vert ^{2}%
d\omega\left(  x\right)  }+\sqrt{\int_{Q}\left\vert T_{\sigma}^{\alpha}\left(
\mathbf{1}_{Q}-b_{Q,\delta}\right)  \left(  x\right)  \right\vert ^{2}%
d\omega\left(  x\right)  }\\
&  \leq\sqrt{\int_{Q}\left\vert T_{\sigma}^{\alpha}b_{Q,\delta}\left(
x\right)  \right\vert ^{2}d\omega\left(  x\right)  }+\mathfrak{N}_{T^{\alpha}%
}\sqrt{\int_{Q}\left\vert \left(  \mathbf{1}_{Q}-b_{Q,\delta}\right)  \left(
x\right)  \right\vert ^{2}d\sigma\left(  x\right)  },
\end{align*}
where%
\begin{align*}
&  \sqrt{\int_{Q}\left\vert T_{\sigma}^{\alpha}b_{Q,\delta}\left(  x\right)
\right\vert ^{2}d\omega\left(  x\right)  }=\left\vert Q\right\vert \sqrt
{\int_{Q}\left\vert T_{\sigma}^{\alpha}\left(  \int_{Q}\tau_{z}b_{\delta
\ell\left(  Q\right)  }\frac{dz}{\left\vert Q\right\vert }\right)  \left(
x\right)  \right\vert ^{2}d\omega\left(  x\right)  }\\
&  \leq\left\vert Q\right\vert \int_{Q}\sqrt{\int_{Q}\left\vert T_{\sigma
}^{\alpha}\left(  \tau_{z}b_{\delta\ell\left(  Q\right)  }\right)  \left(
x\right)  \right\vert ^{2}d\omega\left(  x\right)  }\frac{dz}{\left\vert
Q\right\vert }\\
&  \leq\int_{Q}\mathfrak{F}_{\delta}\mathfrak{T}_{T^{\alpha}}^{\mathcal{B}%
}\sqrt{\int_{Q}\left\vert \left(  \tau_{z}b_{\delta\ell\left(  Q\right)
}\right)  \left(  y\right)  \right\vert ^{2}d\sigma\left(  y\right)  }dz\\
&  \leq\mathfrak{F}_{\delta}\mathfrak{T}_{T^{\alpha}}^{\mathcal{B}}%
\frac{C\left\vert Q\right\vert }{\left\vert Q_{\delta\ell\left(  Q\right)
}\right\vert }\sqrt{\left\vert Q\right\vert _{\sigma}}\leq C_{\delta
}\mathfrak{F}_{\delta}\mathfrak{T}_{T^{\alpha}}^{\mathcal{B}}\sqrt{\left\vert
Q\right\vert _{\sigma}},
\end{align*}
and%
\[
\sqrt{\int_{Q}\left\vert \left(  \mathbf{1}_{Q}-b_{Q,\delta}\right)  \left(
x\right)  \right\vert ^{2}d\sigma\left(  x\right)  }\leq C\sqrt{\frac{1}%
{\ln\frac{1}{\delta}}\left\vert Q\right\vert _{\sigma}}.
\]
Thus altogether we obtain%
\begin{equation}
\mathfrak{T}_{T^{\alpha}}=\sup_{Q}\frac{\sqrt{\int_{Q}\left\vert T_{\sigma
}^{\alpha}\mathbf{1}_{Q}\left(  x\right)  \right\vert ^{2}d\omega\left(
x\right)  }}{\sqrt{\left\vert Q\right\vert _{\sigma}}}\leq C_{\delta
}\mathfrak{F}_{\delta}\mathfrak{T}_{T^{\alpha}}^{\mathcal{B}}+C\sqrt{\frac
{1}{\ln\frac{1}{\delta}}}\mathfrak{N}_{T^{\alpha}}\ , \label{alt}%
\end{equation}
and now using $\sqrt{A_{2}^{\alpha}}$ to control $\mathfrak{F}_{\delta
}\mathfrak{T}_{T^{\alpha}}^{\mathcal{B}}$ by $\mathfrak{T}_{T^{\alpha}%
}^{\mathcal{B}}$, we have shown that a $T1$ theorem for $T^{\alpha}$ over
cubes follows from a $Tb$ theorem where the functions $b_{Q}$ are translates
of the fixed function $b_{\ell\left(  Q\right)  }$. In particular we can take
$b_{t}=\frac{1}{t^{n}}\mathbf{1}_{B\left(  0,t\right)  }\left(  x\right)  $ to
get ball-testing implies cube-testing.

More precisely, we have the following theorem.

\begin{theorem}
\label{trans}Let $\sigma$ and $\omega$ be doubling measures on $\mathbb{R}%
^{n}$. Then with $T^{\alpha}$ and $\mathcal{B}$ as above we have both%
\begin{align}
\mathfrak{F}_{\delta}\mathfrak{T}_{T^{\alpha}}^{\mathcal{B}}\left(
\sigma,\omega\right)   &  \leq C_{\delta}\mathfrak{T}_{T^{\alpha}%
}^{\mathcal{B}}\left(  \sigma,\omega\right)  +C_{\delta}\sqrt{A_{2}^{\alpha
}\left(  \sigma,\omega\right)  }+C\sqrt{\frac{1}{\ln\frac{1}{\delta}}%
}\mathfrak{N}_{T^{\alpha}}\left(  \sigma,\omega\right)  ,\label{both}\\
\mathfrak{T}_{T^{\alpha}}\left(  \sigma,\omega\right)   &  \leq C_{\delta
}\mathfrak{F}_{\delta}\mathfrak{T}_{T^{\alpha}}^{\mathcal{B}}\left(
\sigma,\omega\right)  +C\sqrt{\frac{1}{\ln\frac{1}{\delta}}}\mathfrak{N}%
_{T^{\alpha}}\left(  \sigma,\omega\right)  .\nonumber
\end{align}
Altogether then we have%
\[
\mathfrak{T}_{T^{\alpha}}\left(  \sigma,\omega\right)  \leq C_{\delta
}\mathfrak{T}_{T^{\alpha}}^{\mathcal{B}}\left(  \sigma,\omega\right)
+C_{\delta}\sqrt{A_{2}^{\alpha}\left(  \sigma,\omega\right)  }+C\sqrt{\frac
{1}{\ln\frac{1}{\delta}}}\mathfrak{N}_{T^{\alpha}}\left(  \sigma
,\omega\right)  ,
\]
and hence
\[
\mathfrak{N}_{T^{\alpha}}\left(  \sigma,\omega\right)  \leq C\left(
\mathfrak{T}_{T^{\alpha}}^{\mathcal{B}}+\mathfrak{T}_{\left(  T^{\alpha
}\right)  ^{\ast}}^{\operatorname*{ind},\mathcal{B}}+\sqrt{A_{2}^{\alpha}%
}\right)  ,
\]
where%
\[
\mathfrak{T}_{\left(  T^{\alpha}\right)  ^{\ast}}^{\operatorname*{ind}%
,\mathcal{B}}\equiv\sup_{B}\sqrt{\frac{1}{\left\vert B\right\vert _{\omega}%
}\sup_{E\subset B}\int_{B}\left\vert T_{\omega}^{\alpha,\ast}\mathbf{1}%
_{E}\right\vert ^{2}\sigma},
\]
and the supremum is taken over all balls $B$.

\begin{proof}
The bound for $\mathfrak{F}_{\delta}\mathfrak{T}_{T^{\alpha}}^{\mathcal{B}%
}\left(  \sigma,\omega\right)  $ in the first line of (\ref{both}) is in
Theorem \ref{no tails}, and the bound for $\mathfrak{T}_{T^{\alpha}}\left(
\sigma,\omega\right)  $ in the second line of (\ref{both}) is in (\ref{alt}).
The reader can easily check that the indicator / cube constant $\mathfrak{T}%
_{\left(  T^{\alpha}\right)  ^{\ast}}^{\operatorname*{ind}}$ is comparable to
the indicator / ball constant $\mathfrak{T}_{\left(  T^{\alpha}\right)
^{\ast}}^{\operatorname*{ind},\mathcal{B}}$ because the measures are doubling.
An absorption argument proves the final display.
\end{proof}
\end{theorem}

\begin{remark}
The only property of cubes used in Theorem \ref{trans}, was that their
boundaries are not charged by doubling measures (the specific decay
$\sqrt{\frac{1}{\ln\frac{1}{\delta}}}$ of the halo can be replaced by any
decay to zero). Thus we can replace cube testing $\mathfrak{T}_{T^{\alpha}%
}\left(  \sigma,\omega\right)  $ in Theorem \ref{trans} by ball testing, or
any other `shape testing' in which the boundary of the `shape' is not charged
by doubling measures.
\end{remark}

In fact one can extend Theorem \ref{main} to testing over balls instead of
testing over cubes. Due to the above arguments, it remains only to verify that
the ball testing conditions are equivalent to the cancellation conditions over
spherical annuli, which follows by mimicking the arguments in \cite[see
Subsection 9.2]{Saw6}, with spherical annuli replacing cubical annuli.

\section{A characterization of weak type inequalities}

The passage from weak type to strong type above required the recent technology
of NTV good cubes \cite{NTV}, weighted Alpert wavelets \cite{RaSaWi} and
corona decompositions \cite{NTV4}. On the other hand, the characterization of
weak type inequalities derived here uses the classical machinery of Whitney
decompositions, maximum principles and good $\lambda$ inequalities for maximal
singular integrals; see e.g. \cite{LaSaUr1} and \cite{Ste2}. Theorem
\ref{weak type} links the two approaches.

Throughout this section, we suppose that the kernel $K^{\alpha}\left(
x,y\right)  $ of $T^{\alpha}$ satisfies
\begin{align}
\left\vert K^{\alpha}\left(  x,y\right)  \right\vert  &  \leq
C_{\operatorname*{CZ}}\left\vert x-y\right\vert ^{\alpha-n},\label{basic}\\
\left\vert K^{\alpha}\left(  x,y\right)  -K^{\alpha}\left(  x^{\prime
},y\right)  \right\vert  &  \leq C_{\operatorname*{CZ}}\eta\left(
\frac{\left\vert x-x^{\prime}\right\vert }{\left\vert x^{\prime}-y\right\vert
}\right)  \left\vert x-y\right\vert ^{\alpha-n},\ \ \ \ \ \text{when
}\left\vert x^{\prime}-y\right\vert \geq c\left\vert x-x^{\prime}\right\vert
,\nonumber
\end{align}
where $\eta$ is a unit Dini modulus, i.e. $\eta$ is nondecreasing on $\left[
0,1\right]  $ and $\int_{0}^{1}\eta\left(  s\right)  \frac{ds}{s}=1$.

\begin{remark}
The reader can easily check that all of the results in this section extend to
weighted $L^{p}$ spaces for $1<p<\infty$ as well.
\end{remark}

\subsection{Maximal singular integrals}

We begin by adapting an argument in \cite[Chapter I, subsection 7.3, pages
34-36.]{Ste2} in order to control $T_{\flat}^{\alpha}$ by weighted operators
and $T^{\alpha}$.

\begin{enumerate}
\item For any locally finite positive Borel measure $\mu$ we define the
\emph{centered} maximal operator $\mathcal{M}_{\mu}$ acting on a finite
positive Borel measure $\nu$ by%
\[
\mathcal{M}_{\mu}\nu\left(  x\right)  \equiv\sup_{B\in\mathcal{B}^{n}%
:\ c_{B}=x}\frac{\left\vert B\right\vert _{\nu}}{\left\vert B\right\vert
_{\mu}},
\]
where $\mathcal{B}^{n}$ is the collection of all balls $B$ in $\mathbb{R}^{n}%
$, and $c_{B}$ is the center of $B$.

\item For $0\leq\alpha<n$, we define $\mathcal{M}^{\alpha}$ on a finite
positive Borel measure $\nu$ by%
\[
\mathcal{M}^{\alpha}\nu\left(  x\right)  \equiv\sup_{B\in\mathcal{B}%
^{n}:\ c_{B}=x}\frac{1}{\left\vert B\right\vert ^{1-\frac{\alpha}{n}}}\int
_{B}d\nu.
\]

\item Given any $0<r<\infty$, we define $M_{\mu,r}$ acting on a function $f\in
L_{\operatorname{loc}}^{1}\left(  \mu\right)  $ by%
\[
M_{\mu,r}f\left(  x\right)  \equiv\sup_{B\in\mathcal{B}^{n}:\ c_{B}=x}\left(
\frac{1}{\left\vert B\right\vert _{\mu}}\int_{B}\left\vert f\right\vert
^{r}d\mu\right)  ^{\frac{1}{r}}.
\]

\end{enumerate}

\begin{lemma}
\label{flat}Suppose $\sigma$ and $\omega$ are locally finite positive Borel
measures on $\mathbb{R}^{n}$, and that $T^{\alpha}$ is a Calder\'{o}n-Zygmund
operator on $\mathbb{R}^{n}$ with kernel satisfying just (\ref{basic}). Then
for any $r>0$,%
\[
T_{\flat,\sigma}^{\alpha}f\left(  x\right)  \leq A_{\alpha,n,r}\left\{
M_{\omega,r}\left\vert T_{\sigma}^{\alpha}f\right\vert \left(  x\right)
+\mathfrak{N}_{T^{\alpha}}^{\operatorname*{weak}}\left(  \sigma,\omega\right)
\sqrt{\mathcal{M}_{\omega}\left(  \left\vert f\right\vert ^{2}\sigma\right)
\left(  x\right)  }+\mathcal{M}^{\alpha}\left(  \left\vert f\right\vert
\sigma\right)  \left(  x\right)  \right\}  ,\ \ \ \ \ x\in\mathbb{R}^{n}.
\]

\end{lemma}

\begin{proof}
Fix $z\in\mathbb{R}^{n}$ and $\varepsilon>0$. Write
\[
f=f\mathbf{1}_{B\left(  z,\varepsilon\right)  }+f\mathbf{1}_{\mathbb{R}%
^{n}\setminus B\left(  z,\varepsilon\right)  }=f_{1}+f_{2},
\]
so that
\[
T_{\varepsilon,\sigma}^{\alpha}f_{1}\left(  z\right)  =0\text{ and
}T_{\varepsilon,\sigma}^{\alpha}f\left(  z\right)  =T_{\varepsilon,\sigma
}^{\alpha}f_{2}\left(  z\right)  =T_{\sigma}^{\alpha}f_{2}\left(  z\right)  .
\]
We now claim that%
\[
\left\vert T_{\sigma}^{\alpha}f_{2}\left(  z\right)  -T_{\sigma}^{\alpha}%
f_{2}\left(  x\right)  \right\vert \leq A^{\prime}\mathcal{M}^{\alpha}\left(
\left\vert f\right\vert \sigma\right)  \left(  z\right)  ,\ \ \ \ \ \text{for
}\left\vert x-z\right\vert <\frac{\varepsilon}{c}.
\]
Indeed, we have%
\begin{align*}
&  \left\vert T_{\sigma}^{\alpha}f_{2}\left(  z\right)  -T_{\sigma}^{\alpha
}f_{2}\left(  x\right)  \right\vert \leq\int_{\left\vert y-z\right\vert
\geq\varepsilon}\left\vert K^{\alpha}\left(  x,y\right)  -K^{\alpha}\left(
z,y\right)  \right\vert \left\vert f\left(  y\right)  \right\vert
d\sigma\left(  y\right) \\
&  \leq\sum_{k=0}^{\infty}\int_{2^{k+1}\varepsilon\geq\left\vert
y-z\right\vert \geq2^{k}\varepsilon}\left\vert K^{\alpha}\left(  x,y\right)
-K^{\alpha}\left(  z,y\right)  \right\vert \left\vert f\left(  y\right)
\right\vert d\sigma\left(  y\right) \\
&  \lesssim\sum_{k=0}^{\infty}\eta\left(  \frac{1}{2^{k}c}\right)  \frac
{1}{\left\vert B\left(  z,\varepsilon2^{k}\right)  \right\vert ^{1-\frac
{\alpha}{n}}}\int_{B\left(  z,\varepsilon2^{k+1}\right)  }\left\vert f\left(
y\right)  \right\vert d\sigma\left(  y\right) \\
&  \leq c^{\prime}\sum_{k=0}^{\infty}\eta\left(  \frac{1}{2^{k}c}\right)
\mathcal{M}^{\alpha}\left(  \left\vert f\right\vert \sigma\right)  \left(
z\right)  \leq A^{\prime}\mathcal{M}^{\alpha}\left(  \left\vert f\right\vert
\sigma\right)  \left(  z\right)  .
\end{align*}
Thus we conclude that%
\[
\left\vert T_{\varepsilon,\sigma}^{\alpha}f\left(  z\right)  \right\vert
\leq\left\vert T_{\sigma}^{\alpha}f\left(  x\right)  \right\vert +\left\vert
T_{\sigma}^{\alpha}f_{1}\left(  x\right)  \right\vert +A^{\prime}%
\mathcal{M}^{\alpha}\left(  \left\vert f\right\vert \sigma\right)  \left(
z\right)  ,\ \ \ \ \ \text{for }x\in B\left(  z,\frac{\varepsilon}{c}\right)
.
\]

Now we note that%
\begin{align*}
&  \left\vert \left\{  x\in B\left(  z,\frac{\varepsilon}{c}\right)
:\left\vert T_{\sigma}^{\alpha}f\left(  x\right)  \right\vert >\lambda
\right\}  \right\vert _{\omega}\leq\lambda^{-r}\int_{B\left(  z,\frac
{\varepsilon}{c}\right)  }\left\vert T_{\sigma}^{\alpha}f\left(  x\right)
\right\vert ^{r}d\omega\left(  x\right) \\
&  \leq\lambda^{-r}\left\vert B\left(  z,\frac{\varepsilon}{c}\right)
\right\vert _{\omega}M_{\omega}\left\vert T_{\sigma}^{\alpha}f\right\vert
^{r}\left(  z\right)  \leq\frac{1}{4}\left\vert B\left(  z,\frac{\varepsilon
}{c}\right)  \right\vert _{\omega}\ ,
\end{align*}
provided%
\[
\lambda\geq4^{\frac{1}{r}}\left(  M_{\omega}\left\vert T_{\sigma}^{\alpha
}f\right\vert ^{r}\right)  \left(  z\right)  ^{\frac{1}{r}}.
\]
Moreover we have%
\begin{align*}
&  \left\vert \left\{  x\in B\left(  z,\frac{\varepsilon}{c}\right)
:\left\vert T_{\sigma}^{\alpha}f_{1}\left(  x\right)  \right\vert
>\lambda\right\}  \right\vert _{\omega}\leq\mathfrak{N}_{T^{\alpha}%
}^{\operatorname*{weak}}\left(  \sigma,\omega\right)  ^{2}\frac{1}{\lambda
^{2}}\int\left\vert f_{1}\right\vert ^{2}d\sigma\\
&  =\mathfrak{N}_{T^{\alpha}}^{\operatorname*{weak}}\left(  \sigma
,\omega\right)  ^{2}\frac{1}{\lambda^{2}}\int_{B\left(  z,\frac{\varepsilon
}{c}\right)  }\left\vert f\right\vert ^{2}d\sigma\leq\mathfrak{N}_{T^{\alpha}%
}^{\operatorname*{weak}}\left(  \sigma,\omega\right)  ^{2}\frac{1}{\lambda
^{2}}\left\vert B\left(  z,\frac{\varepsilon}{c}\right)  \right\vert _{\omega
}\mathcal{M}_{\omega}\left(  \left\vert f\right\vert ^{2}\sigma\right)
\left(  z\right)  \leq\frac{1}{4}\left\vert B\left(  z,\frac{\varepsilon}%
{c}\right)  \right\vert _{\omega}\ ,
\end{align*}
provided%
\[
\lambda\geq2\mathfrak{N}_{T^{\alpha}}^{\operatorname*{weak}}\left(
\sigma,\omega\right)  \sqrt{\mathcal{M}_{\omega}\left(  \left\vert
f\right\vert ^{2}\sigma\right)  \left(  z\right)  }.
\]

Altogether, if
\[
\lambda=\max\left\{  4^{\frac{1}{r}}\left(  M_{\omega}\left\vert T_{\sigma
}^{\alpha}f\right\vert ^{r}\right)  \left(  z\right)  ^{\frac{1}{r}%
},2\mathfrak{N}_{T^{\alpha}}^{\operatorname*{weak}}\left(  \sigma
,\omega\right)  \sqrt{\mathcal{M}_{\omega}\left(  \left\vert f\right\vert
^{2}\sigma\right)  \left(  z\right)  }\right\}  ,
\]
then there exists $x\in B\left(  z,\frac{\varepsilon}{c}\right)  $ such that
both $\left\vert T_{\sigma}^{\alpha}f\left(  x\right)  \right\vert $ and
$\left\vert T_{\sigma}^{\alpha}f_{1}\left(  x\right)  \right\vert $ are at
most $\lambda$, and it follows that%
\begin{align*}
\left\vert T_{\varepsilon,\sigma}^{\alpha}f\left(  z\right)  \right\vert  &
\leq\left\vert T_{\sigma}^{\alpha}f\left(  x\right)  \right\vert +\left\vert
T_{\sigma}^{\alpha}f_{1}\left(  x\right)  \right\vert +A^{\prime}%
\mathcal{M}^{\alpha}\left(  f\sigma\right)  \left(  z\right)  \leq
2\lambda+A^{\prime}\mathcal{M}^{\alpha}\left(  \left\vert f\right\vert
\sigma\right)  \left(  z\right) \\
&  \leq2\cdot4^{\frac{1}{r}}\left(  M_{\omega}\left\vert T_{\sigma}^{\alpha
}f\right\vert ^{r}\right)  ^{\frac{1}{r}}\left(  z\right)  +2\mathfrak{N}%
_{T^{\alpha}}^{\operatorname*{weak}}\left(  \sigma,\omega\right)
A\sqrt{\mathcal{M}_{\omega}\left(  \left\vert f\right\vert ^{2}\sigma\right)
\left(  z\right)  }+A^{\prime}\mathcal{M}^{\alpha}\left(  f\sigma\right)
\left(  z\right)  .
\end{align*}
Taking the supremum in $\varepsilon>0$ now completes the proof of Lemma
\ref{flat}.
\end{proof}

Now we can quickly prove Theorem \ref{weak type}\ using the well known
properties that $M_{\mu,r}$ is bounded on both $L^{p}\left(  \mu\right)  $ and
$L^{p,\infty}\left(  \mu\right)  $ for $1\leq r<p\leq\infty$ (use the
Besicovitch covering lemma for strong type, and then interpolation for weak
type), and in addition, that $\mathcal{M}_{\mu}$ is `weak type on measures
$\nu$', i.e.%
\[
\left\vert \left\{  x\in\mathbb{R}^{n}:\mathcal{M}_{\mu}\nu\left(  x\right)
>\lambda\right\}  \right\vert _{\mu}\leq\frac{1}{\lambda}\int_{\mathbb{R}^{n}%
}d\nu.
\]
See e.g. \cite[8.17 on page 44]{Ste2}.

\begin{proof}
[Proof of Theorem \ref{weak type}]Recall that we are assuming only the
conditions in (\ref{basic}) on the kernel of $T$. The first inequality in
(\ref{weak flat}) follows from \cite[Theorem 1.8 (3)]{LaSaUr1}. To prove the
second inequality in (\ref{weak flat}), we use Lemma \ref{flat} to write,%
\begin{align*}
\left\vert \left\{  T_{\flat,\sigma}^{\alpha}f>\lambda\right\}  \right\vert
_{\omega}  &  \lesssim\left\vert \left\{  \left(  M_{\omega}\left\vert
T_{\sigma}^{\alpha}f\right\vert \left(  x\right)  ^{r}\right)  ^{\frac{1}{r}%
}>\frac{1}{3A_{\alpha,n,r}}\lambda\right\}  \right\vert _{\omega}\\
&  +\left\vert \left\{  \sqrt{\mathcal{M}_{\omega}\left(  \left\vert
f\right\vert ^{2}\sigma\right)  \left(  z\right)  }>\frac{1}{3\mathfrak{N}%
_{T^{\alpha}}^{\operatorname*{weak}}\left(  \sigma,\omega\right)
A_{\alpha,n,r}}\lambda\right\}  \right\vert _{\omega}\\
&  +\left\vert \left\{  \mathcal{M}^{\alpha}\left(  f\sigma\right)  \left(
x\right)  >\frac{1}{3A_{\alpha,n,r}}\lambda\right\}  \right\vert _{\omega}\\
&  \equiv I\left(  \lambda\right)  +II\left(  \lambda\right)  +III\left(
\lambda\right)  ,
\end{align*}
where
\begin{align*}
&  \lambda\sqrt{I\left(  \lambda\right)  }=\lambda\sqrt{\left\vert \left\{
M_{\omega,r}\left(  T_{\sigma}^{\alpha}f\right)  \left(  x\right)  >\frac
{1}{3A_{\alpha,n,r}}\lambda\right\}  \right\vert _{\omega}}\leq\left\Vert
M_{\omega,r}\left(  T_{\sigma}^{\alpha}f\right)  \right\Vert _{L^{2,\infty
}\left(  \omega\right)  }\\
&  \lesssim\left\Vert T_{\sigma}^{\alpha}f\right\Vert _{L^{2,\infty}\left(
\omega\right)  }\leq\mathfrak{N}_{T^{\alpha}}^{\operatorname*{weak}}\left(
\sigma,\omega\right)  \left\Vert f\right\Vert _{L^{2}\left(  \sigma\right)  }.
\end{align*}
Using the Besicovitch covering lemma, we can write,
\[
\left\{  z:\sqrt{\mathcal{M}_{\omega}\left(  \left\vert f\right\vert
^{2}\sigma\right)  \left(  z\right)  }>\frac{1}{3\mathfrak{N}_{T^{\alpha}%
}^{\operatorname*{weak}}\left(  \sigma,\omega\right)  A_{\alpha,n,r}}%
\lambda\right\}  =\bigcup_{i=1}^{\infty}B_{i}%
\]
where the collection of balls $\left\{  B_{i}\right\}  _{i=1}^{\infty}$ has
bounded overlap $\beta$, and each $B_{i}$ satisfies%
\[
\sqrt{\frac{1}{\left\vert B_{i}\right\vert _{\omega}}\int_{B_{i}}\left\vert
f\right\vert ^{2}d\sigma}>\frac{1}{3\mathfrak{N}_{T^{\alpha}}%
^{\operatorname*{weak}}\left(  \sigma,\omega\right)  A_{\alpha,n,r}}\lambda
\]
Then
\begin{align*}
&  \lambda\sqrt{II\left(  \lambda\right)  }=\lambda\sqrt{\left\vert \left\{
z:\mathcal{M}_{\omega}\left(  \left\vert f\right\vert ^{2}\sigma\right)
\left(  z\right)  >\frac{1}{9\mathfrak{N}_{T^{\alpha}}^{\operatorname*{weak}%
}\left(  \sigma,\omega\right)  ^{2}A_{\alpha,n,r}^{2}}\lambda^{2}\right\}
\right\vert _{\omega}}\lesssim\sqrt{\sum_{i\in\mathbb{N}}\lambda^{2}\left\vert
B_{i}\right\vert _{\omega}}\\
&  \leq\sqrt{\sum_{i\in\mathbb{N}}9A_{\alpha,n,r}^{2}\mathfrak{N}_{T^{\alpha}%
}^{\operatorname*{weak}}\left(  \sigma,\omega\right)  ^{2}\left(  \frac
{1}{\left\vert B_{i}\right\vert _{\omega}}\int_{B_{i}}\left\vert f\right\vert
^{2}d\sigma\right)  \left\vert B_{i}\right\vert _{\omega}}\\
&  \leq3A_{\alpha,n,r}\mathfrak{N}_{T^{\alpha}}^{\operatorname*{weak}}\left(
\sigma,\omega\right)  \sqrt{\beta\int_{\mathbb{R}^{n}}\left\vert f\right\vert
^{2}d\sigma}=3A_{\alpha,n,r}\sqrt{\beta}\mathfrak{N}_{T^{\alpha}%
}^{\operatorname*{weak}}\left(  \sigma,\omega\right)  \left\Vert f\right\Vert
_{L^{2}\left(  \sigma\right)  }.
\end{align*}
Finally we use the Besicovitch covering lemma once more to write,
\[
\left\{  \mathcal{M}^{\alpha}\left(  \left\vert f\right\vert \sigma\right)
\left(  x\right)  >\frac{1}{3}\lambda\right\}  =\bigcup_{i=1}^{\infty}B_{i}%
\]
where the collection of balls $\left\{  B_{i}\right\}  _{i=1}^{\infty}$ has
bounded overlap $\beta$ and each $B_{i}$ satisfies%
\[
\frac{1}{\left\vert B_{i}\right\vert ^{1-\frac{\alpha}{n}}}\int_{B_{i}%
}\left\vert f\right\vert d\sigma>\frac{1}{3}\lambda.
\]
Then
\begin{align*}
&  \lambda\sqrt{III\left(  \lambda\right)  }=\lambda\sqrt{\left\vert \left\{
\mathcal{M}^{\alpha}\left(  \left\vert f\right\vert \sigma\right)  \left(
x\right)  >\frac{1}{3}\lambda\right\}  \right\vert _{\omega}}=\sqrt{\sum
_{i\in\mathbb{N}}\lambda^{2}\left\vert B_{i}\right\vert _{\omega}}\\
&  \leq\sqrt{\sum_{i\in\mathbb{N}}9\left(  \frac{1}{\left\vert B_{i}%
\right\vert ^{1-\frac{\alpha}{n}}}\int_{B_{i}}\left\vert f\right\vert
d\sigma\right)  ^{2}\left\vert B_{i}\right\vert _{\omega}}\leq3\sqrt
{\sum_{i\in\mathbb{N}}\frac{\left\vert B_{i}\right\vert _{\omega}\left\vert
B_{i}\right\vert _{\sigma}}{\left\vert B_{i}\right\vert ^{2\left(
1-\frac{\alpha}{n}\right)  }}\int_{B_{i}}\left\vert f\right\vert ^{2}d\sigma
}\leq3\sqrt{A_{2}^{\alpha}\left(  \sigma,\omega\right)  }\left\Vert
f\right\Vert _{L^{2}\left(  \sigma\right)  }.
\end{align*}

Altogether we have,%
\begin{align*}
&  \mathfrak{N}_{T_{\flat}^{\alpha}}^{\operatorname*{weak}}\left(
\sigma,\omega\right)  =\sup_{f\in L^{2}\left(  \sigma\right)  }\frac
{\left\Vert T_{\flat,\sigma}^{\alpha}f\right\Vert _{L^{2,\infty}\left(
\omega\right)  }}{\left\Vert f\right\Vert _{L^{2}\left(  \sigma\right)  }%
}=\sup_{f\in L^{2}\left(  \sigma\right)  }\frac{\sup_{\lambda>0}\lambda
\sqrt{\left\vert \left\{  T_{\flat,\sigma}^{\alpha}f>\lambda\right\}
\right\vert _{\omega}}}{\left\Vert f\right\Vert _{L^{2}\left(  \sigma\right)
}}\\
&  \lesssim\sup_{f\in L^{2}\left(  \sigma\right)  }\frac{\sup_{\lambda
>0}\lambda\sqrt{I\left(  \lambda\right)  }}{\left\Vert f\right\Vert
_{L^{2}\left(  \sigma\right)  }}+\sup_{f\in L^{2}\left(  \sigma\right)  }%
\frac{\sup_{\lambda>0}\lambda\sqrt{II\left(  \lambda\right)  }}{\left\Vert
f\right\Vert _{L^{2}\left(  \sigma\right)  }}+\sup_{f\in L^{2}\left(
\sigma\right)  }\frac{\sup_{\lambda>0}\lambda\sqrt{III\left(  \lambda\right)
}}{\left\Vert f\right\Vert _{L^{2}\left(  \sigma\right)  }}\\
&  \lesssim\mathfrak{N}_{T^{\alpha}}^{\operatorname*{weak}}\left(
\sigma,\omega\right)  +\sqrt{A_{2}^{\alpha}\left(  \sigma,\omega\right)  },
\end{align*}
which completes the proof of the second inequality in (\ref{weak flat}).
\end{proof}


\begin{thebibliography}{999999999}                                                                                        %


\bibitem[AlSaUr]{AlSaUr}\textsc{M. Alexis, E. T. Sawyer and I. Uriarte-Tuero,}
\textit{Tops of dyadic grids and }$T1$\textit{ theorems},
\texttt{arXiv.2201.02897}.

\bibitem[Gra]{Gra}\textsc{Grafokas, L.}, \textit{Classical Fourier Analysis},
Graduate Texts in Mathematics, 2nd ed. \textbf{249} (2008), Springer.

\bibitem[Hyt2]{Hyt2}\textsc{Hyt\"{o}nen, Tuomas, }\textit{The two weight
inequality for the Hilbert transform with general measures, }%
\texttt{arXiv:1312.0843v2.}

\bibitem[Lac]{Lac}\textsc{Lacey, Michael T.,}\textit{\ Two weight inequality
for the Hilbert transform: A real variable characterization, II}, Duke Math.
J. Volume \textbf{163}, Number 15 (2014), 2821-2840.

\bibitem[LaSaShUr3]{LaSaShUr3}\textsc{Lacey, Michael T., Sawyer, Eric T.,
Shen, Chun-Yen, Uriarte-Tuero, Ignacio,} \textit{Two weight inequality for the
Hilbert transform: A real variable characterization I}, Duke Math. J, Volume
\textbf{163}, Number 15 (2014), 2795-2820.

\bibitem[LaSaUr1]{LaSaUr1}\textsc{Lacey, Michael T., Sawyer, Eric T.,
Uriarte-Tuero, Ignacio,} \textit{A characterization of two weight norm
inequalities for maximal singular integrals with one doubling measure,}
Analysis \& PDE, Vol. \textbf{5} (2012), No. 1, 1-60.

\bibitem[LaWi]{LaWi}\textsc{Lacey, Michael T., Wick, Brett D.,} \textit{Two
weight inequalities for Riesz transforms: uniformly full dimension weights},
\texttt{arXiv:1312.6163v3}.

\bibitem[NTV]{NTV}\textsc{F. Nazarov, S. Treil and A. Volberg}, \textit{Cauchy
Integral and Calder\'{o}n Zygmund operators on nonhomogeneous spaces},
International Mathematics Research Notices, Volume \textbf{1997}, Issue 15,
1997, Pages 703--726.\textit{\texttt{v3}}.

\bibitem[NTV3]{NTV3}\textsc{Nazarov, F., Treil, S., and Volberg, A.,}
\textit{Accretive system }$Tb$\textit{-theorems on nonhomogeneous spaces},
Duke Math. J. \textbf{113} (2) (2002), 259--312.

\bibitem[NTV4]{NTV4}\textsc{F. Nazarov, S. Treil and A. Volberg,} \textit{Two
weight estimate for the Hilbert transform and corona decomposition for
non-doubling measures}, preprint (2004) \texttt{arXiv:1003.1596.}

\bibitem[RaSaWi]{RaSaWi}\textsc{Robert Rahm, Eric T. Sawyer and Brett D.
Wick,} \textit{Weighted Alpert wavelets}, \texttt{arXiv:1808.01223v2}.

\bibitem[Saw6]{Saw6}\textsc{E. Sawyer,} \textit{A }$T1$\textit{\ theorem for
general Calder\'{o}n-Zygmund operators with comparable doubling weights and
optimal cancellation conditions}, \texttt{arXiv:1906.05602v10}, Journal d'Analyse.

\bibitem[SaShUr7]{SaShUr7}\textsc{Sawyer, Eric T., Shen, Chun-Yen,
Uriarte-Tuero, Ignacio,} \textit{A} \textit{two weight theorem for }$\alpha
$\textit{-fractional singular integrals with an energy side condition},
Revista Mat. Iberoam. \textbf{32} (2016), no. 1, 79-174.

\bibitem[SaShUr9]{SaShUr9}\textsc{Sawyer, Eric T., Shen, Chun-Yen,
Uriarte-Tuero, Ignacio,} A \textit{two weight fractional singular integral
theorem with side conditions, energy and }$k$\textit{-energy dispersed,}
Harmonic Analysis, Partial Differential Equations, Complex Analysis, Banach
Spaces, and Operator Theory (Volume 2) (Celebrating Cora Sadosky's life),
Springer 2017 (see also \texttt{arXiv:1603.04332v2}).

\bibitem[SaShUr10]{SaShUr10}\textsc{Sawyer, Eric T., Shen, Chun-Yen,
Uriarte-Tuero, Ignacio,} \textit{A good-}$\lambda$\textit{\ lemma, two weight
}$T1$\textit{\ theorems without weak boundedness, and a two weight accretive
global }$Tb$\textit{\ theorem,} Harmonic Analysis, Partial Differential
Equations and Applications (In Honor of Richard L. Wheeden), Birkh\"{a}user
2017 (see also \texttt{arXiv:1609.08125v2}).

\bibitem[SaShUr12]{SaShUr12}\textsc{Sawyer, Eric T., Shen, Chun-Yen,
Uriarte-Tuero, Ignacio,} \textit{A two weight local }$Tb$\textit{\ theorem for
the Hilbert transform}, to appear in Revista Mat. Iberoam.,
\texttt{arXiv.1709.09595v14}.

\bibitem[SaUr]{SaUr}\textsc{E. Sawyer and I. Uriarte-Tuero,} \textit{Control
of the bilinear indicator cube testing property, }\texttt{arXiv:1910.09869}.

\bibitem[Ste]{Ste}\textsc{Stein, Elias M.}, \textit{Singular Integrals and
Differentiability Properties of Functions}, Princeton University Press,
Princeton, N. J., 1970.

\bibitem[Ste2]{Ste2}\textsc{E. M. Stein,} \textit{Harmonic Analysis:
real-variable methods, orthogonality, and oscillatory integrals}%
,\textit{\ }Princeton University Press, Princeton, N. J., 1993.

\bibitem[StWe]{StWe}\textsc{E. M. Stein and G. Weiss,} \textit{Fourier
analysis on Euclidean spaces}, Princeton University Press, Princeton, N. J., 1971.

\bibitem[Vol]{Vol}\textsc{A. Volberg,} \textit{Calder\'{o}n-Zygmund capacities
and operators on nonhomogeneous spaces,} CBMS Regional Conference Series in
Mathematics (2003), MR\{2019058 (2005c:42015)\}.
\end{thebibliography}
\end{document}